\documentclass[twoside]{amsart}
\usepackage{latexsym, amsmath,amssymb,  bbold}

\usepackage{subcaption}

	\usepackage{xcolor}
	
	\usepackage{graphicx}

\renewcommand{\epsilon}{\varepsilon}
\newcommand{\newsection}[1] 
{\subsection{#1}\setcounter{theorem}{0} \setcounter{equation}{0} 
\par\noindent}

\newtheorem{theorem}{Theorem}

\newtheorem{lemma}[theorem]{Lemma}
\newtheorem{corr}[theorem]{Corollary}

\newtheorem{proposition}[theorem]{Proposition}
\newtheorem{deff}[theorem]{Definition}
\newtheorem{remark}[theorem]{Remark}
\newcommand{\bth}{\begin{theorem}}
\newcommand{\ble}{\begin{lemma}}
\newcommand{\bcor}{\begin{corr}}

\newcommand{\bdeff}{\begin{deff}}

\newcommand{\bprop}{\begin{proposition}}
\newcommand{\ele}{\end{lemma}}
\newcommand{\ecor}{\end{corr}}
\newcommand{\edeff}{\end{deff}}

\newcommand{\eprop}{\end{proposition}}

\newcommand{\la}{\lambda}

\newcommand{\e}{\varepsilon}

\renewcommand{\Pi}{\varPi}

\renewcommand{\epsilon}{\varepsilon}

\newcommand{\R}{{\mathbb R}}

\newcommand{\1}{\mathbb{1}}

\newcommand{\tis}{\tilde s}
\newcommand{\tit}{\tilde t}
\newcommand{\tix}{\tilde x}
\newcommand{\tiy}{\tilde y}

\newcommand{\diag}{\Upsilon^{\text{diag}}}

\newcommand{\tdiag}{\tilde\Upsilon^{\text{diag}}}

\newcommand{\far}{\Upsilon^{\text{far}}}

\newcommand{\kt}{L^2_tL^{q_e}_x}

\newcommand{\qe}{q_e}
\newcommand{\qc}{q_c}

\newcommand{\Atn}{A^{\theta_0}_\nu}

\newcommand{\tAtn}{\tilde A^{\theta_0}_{\nu}}

\newcommand{\Atnt}{A^{\theta_0}_{\tilde\nu}}

\newcommand{\tAtnt}{\tilde A^{\theta_0}_{\tilde \nu}}

\newcommand{\xid}{\Xi_{\theta_0}}

\newcommand{\speed}{\kappa_\ell^{c_0\theta}}

\begin{document}

\title[Strichartz estimates  on negatively curved manifolds]
{Strichartz estimates for the Schr\"odinger equation on negatively curved compact manifolds}
%{Remarks on Improved Spectral Cluster and Weyl Remainder Estimates on  Some Product Manifolds}
%\thanks{The authors were supported in part by the NSF}

\thanks{The second author was supported in part by an AMS-Simons travel grant. The third author was supported in part by the NSF (DMS-1665373). }


\author{Matthew D. Blair}
\address[M.D.B.]{Department of Mathematics, 
University of New Mexico, Albuquerque, NM 87131}
\email{blair@math.unm.edu} 
\author{Xiaoqi Huang}
\address[X.H.]{Department of Mathematics, University of Maryland, College Park. MD 20742}
\email{xhuang49@umd.edu}
\author{Christopher D. Sogge}
\address[C.D.S.]{Department of Mathematics,  Johns Hopkins University,
Baltimore, MD 21218}
\email{sogge@jhu.edu}

\begin{abstract}
We obtain improved Strichartz estimates for solutions of the Schr\"odinger equation on negatively curved compact manifolds which improve the
classical universal results results of Burq, G\'erard and Tzvetkov~\cite{bgtmanifold} in this geometry.  In the case where the spatial manifold is a 
hyperbolic surface we are able to obtain no-loss $L^{q_c}_{t,x}$-estimates on intervals of length $\log \la\cdot \la^{-1} $ for initial data whose frequencies
are comparable to $\la$, which, given the role of the Ehrenfest time, is the natural analog of the universal results
in \cite{bgtmanifold}.     We are also obtain improved endpoint Strichartz estimates for manifolds of nonpositive
curvature, which cannot
hold for spheres.
\end{abstract}

%\bigskip\bigskip
\maketitle

\newsection{Introduction}

It has been almost two decades since Burq, G\'erard
and Tzvetkov~\cite{bgtmanifold} obtained their now
classical
universal Strichartz estimates for the 
Schr\"odinger equation on compact manifolds.
Besides the notable exception of near lossless
estimates on general tori by 
Bourgain and Demeter~\cite{BoDe}, 
and more recent related work in this setting by Deng, Germain and Guth~\cite{DGG} and
Deng, Germain, Guth and Meyerson~\cite{DGGM},
to the best of
our knowledge, there have not been significant 
improvements of the results in \cite{bgtmanifold}, in other geometries.

The purpose of this paper is to obtain improvement
of the universal bounds in \cite{bgtmanifold} under
the assumption of negative curvature, as well as, more generally,
nonpositive curvature.

Let us now recall the universal estimates
of Burq, G\'erard and Tzvetkov~\cite{bgtmanifold}.
If $(M^d,g)$ is a compact Riemannian manifold of
dimension $d\ge2$, then the main estimate
in \cite{bgtmanifold} is that if $\Delta_g$ is
the associated Laplace-Beltrami operator and
\begin{equation}\label{00.1}
u(x,t)=\bigl(e^{-it\Delta_g}f\bigr)(x)
\end{equation}
is the solution of the Schr\"odinger equation
on $M^d\times \R$,
\begin{equation}\label{00.2}
i \partial_tu(x,t)=\Delta_gu(x,t), \quad u(x,0)=f(x),
\end{equation}
then one has the mixed-norm Strichartz estimates
\begin{equation}\label{00.3}
\|u\|_{L^p_tL^q_x(M^d\times [0,1])}
\lesssim \|f\|_{H^{1/p}(M^d)}
\end{equation}
for all {\em admissible} pairs $(p,q)$.  By the latter
we mean, as in Keel and Tao~\cite{KT},
\begin{equation}\label{00.4}
d(\tfrac12-\tfrac1q)=\tfrac2p\, \, \,
\text{and } \, \, 2< q\le \tfrac{2d}{d-2} \, \,
\text{if } \, d\ge 3, \, \, \, 
\text{or } \, 2< q<\infty \, \, 
\text{if } \, \, d=2.
\end{equation}
Also, in \eqref{00.3} $H^\mu$ denotes the
standard Sobolev space
\begin{equation}\label{00.5}
\|f\|_{H^\mu(M^d)}=
\bigl\| \, (I+P)^\mu f\, \bigr\|_{L^2(M^d)}, 
\quad \text{with } \, \, P=\sqrt{-\Delta_g},
\end{equation}
and ``$\lesssim$'' in \eqref{00.3} and, in what
follows, denotes an inequality with an implicit,
but unstated, constant $C$ which can change at each occurrence.

Note that if $e_\la$ is an eigenfunction of
$P$ with eigenvalue $\la$, i.e.,
\begin{equation}\label{00.6}
-\Delta_g e_\la = \la^2 e_\la,
\end{equation}
then
\begin{equation}\label{00.7}
u(x,t)=e^{it\la^2}e_\la(x)
\end{equation}
solves \eqref{00.2} with initial data $f=e_\la$.
From this one immediately sees that, unlike for
the Euclidean case originally treated by
Strichartz~\cite{Strichartz77}, one can never
obtain any sort of global analog of
\eqref{00.3} where $[0,1]$ is replaced by
$\R$.  On the other hand, the proof of \eqref{00.3}
in \cite{bgtmanifold} shows that one can replace
$[0,1]$ by a larger interval $I$ at the expense
of an additional factor $|I|^{1/p}$ in the 
implicit constant in the right side of \eqref{00.3}.  Also, in some cases,
the special solutions \eqref{00.7} involving
eigenfunctions saturate \eqref{00.3}.  Specifically,
for the endpoint Strichartz estimates where
$p=2$ and $q=\tfrac{2d}{d-2}$ with $d\ge3$ the
solutions where $e_\la=Z_\la$ are zonal eigenfunctions
on $S^d$ with eigenvalue $\la=({k(k+\tfrac{n-1}2)})^{1/2}$, $k=1,2,\dots$, which saturate \eqref{00.3} since
\eqref{00.3} as
\begin{equation}\label{00.8}
\|Z_\la\|_{L^{\frac{2d}{d-2}}(S^d)}/\|Z_\la\|_{L^2(S^d)}
\approx \la^{1/2}
\end{equation}
(see, e.g., \cite{sogge86}).  We shall have more to
say about solutions arising from eigenfunction  in what follows.

To align with the numerology in related earlier results
involving eigenfunction and spectral projection estimates, as well as
parabolic Fourier restriction problems, in what follows,
we shall always take $d=n-1$.  Thus, we are interested
in estimates of solutions of Schr\"odinger's equation
\eqref{00.2} on the $n$-dimensional space
$M^{n-1}\times [0,1]$.  As we mentioned before,
we are focusing here on improvements of the 
universal bounds \eqref{00.3} of 
\cite{bgtmanifold} when $M^{n-1}$ has nonpositive
curvature.  We shall take $d=n-1\ge2$, since
the case where $d=1$ boils down to
the spatial manifold being the circle, $S^1$, and optimal
results in this case were obtained by
Bourgain~\cite{BourgainStrichartz1}.  In what follows
(just as in \cite{BourgainStrichartz1} and \cite{BoDe}) we shall mainly focus
on the unique admissible pair $(p,q)$ in \eqref{00.4}
where $p=q$, i.e.,
\begin{equation}\label{00.9}
q=q_c=\tfrac{2(n+1)}{n-1}.
\end{equation}

One of our main results is that in this case
we have logarithmic improvements of the
universal bounds in \cite{bgtmanifold} under our curvature assumptions.

\begin{theorem}\label{nonposthm}  Let $M^{n-1}$ be a
$d=n-1\ge 2$ dimensional compact manifold all of whose
sectional curvatures are nonpositive.  Then
\begin{equation}\label{00.10}
\|u\|_{L^{q_c}(M^{n-1}\times [0,1] )}\lesssim \bigl\|(I+P)^{1/q_c} \, (\log(2I+P))^{-\frac{n-1}{(n+1)^2}}f \bigr\|_{L^2(M^{n-1})}.
\end{equation}
\end{theorem}

To prove this estimate we shall employ a similar strategy
to the one used in \cite{bgtmanifold}, which we now
recall.  We first note that, by Littlewood-Paley theory,
we may reduce matters to proving certain dyadic estimates.

To this end, fix a Littlewood-Paley bump function
$\beta$ satisfying
\begin{equation}\label{00.11}
\beta \in C^\infty_0((1/2,2)) \quad
\text{and } \, \, 1=\sum_{k=-\infty}^\infty 
\beta(2^{-k}s), \, \, s>0.
\end{equation}
Then, if we set $\beta_0(s)=1-\sum_{k=1}^\infty
\beta(2^{-k}s)\in C^\infty_0(\R_+)$ and
$\beta_k(s)=\beta(2^{-k}s)$, $k=1,2,\dots$, we
have (see e.g., \cite{SFIO2})
\begin{equation}\label{00.12}
\|h\|_{L^q(M^{n-1})}\approx
\bigl\| \, (\, \sum_{k=0}^\infty |\beta_k(P)h|^2\, 
)^{1/2} \, \bigr\|_{L^q(M^{n-1})}, \, \, \,
1<q<\infty.
\end{equation} 
Trivially, $\|\beta_0(P)e^{-it\Delta_g}\|_{L^2(M^{n-1})
\to L^q(M^{n-1}\times [0,1])}=O(1)$, and, similarly such results
where $k=0$ is replaced by a small fixed $k\in {\mathbb N}$ are
also standard.
So, as noted
in \cite{bgtmanifold}, one can use \eqref{00.12} and
Minkowski's inequality to see that the special case
of 
\eqref{00.3} where $p=q=q_c$ follows
from the uniform bounds
\begin{equation}\label{00.3'}\tag{1.3$'$}
\| e^{-it\Delta_g}\beta(P/\la)f\|_{L^{q_c}(M^{n-1}
\times [0,1])}\le C\la^{\frac1{q_c}}\, \|f\|_{L^2(M^{n-1})}, \quad \la \gg 1.
\end{equation}
Burq, G\'erard and Tzvetkov proved this estimate in
\cite{bgtmanifold} by showing that one always has the 
following uniform dyadic estimates over very small intervals:
\begin{equation}\label{00.3''}\tag{1.3$''$}
\|e^{-it\Delta_g} \beta(P/\la)f\|_{L^{q_c}(M^{n-1}
\times [0, \, \la^{-1}])}\le C\, \|f\|_{L^2(M^{n-1})}, \quad \la \gg 1.
\end{equation}
Indeed, \eqref{00.3''} immediately yields 
\eqref{00.3'}, since one can write $[0,1]$ as the
union of $\approx \la$ intervals of length $\la^{-1}$
and thus obtain \eqref{00.3'} by adding up the
uniform estimates on each of these subintervals
that \eqref{00.3''} affords.  As was noted
 in \cite{bgtmanifold}, one can also obtain
the universal Strichartz estimates of
Burq, G\'erard and Tzvetkov using 
local smoothing estimates of Staffilani and
Tataru~\cite{StTa}; however, it seems difficult
to obtain improvements like the ones 
in Theorem~\ref{nonposthm} using this approach.

The time scale here of $|t|\le \la^{-1}$ is natural
since the dyadic operators in \eqref{00.3''} behave somewhat 
like standard half-wave operators $e^{itcP}$ of speed $c=\la$,
although this is a somewhat cartoonish reduction.  
Being more specific, it is possible to construct parametrices
for the dyadic operators in such small time scales
that allow one to use the Keel-Tao \cite{KT} theorem
to deduce \eqref{00.3''}.  Similar arguments show that
the other cases in \eqref{00.3} also follow from
uniform dyadic estimates for this time scale.

It is a simple matter to see that on {\em any} manifold
the bounds in \eqref{00.3''} cannot be improved even
though the time intervals are very small.  For instance,
if $\beta(P/\la)(x,y)$ is the kernel of the 
Littlewood-Paley operators $\beta(P/\la)$ and
$f(x)=f_\la(x)=\beta(P/\la)(x,x_0)$ for any fixed
$x_0\in M^{n-1}$, then the ratio of the norms
in \eqref{00.3''} is comparable to one for
$\la \gg 1$.  As a result, in order to obtain
improvements such as those in \eqref{00.10},
one must use larger time intervals.  
%(See 
%\cite{BHS_Str} for more details.)

Since we are working on manifolds of nonpositive
curvature, due to the expected role of the 
Ehrenfest time in the analysis, it is natural to
consider time intervals of length
$\approx \log\la \cdot \la^{-1}$.  This is what
we shall do.  Specifically, we shall show that if
$M^{n-1}$ is as in Theorem~\ref{nonposthm} then
we have the uniform bounds
\begin{equation}\label{00.10'}\tag{1.10$'$}
\| e^{-it\Delta_g} \beta(P/\la) f\|_{L^{q_c}(M^{n-1}
\times [0, \, \log\la\cdot \la^{-1}])}\le C\, (\log\la)^{\frac2{(q_c)^2}} \, \|f\|_{L^2(M^{n-1})}, \quad \la\gg 1.
\end{equation}
Since the logarithmic gain of $\tfrac{n-1}{(n+1)^2}$ in \eqref{00.10} versus \eqref{00.3} is just
$\tfrac1{q_c}(1-\tfrac2{q_c})$, by the above counting
arguments, one obtains \eqref{00.10} from 
\eqref{00.10'} since $[0,1]$ can be covered
by $\approx \la/\log\la$ intervals of length
$\log\la\cdot \la^{-1}$.  Also,   the universal
bounds \eqref{00.3''}  imply the analog of 
this inequality with $2/(q_c)^2$ replaced by the larger
exponent $1/q_c$ (since $q_c>2$), which is another way of recognizing the improvement
of \eqref{00.10'} versus \eqref{00.3''}.

We shall also show that if one strengthens the hypothesis
in the above theorem by assuming that the manifolds
are of negative curvature than we can obtain
stronger results, including a natural analog
of the estimates \eqref{00.3''} for hyperbolic
surfaces:

\begin{theorem}\label{negthm}  Assume that
$d=n-1\ge 2$ and that
all of the sectional curvatures of $M^{n-1}$
are negative.  Then if $d=n-1\ge3$
\begin{equation}\label{00.13}
\|u\|_{L^{q_c}( M^{n-1}\times [0,1])}\lesssim \bigl\|(I+P)^{1/q_c} \, (\log(2I+P))^{-\frac1{(n+1)}}f \bigr\|_{L^2(M^{n-1})}.
\end{equation}
Moreover, if $d=n-1=2$, in which case $q_c=4$, we have
\begin{equation}\label{00.14}
\|  e^{-it\Delta_g} \beta(P/\la) f\|_{L^{4}(M^{2}
\times [0,\, \log\la\cdot \la^{-1}])}\le C\, \, \|f\|_{L^2(M^{2})}, \quad \la\gg 1,
\end{equation}
and
\begin{equation}\label{00.14'}\tag{1.14$'$} 
\|u\|_{L^{4}( M^{2}\times [0,1])}\lesssim \bigl\|(I+P)^{1/4} \, (\log(2I+P))^{-1/4}f \bigr\|_{L^2(M^{2})}.
\end{equation}
\end{theorem}

By the above discussion of course \eqref{00.14} yields
\eqref{00.14'}.  Moreover, we point out that 
\eqref{00.14} is the natural extension of the uniform
small-time scale estimates \eqref{00.3''} of 
Burq, G\'erard and Tzvetkov to time intervals
which are perhaps the largest one can hope to obtain
such estimates in the geometry we are focusing on  using available techniques, due to
the role of the Ehrenfest time.

As we shall see, the improvement in Theorem~\ref{negthm} compared to those in
Theorem~\ref{nonposthm} are due to the much stronger dispersive properties of the
kernel for the solution operators for the wave equation.  On the other hand, in proving 
Theorem~\ref{negthm}, we have to balance this with the exponential volume growth
 of 
manifolds of strictly negative curvature as we have in some earlier works.  We 
accomplish 
this using arguments involving microlocal pseudo-differential cutoffs.

By interpolating with the endpoint Strichartz estimates
of Burq, G\'erard and Tzvetkov \cite{bgtmanifold}, one can also
obtain logarithmic--power improvements for all of the
other pairs of exponents $(p,q)$ in \eqref{00.4}
besides the endpoint case where $p=2$ and $d=n-1\ge3$.
Although these techniques break down for the important
endpoint case, we are able to adapt arguments from
one of us \cite{sogge2015improved} to get the following
more modest improvements for this case.

\begin{theorem}\label{endthm}  Let $M^d$ be a
$d\ge 3$ dimensional compact manifold all of whose
sectional curvatures are nonpositive.  Then
\begin{equation}\label{00.15}
\|u\|_{L^2_tL^{\frac{2d}{d-2}}_x(  M^{d}\times [0,1])}\lesssim \bigl\|(I+P)^{1/2} \, (\log(\log(2I+P)))^{-1/2}f \bigr\|_{L^2(M^{d})}.
\end{equation}
\end{theorem}

Our mixed-norm notation differs a bit from some other works when we define
$$\|u\|_{L^p_tL^q_x(M^{d}\times [0,1])}=\bigr(\, \int_0^1 \, \|u(\, \cdot\, , t)\|_{L^q_x(M^{d})}^p \, dt\,\bigr)^{1/p}.$$
We choose to write $M^{d}\times [0,1]$ instead of $[0,1]\times M^{d}$ 
inside the norm in \eqref{00.15}, and ones to follow, since most of the crucial local analysis, as well
as the pseudodifferential cutoffs employed, involve the spatial variables.  We hope that our choice of notation does
not confuse the reader.

A very interesting, but perhaps difficult problem,
would be to show that, like in \eqref{00.14'}, one
could replace the $(\log(\log (2I+P)))^{1/p}$ gain
in \eqref{00.15} with a $(\log (2I+P))^{1/p}$ gain, with $p$ in \eqref{00.15}
being $2$ as opposed to $4$ in \eqref{00.14}.
This would provide a  potentially difficult generalization
of  an important special case of the $(\log\la)^{-1/2}$ eigenfunction gains
$$\|e_\la\|_{L^{\frac{2d}{d-2}}(M^d)}
\lesssim \la^{1/2}\, (\log\la)^{-1/2}\|e_\la\|_{L^2
(M^d)}$$
of Hassell and Tacy \cite{HassellTacy} for
manifolds of nonpositive curvature versus the
universal eigenfunction estimates of 
one of us \cite{sogge88} for $q>\tfrac{2(d+1)}{d-1}$.

As we shall show, for $d\ge3$ dimensional tori, we can
strengthen our endpoint estimates in \eqref{00.15}
by replacing, in this case, $(\log(\log(2+P)))^{-1/2}$
with $P^{\frac2{d+2}-\frac12+\e}$, $\forall \, \e>0$.  
This follows directly from using the $L^{q_c}_{t,x}$ toral estimates
of Bourgain and Demeter~\cite{BoDe} along with Sobolev estimates.
We have no doubt that
stronger estimates should hold; however,
we are not aware of any.  This seems worth
of further investigation.   The decoupling methods
of Bourgain and Demeter~\cite{BoDe} that work so
well for the case $p=q=q_c$ might not apply as well 
for the endpoint case $(p,q)=(2,\tfrac{2d}{d-2})$.
We have to prove our bounds
\eqref{00.15}
for general manifolds of nonpositive curvature
 in a somewhat 
circuitous way (leading to only log-log  power gains)
due to the fact that the related bilinear techniques that we utilize
 break down for this endpoint case.

The estimates in Theorems \ref{nonposthm}
and \ref{negthm} of course improve the universal
estimates of Burq, G\'erard and Tzvetkov~\cite{bgtmanifold} in  the geometry
that we are focusing on here, manifolds of nonpositive
curvature.  On the other hand, they are weaker
than the (near) optimal toral results of
Bourgain and Demeter~\cite{BoDe}, as well
as the non-endpoint Strichartz estimates for the sphere of 
Burq, G\'erard and Tzvetkov~\cite{bgtmanifold}.  The estimates in
\cite{BoDe} were obtained via decoupling using,
in part, that the types of microlocal cutoffs that we
shall employ commute well with  Schr\"odinger
propagators on tori, and, moreover, lend themselves there to 
analysis on  much larger time scales than
we are able to handle on general manifolds of
nonpositive curvature.  The improved estimates for spheres simply follow from specific arithmetic properties of the distinct eigenvalues of the Laplacian on $S^d$.

Even though we cannot obtain estimates that are
as strong as those for the sphere for the 
non-endpoint exponents in \eqref{00.4}, our
endpoint Strichartz estimates in Theorem~\ref{endthm}
are improvements of the ones for the sphere,
where, by \eqref{00.8}, there can be no improvement
of the $H^{1/2}(S^d)$ endpoint estimates of
Burq, G\'erard and Tzvetkov~\cite{bgtmanifold}.

This paper is organized as follows.  In the next section
we present the main arguments that allow us to prove
the above theorems.  The proofs require local
bilinear arguments from harmonic analysis and a
detailed analysis of the kernels that arise in both
the ``local'' and ``global'' arguments.  We carry out these in Sections 3 and 4, respectively.

The local harmonic analysis arguments that we use rely
on bilinear oscillatory integral estimates 
of Lee~\cite{LeeBilinear} and are variable coefficient
analogs of the arguments of Tao, Vargas and Vega~\cite{TaoVargasVega} that were used to study
parabolic restriction problems for the Fourier 
transform, which, of course is 
related to Strichartz estimates for 
Schr\"odinger's equation.  As we shall see, the kernels of the 
local operators oscillate most rapidly 
along curves of the form $s\to (x(\kappa s),
-(s-s_0))\in M^{n-1}\times \R$, where $x(s)\in M^{n-1}$
is a unit-speed geodesic.  We call such space-time
curves ``Schr\"odinger curves" of
varying speeds $\kappa$, which we shall be able to take 
to be comparable to one.
They are integral curves of the Hamilton vector
field $H_P$ associated with the Schr\"odinger operator
$P=D_t+\Delta_g$.  Such curves naturally arise in our analysis,
as well as in related past work (cf. \cite{AnanNon}, 
\cite{bgtmanifold} and \cite{GGH}).  
Perhaps a 
novelty here, though, is that, in order to 
apply the bilinear oscillatory integral estimates
of Lee~\cite{LeeBilinear}, it is very convenient
to work in what we call ``Schr\"odinger coordinates''
about one of these curves.  

These coordinates are
the analog of Fermi normal coordinates that naturally
arise in relativity theory and Riemannian
geometry (see, e.g., \cite{Tubes}, 
\cite{Fermi1} and \cite{Fermi2}).   In relativity theory, Fermi normal coordinates
are chosen so that, for an observer in a free fall (geodesic) path in an arbitrary spacetime,
the geometry will appear to be ``flat" up to higher order terms.
The Schr\"odinger coordinates that we shall employ have a similar property for quantum
``observers'' traveling along what we call Schr\"odinger curves.  The use of these ``Schr\"odinger coordinates" is key 
to be able to adapt the Euclidean harmonic analysis techniques of \cite{LeeBilinear} and \cite{TaoVargasVega}
to our variable coefficient setting.

In order to apply Lee's results we also need detailed
estimates for the kernels of the local operators
that arise.  Motivated by the earlier local quasimode
analysis of the last two authors \cite{HSSchro}, we
are able to construct local operators whose kernels
can essentially be calculated using  techniques
from the first and last authors \cite{SBLog}, while, at the same time, be of use
for studying the ``global operators'' that necessarily
arise in the proofs of the above theorems.  We need to 
compose the ``global'' operators with ``local'' ones
to apply the bilinear harmonic analysis techniques,
and, motivated by the earlier work in by the last two authors in \cite{HSSchro},
they can be constructed so that the difference between
the original global operators and the ones composed
with the local ones has small norm.  Besides the
relatively intricate application of harmonic analysis
techniques that we require, we also need to show, that when we
microlocalize the solution operators for 
Schr\"odinger's equation \eqref{00.2}, our geometric
assumptions imply that there are favorable bounds for
the resulting kernels.  Using the Fourier transform, this amounts to a classical
argument involving the Hadamard parametrix going back
to B\'erard~\cite{Berard}, with microlocal variants
in a more recent work of the first and third authors \cite{BSTop}, as well
as in that \cite{BHSsp} of all 
three of the authors.

The authors are  grateful to W. Minicozzi for
patiently answering numerous questions about
Fermi normal coordinates, as well as  for referring us
to the classical reference Manasse and
Misner~\cite{Fermi2}.

\newsection{Main arguments}

Let us start by proving Theorems \ref{nonposthm} and \ref{negthm} which concern the non-endpoint
Strichartz estimates.    Then at the end of this section we shall give the modifications needed to prove
the endpoint estimates in Theorem~\ref{endthm}.  For the proofs we shall require certain bilinear
estimates and pointwise estimates for kernels that arise in the arguments, which will be addressed
in the next two sections.

To start, let $\beta$ be the Littlewood-Paley bump function in \eqref{00.11}, and also fix
\begin{equation}\label{22.1}
\eta\in C^\infty_0((-1,1)) \quad
\text{with } \, \, \eta(t)=1, \, \, \, |t| \le 1/2.
\end{equation}
We then shall consider the dyadic time-localized dilated Schr\"odinger operators
\begin{equation}\label{22.2} 
S_\la =\eta(t/T) e^{-it\la^{-1}\Delta_g} \beta(P/\la),
\end{equation}
and claim that the estimates in Theorems \ref{nonposthm} and \ref{negthm}
are a consequence of the following.

\begin{proposition}\label{mainprop}  Let $M^d$, $d=n-1\ge2$ be a fixed compact manifold all of whose
sectional curvatures are nonpositive.  Then we can fix $c_0>0$ so that for large $\la \gg 1$ we have the
uniform bounds
\begin{equation}\label{22.3}
\| S_\la f\|_{L^{q_c}(M^{n-1}\times \R)} \le C
\la^{\frac1{q_c}} T^{\frac1{q_c}\cdot \frac2{q_c}} \, \|f\|_{L^2(M^{n-1})}, \quad
\text{if } \, \, T=c_0\log\la.
\end{equation}
Moreover, if all of the sectional curvatures of $M^{n-1}$ are negative
$c_0>0$ can be fixed so that for all $\la \gg 1$ we have
\begin{equation}\label{22.4}
\|S_\la f\|_{L^{q_c}(M^{n-1}\times \R)}\le C\la^{\frac1{q_c}}
T^{\frac{4-q_c}{2q_c}} \, \|f\|_{L^2(M^{n-1})}, \quad
\text{if } \, \, T=c_0\log\la.
\end{equation}
\end{proposition}

We claim that \eqref{22.3} and \eqref{22.4} imply Theorems \ref{nonposthm} and
\ref{negthm}, respectively.  For the former, we note that
just by changing scales \eqref{22.1} and \eqref{22.3} imply that for large
enough $\la$ we have the analog of \eqref{00.10'} where the interval
$[0,\log\la\cdot \la^{-1}]$ in the left is replaced by
$[0, \tfrac12c_0\log\la\cdot \la^{-1}]$, and this of course implies \eqref{00.10'}
at the expense of including an additional factor of $(c_0/2)^{-1/q_c}$ in the constant in 
the right if $c_0<2$.  As we indicated before, the estimate \eqref{00.10'} for
large $\la$ and Littlewood-Paley theory yield Theorem~\ref{nonposthm}, which verifies
our claim regarding \eqref{22.3}.  Repeating this argument, we see that \eqref{22.4}
implies that, for large enough $\la$, we have
\begin{equation}\label{22.4'}\tag{2.4$'$}
\| e^{-it\Delta_g}  \beta(P/\la)f\|_{L^{q_c}(M^{n-1}\times [0,\log\la \cdot \la^{-1}])}
\le C \, (\log\la)^{\frac{4-q_c}{2q_c}} \, \|f\|_{L^2(M^{n-1})},
\end{equation}
which yields the first estimate in Theorem~\ref{negthm} as
$$\tfrac1{n+1}=\tfrac1{q_c}-\tfrac{4-q_c}{2q_c}, \quad \text{if } \, \, d=n-1\ge3,$$
as well as \eqref{00.14} and hence \eqref{00.14'} since $q_c=4$ when
$d=n-1=2$.

In order to prove Proposition~\ref{mainprop}, as in earlier works, we shall use bilinear
techniques requiring us to compose the ``global operators'' $S_\la$ with related
local ones.  Motivated  by the recent work of the last two authors \cite{HSSchro},
our ``local'' auxiliary operators will be  the following ``quasimode'' operators adapted to the
scaled Schr\"odinger operators $\la D_t+\Delta_g$, 
\begin{equation}\label{22.5}
\sigma_\la = \sigma\bigl(\la^{1/2}|D_t|^{1/2}-P\bigr) \, \tilde \beta(D_t/\la),
\end{equation}
where
\begin{equation}\label{22.6}
\sigma\in {\mathcal S}(\R) \, \, \text{satisfies } \, \,
\sigma(0)=1 \, \, \text{and } \, \, \text{supp }\Hat \sigma
\subset \delta\cdot [1-\delta_0,1+\delta_0]=[\delta-\delta_0\delta, \delta+\delta_0\delta],
\end{equation}
with $0<\delta,\delta_0<1/8$ to be specified later, and, also here
\begin{equation}\label{22.7}
\tilde \beta\in C^\infty_0((1/8,8)) \quad \text{satisfies } \, \,
\tilde \beta=1 \, \, 
\text{on } \, \, [1/6,6].
\end{equation}
  We shall want $\delta$ in \eqref{22.6}
to be smaller than the injectivity radius of $(M^{n-1},g)$ and $\delta_0$ to be small enough
so that we can verify the hypotheses of the bilinear oscillatory integral estimates that we shall
use in the next section.

To handle the bilinear arguments it will be convenient
to introduce an initial microlocalization.  So,
let us write
\begin{equation}\label{22.8}
I=\sum_{j=1}^N B_j(x,D),
\end{equation}
where each $B_j\in S^0_{1,0}(M^{n-1})$ is a standard
pseudo-differential operator with symbol supported in
a small conic neighborhood of some $(x_j,\xi_j)\in
S^*M$.  The size of the support will be described
later; however, these operators will not depend
on our parameter $\la\gg 1$.  Next, if 
$\tilde \beta$ is as in \eqref{22.7} then the
dyadic operators
\begin{equation}\label{22.9}
B=B_{j,\la} = B_j\circ \tilde \beta(P/\la)
\end{equation}
are uniformly bounded on $L^p$, i.e.,
\begin{equation}\label{22.10}
\|B\|_{L^p(M^{n-1})\to L^p(M^{n-1})}
=O(1) \quad \text{for } \, \, 1\le p\le \infty.
\end{equation}
Also, note that since $\sigma\in {\mathcal S}(\R)$
a simple calculation shows that if $\la_k$ is an
eigenvalue of $P$
$$(1-\tilde \beta(\la_k/\la))
\, \sigma(\la^{1/2}|\tau|^{1/2}-P) \, 
\tilde \beta(\tau/\la)
=O(\la^{-N}(1+\la_k+|\tau|)^{-N}) \, \, \forall \, N.
$$
Consequently,
$$\|\sigma_\la -\tilde \beta(P/\la)\circ \sigma_\la
\|_{L^2(M^{n-1}\times [0,T])\to L^{q_c}(M^{n-1}\times [0,T])}
=O(\la^{-N}) \quad \forall \, N.$$

Thus, if $B_j$ is as in \eqref{22.8} and 
$B_{j,\la}$ is the corresponding dyadic operator
in \eqref{22.9}
\begin{equation}\label{22.11}
\|B_j\sigma_\la -B_{j,\la}\sigma_\la\|_{L^2(M^{n-1}\times [0,T])\to L^{q_c}(M^{n-1}\times [0,T])}
=O(\la^{-N}) \quad \forall \, N,
\end{equation}
since operators in $S^0_{1,0}(M^{n-1})$ are bounded on $L^p$ for $1<p<\infty$.

We need one more result for now about these local
operators:

\begin{lemma}\label{lemmadiff}
If $S_\la$ as in \eqref{22.2} and $\sigma_\la$ is as in \eqref{22.5} then
\begin{equation}\label{22.12}
\|(I-\sigma_\la)\circ S_\la f\|_{L^{q_c}(M^{n-1}\times [0,T])}\le CT^{\frac1{q_c}-\frac12}\la^{\frac1{q_c}}\|f\|_2.
\end{equation}
\end{lemma}

For a given $B=B_{j,\la}$ as in \eqref{22.9} let
us define the microlocalized variant of 
$\sigma_\la$ as follows
\begin{equation}\label{22.13}
\tilde \sigma_\la = B\circ \sigma_\la, 
\quad B=B_{j,\la},
\end{equation}
and the associated ``semi-global'' operators
\begin{equation}\label{22.14}
\tilde S_\la =\tilde \sigma_\la \circ S_\la.
\end{equation}

By \eqref{22.8}, \eqref{22.11} and \eqref{22.12},
in order to prove Proposition~\ref{mainprop}, it suffices
to show that if $T=c_0\log\la$ with 
$c_0>0$ sufficiently small (depending on $M^{n-1}$), then, if all the sectional curvatures of
$M^{n-1}$ are nonpositive,
\begin{equation}\label{22.3'}\tag{2.3$'$}
\|\tilde S_\la f\|_{L^{q_c}(M^{n-1}\times \R)} \le C
\la^{\frac1{q_c}} T^{\frac1{q_c}\cdot \frac2{q_c}} \, \|f\|_{L^2(M^{n-1})},
\end{equation}
and, if all of the sectional curvatures of $M^{n-1}$ are negative,
\begin{equation}\label{22.4''}\tag{2.4$''$}
\|\tilde S_\la f\|_{L^{q_c}(M^{n-1}\times \R)}\le C\la^{\frac1{q_c}}
T^{\frac{4-q_c}{2q_c}} \, \|f\|_{L^2(M^{n-1})}.
\end{equation}

As we shall see, in order to prove \eqref{22.3'} and \eqref{22.4''} we shall need to take $\delta$ and $\delta_0$ in
\eqref{22.6} and \eqref{22.7} to be sufficiently small for each $j$; however, since, by the compactness of $M^{n-1}$
and the arguments to follow, the sum in \eqref{22.8} can be taken to be finite,
 we can take these two parameters to be the minimum over what is needed for $j=1,\dots,N$.

\begin{proof}[Proof of Lemma~\ref{lemmadiff}]

We shall follow the strategy in \cite{HSSchro}.  In proving \eqref{22.12} we may assume,
as we shall throughout, that
\begin{equation}\label{normalize}\|f\|_2=1.
\end{equation}
Also, we  notice that, if $E_kf$ denotes the projection of $f$ onto the eigenspace
of $P=\sqrt{-\Delta_g}$ with eigenvalue $\la_k$, we have
\begin{align*}
S_\la f(x,t)&=\sum_k  \eta(t/T) e^{-it\la^{-1}\la_k^2} \beta(\la_k/\la)E_kf(x)
\\
&=(2\pi)^{-1}\sum_k \int_{-\infty}^\infty e^{it\tau} T\Hat \eta(T(\tau-\la^{-1}\la_k^2)) \, \beta(\la_k/\la) E_kf(x)\, d\tau.
\end{align*}
Since, by \eqref{00.11}, $\beta(s)=0$ if $s\notin [1/2,2]$, $\Hat \eta\in {\mathcal S}(\R)$ and
$\tilde \beta(s)=1$ for $s\in [1/6,6]$, it is not difficult to check that
$$(1-\tilde \beta(\tau/\la)) \, T\Hat \eta(T(\tau-\la^{-1}\la_k^2)) \, \beta(\la_k/\la)
=O(\la^{-N}(1+|\tau|)^{-N}) \, \, \forall \, N,
$$
and so trivially
$$\|(I-\tilde \beta(D_t/\la))S_\la f\|_{L^{q_c}(M^{n-1}\times [0,T])}=O(\la^{-N}) \, \, \forall N.$$
Consequently, in order to prove \eqref{22.12}, it suffices to show that
\begin{equation}\label{diff'}
\bigl\| \, (I-\sigma(\la^{1/2}|D_t|^{1/2}-P))\circ \tilde \beta(D_t/\la) S_\la f \, \bigr\|_{L^{q_c}(M^{n-1}\times [0,T])}
\le CT^{\frac1{q_c}-\frac12}\la^{\frac1{q_c}}.
\end{equation}

To prove this  let 
\begin{equation}\label{22.16}
\alpha\in C^\infty_0((-1,1)) \quad \text{satisfy } \, \, 1\equiv \sum_{m=-\infty}^\infty \alpha_m(t),
\end{equation}
if
\begin{equation}\label{22.17}
\alpha_m(t)=\alpha(t-m).
\end{equation}
Then, in order to prove \eqref{diff'},
it suffices to see that
\begin{multline}\label{suffice}
\bigl\|\alpha_m(t)\bigl(I-\sigma(\la^{1/2}|D_t|^{1/2}-P)) \,  \tilde \beta(D_t/\la) \bigr) \circ
\bigl( \eta(t/T)\bigr)e^{-it\la^{-1}\Delta_g} \beta(P/\la)f\bigr) \bigr\|_{L^{q_c}_{t,x}}
\\
\lesssim T^{-1/2} \la^{1/q_c}.
\end{multline}
Call $w$ the function in the norm in the left, i.e.,
\begin{equation}\label{wfor}
w= \alpha_m(t)\bigl(I-\sigma(\la^{1/2}|D_t|^{1/2}-P)) \tilde \beta(D_t/\la) \bigr) \circ
\bigl( \eta(t/T)e^{-it\la^{-1}\Delta_g} \beta(P/\la)f\bigr).
\end{equation}
It is supported in $[m-1,m+1]$, and so is 
\begin{equation}\label{F}
F
=(i\la\partial_t-\Delta_g)w.
\end{equation}

For later use, note that, by \eqref{00.11} and  \eqref{22.7} 
\begin{equation}\label{F'}
(I-\tilde \beta(P/\la))F=0.
\end{equation}
Also, by the Duhamel formula  for the scaled Schr\"odinger equation and the above support properties
$$w(x,t)= (i\la)^{-1}\int_{m-1}^t\bigl(e^{-i\la^{-1}(t-s)\Delta_g}F(s,\, \cdot\, )\bigr)(x) \, ds.$$
So, by Minkowski's inequality, for each fixed $t$,
\begin{align*}
\|w(\, \cdot \, ,t)\|_{L^{q_c}_x(M^{n-1})}&\le \la^{-1}\int_{m-1}^t \bigl\| e^{-i\la^{-1}t\Delta_g}\bigl(e^{i\la^{-1}s\Delta_g}F(s, \, \cdot \, )\bigl) \, \big\|_{L^{q_c}(M^{n-1})}\, ds
\\
&\le \la^{-1}\int_{-1}^{1} \bigl\| e^{-i\la^{-1}t\Delta_g}\bigl(e^{i\la^{-1}(s+m)\Delta_g}F(s+m, \, \cdot \, )\bigl) \, \big\|_{L^{q_c}(M^{n-1})}\, ds,
\end{align*}
since $F(s,\cdot)=0$ if $s\notin [m-1,m+1]$.
Thus, by Minkowski's inequality, we have
$$\|w\|_{L^{q_c}_{x,t}}\le  \la^{-1}\int_{-1}^{1} \bigl\| e^{-i\la^{-1}t\Delta_g}\bigl(e^{i\la^{-1}(s+m)\Delta_g}F(s+m, \, \cdot \, )\bigl) \, \big\|_{L^{q_c}_{t,x}} \, ds.$$

Furthermore, by \eqref{F'} we can use the local dyadic estimates \eqref{00.3''} of Burq, G\'erard and Tzvetkov along with Schwarz's inequality to obtain
\begin{align*}
\|w\|_{L^{q_c}_{t,x}}&\lesssim \la^{\frac1{q_c}-1}  \int _{-1}^1 \bigl\| e^{i\la^{-1}(s+m)\Delta_g }F(s+m, \, \cdot \, )\bigr\|_{L^2_x} \, ds
\\
& \le \la^{\frac1{q_c}-1} \int_{-1}^1\|F(s+m,\, \cdot \, )\|_{L^2_x} \, ds\notag
\\
&\lesssim \la^{\frac1{q_c}-1}\|F\|_{L^2_{t,x}}.  \notag
\end{align*}
If we put
$$I=\la \, \bigl\|
\alpha'_m(t) \bigl(I-\sigma(\la^{1/2}|D_t|^{1/2}-P)) \tilde \beta(D_t/\la) \circ (\eta(t/T) e^{-it\la^{-1}\Delta_g}\beta(P/\la)f) \, \bigr\|_{L^2_{t,x}},
$$
and
$$II=
\|\alpha_m(t) (i\la\partial_t+P^2)\bigl(I-\sigma(\la^{1/2}|D_t|^{1/2}-P)) \tilde \beta(D_t/\la) \circ (\eta(t/T) e^{-it\la^{-1}\Delta_g}\beta(P/\la)f) \, \bigr\|_{L^2_{t,x}},$$
we conclude that from \eqref{F} and \eqref{F'} that
\begin{equation}\label{999}
\|w\|_{L^{q_c}_{t,x}}\lesssim \la^{\frac1{q_c}-1} \bigl(I + II\bigr).
\end{equation}

To handle $I$ we note that the function in the norm can be written as
$$(2\pi)^{-1}\alpha'_m(t) 
\sum_k \int e^{it\tau}  (1-\sigma(\la^{1/2}\tau^{1/2}-\la_k)) \, \tilde \beta(\tau/\la)  \, 
 T\Hat \eta(T(\tau-\la^{-1}\la_k^2)) \beta(\la_k/\la) E_kf\, d\tau,
$$
Thus, by orthogonality and the support properties of $\tilde\beta$ in \eqref{22.7}, since we are assuming that $\|f\|_2=1$, we have
$$I\le \la T \sup_{\la_k\approx \la} \Bigl( \int_{\la/8}^{8\la}
|1-\sigma(\la^{1/2}\tau^{1/2}-\la_k))|^2 \, 
|\Hat \eta(T \, (\tau-\la^{-1}\la_k^2)|^2 \, d\tau \Bigr)^{1/2},$$
If we change variables $s=\la^{1/2}\tau^{1/2}$ then $ds\approx d\tau$ in the support of the integrand, and so  by the above
\begin{multline*}
I\lesssim \la T \sup_{\la_k\approx \la}
\Bigl(\int_0^\infty |1-\sigma(s-\la_k)|^2 \,  |\Hat \eta(T\la^{-1}(s+\la_k)(s-\la_k))|^2\, ds\Bigr)^{1/2}
\\
=
 \la T \sup_{\la_k\approx \la} \Bigl(\int_0^\infty |1-\sigma(s)|^2\,  |\Hat \eta(T\la^{-1}(s+2\la_k)\cdot s)|^2\, ds\Bigr)^{1/2}
\\
\lesssim\la  T\Bigl(\int_0^\infty s^2(1+|Ts|)^{-N} \, ds \Bigr)^{1/2} \approx \la T^{-1/2},
\end{multline*}
using, in the last step, our assumption in \eqref{22.6} that $\sigma(0)=0$.

If we repeat the arguments we find that 
\begin{multline*}
II \lesssim T \sup_{\la_k\approx \la}\Bigl(\int_{\la/8}^{8\la} 
 |(1-\sigma(\la^{1/2}\tau^{1/2}-\la_k))|^2 \, |-\la\tau +\la^2_k|^2\, |\Hat \eta(T(\tau-\la^{-1}\la_k^2))|^2 \, d\tau)^{1/2}
\\
\le \la \sup_{\la_k\approx \la} \Bigl(\int_0^\infty  
%(1-\sigma(\la^{1/2}\tau^{1/2}-\la_k))|^2 \, 
  \bigl| T(\tau -\la^{-1}\la^2_k) \cdot \Hat \eta(T(\tau-\la^{-1}\la_k^2))
\bigr|^2 \, d\tau\Bigr)^{1/2}
\\
\lesssim \la \Bigl(\int_{-\infty}^\infty (1+T|\tau |)^{-N}\, d\tau\Bigr)^{1/2} =O(\la T^{-1/2}).\end{multline*}

If we combine these two estimates and use \eqref{999} we conclude that
$$\|w\|_{L^{q_c}_{t,x}}\lesssim  T^{-1/2}\la^{\frac1{q_c}},
$$
as posited in \eqref{suffice}, which finishes the proof.
\end{proof}

For later use, let us also see that this argument yields the following  result, which we shall need when
we use local variable coefficient bilinear harmonic analysis techniques.

\begin{lemma}\label{qlemma}  If $\alpha_m$
is as in \eqref{22.17} then for
$m\in {\mathbb Z}$ we have
\begin{equation}\label{qu1}
\bigl\| \alpha_m(t) \sigma_\la H\bigr\|_{L^{q_c}_{t,x}}
\le C\la^{\frac1{q_c}}\|H\|_{L^2(M^{n-1}\times
[m-10,m+10])} +C_N \la^{-N}\|H\|_{L^2(M^{n-1}\times 
\R)},
\end{equation}
for every $N=1,2,\dots$.
\end{lemma}

\begin{proof}
If $\{e_k\}$ is an orthonormal basis of eigenfunctions
of $P$ on $M^{n-1}$ with eigenvalues $\{\la_k\}$ then
the kernel $\sigma_\la(x,t;y,s)$ of $\sigma_\la$ is
\begin{multline*}
(2\pi)^{-1}\sum_k \int e^{i(t-s)\tau} \sigma(\la^{1/2}\tau^{1/2}-\la_k)\, \tilde \beta(\tau/\la)\,  e_k(x)\overline{e_k(y)} \, d\tau
\\
=(2\pi)^{-2} \iint
e^{i(t-s)\tau}e^{i\la^{1/2}\tau^{1/2}r} \,
\tilde \beta(\tau/\la) \, \Hat \sigma(r)
\sum_k e^{-ir\la_k}e_k(x)\overline{e_k(y)} \, dr d\tau.
\end{multline*}
Recall that, by \eqref{22.6}, $\Hat \sigma(r)=0$
if $r\notin [0,1]$.  Therefore, by \eqref{22.7} and a
simple integration by parts argument we have that
\begin{multline*}\iint  
e^{i(t-s)\tau}e^{i\la^{1/2}\tau^{1/2}r} \,
\tilde \beta(\tau/\la) \, \Hat \sigma(r)
 e^{-ir\la_k} \, dr d\tau
=O\bigl((\la+\la_k+|t-s|)^{-N}\bigr), 
\\
\text{if } \, |t-s|\ge 5
.
% \, \, 
%\text{or } \, \, \la_k\notin [\la/100,100 \la].
\end{multline*}
If $\la_k\le 100\la$ one obtains these bounds just
by integrating by parts in $\tau$, while if $\la_k>100\la$ one integrates by parts in both
variables $r$ and $\tau$ to obtain this bound.
Since, by the pointwise Weyl formula (see e.g.
\cite{SFIO2}),
$$
%\sum_{\la_k\in [\la/100,100 \la]}
%|e_k(x)e_k(y)|=O(\la^{n-1}) \, \,
%\text{and } \, \,
\sum_k (1+\la_k)^{-n}|e_k(x)e_k(y)|=O(1),$$
we conclude that
$$\sigma_\la(x,t;y,s)=O((|t-s|+\la)^{-N}) \, \forall \, N,
\quad \text{if } \, \, |t-s|\ge5.$$

Thus, if $H(t,x)=0$ for $t\in [m-10,m+10]$, then the
left side of \eqref{qu1} is dominated by the second
term in the right.  Consequently, to prove 
\eqref{qu1} we may assume that
$$H(t,x)=0 \quad \text{if } \, \, t\notin
[m-10,m+10].$$
If we then let
$$w(x,t)=\alpha_m(t) \bigl(\sigma_\la H\bigr)(x,t), 
\quad \text{and } \, F=(i\la\partial_t-\Delta_g)w$$
and argue as in the proof of Lemma~\ref{lemmadiff},
it suffices to show that
$$\|F\|_{L^2_{t,x}}\lesssim \la \|H\|_{L^2_{t,x}},$$
which would follow from
\begin{equation}\label{qu2} 
\| \, \sigma(\la^{1/2}|D_t|^{1/2}-P) \, 
\tilde \beta(D_t/\la) \, H\, \|_{L^2_{t,x}}
\lesssim \|H\|_{L^2_{t,x}},
\end{equation}
and
\begin{equation}\label{qu3}
\| \,  (i\la\partial_t -\Delta_g) \,[\sigma(\la^{1/2}|D_t|^{1/2}-P) \, 
\tilde \beta(D_t/\la) \, H]\, \|_{L^2_{t,x}}
\lesssim \la \|H\|_{L^2_{t,x}}.
\end{equation}

By orthogonality and the arguments in the proof of Lemma~\ref{lemmadiff}, \eqref{qu2} just follows from the
fact that
$$\sigma(\la^{1/2}\tau^{1/2}-\mu) \, \tilde \beta(\tau/\la)
=O(1),$$
and, \eqref{qu3} is a consequence of the bound
\begin{multline*}
-(\la\tau-\mu^2) \, \sigma(\la^{1/2}\tau^{1/2}-\mu)
\, \tilde \beta(\tau/\la)
\\
=
-(\la^{1/2}\tau^{1/2}+\mu)  \, \tilde \beta(\tau/\la)\, 
\cdot \bigl[
\, (\la^{1/2}\tau^{1/2}-\mu) \, \sigma(\la^{1/2}\tau^{1/2}-\mu) \, \bigr]
=O(\la),
\end{multline*}
which follows from \eqref{22.7} and the fact
that $\sigma\in {\mathcal S}(\R)$.
\end{proof}

\noindent{\bf 2.1. Height Decomposition.}

Next we set up a variation of an argument of
Bourgain~\cite{BourgainBesicovitch} originally used
to study Fourier transform restriction problems, and, more
recently, to study eigenfunction problems in
\cite{BHSsp}, \cite{SBLog} and \cite{sogge2015improved}.
This involves splitting the estimates in 
Proposition~\ref{mainprop} into two heights
involving relatively large and small values of
$|\tilde S_\la f(x,t)|$.  

To describe this, here, and in what follows we shall
assume, as we just did, that $f$ is $L^2$-normalized
as in \eqref{normalize}.  Then, we shall prove
the estimates in Proposition~\ref{mainprop}, using
very different techniques by estimating $L^{q_c}$ bounds
over the two regions
\begin{multline}\label{22.24}A_+=\{(x,t)\in  M^{n-1}\times [0,T]: \, |\tilde S_\la f(t,x)|\ge \la^{\frac{n-1}4+\delta} \},
 \\
  \text{and } \, \, A_-=\{(x,t)\in M^{n-1} \times  [0,T] : \, |\tilde S_\la f(x,t)|< \la^{\frac{n-1}4+\delta} \}.
  \end{multline}
Due to the numerology of the powers of $\la$ arising,
the splitting occurs at height $\la^{\frac{n-1}4+\delta}$, $\delta=1/8$; however, we could have replaced this 
specific value of $\delta$ by any sufficiently
small positive $\delta$.  The transition occurring at, basically, $\la^{\frac{n-1}4}$ is natural and arises
due to Knapp-type phenomena, both in Euclidean problems,
as well as geometric ones that we are considering here.  We choose this specific value of $\delta=1/8$ to 
simplify some of the calculations to follow.

We next notice that Proposition~\ref{mainprop}
(and hence Theorems \ref{nonposthm} and \ref{negthm})
are a consequence of the following two propositions
corresponding to the two regions in \eqref{22.24}.

\begin{proposition}\label{largeprop}  Let $(M^{n-1},g)$,
$n\ge3$ have nonpositive curvature.  We then can
choose $c_0>0$ so that for $\la \gg1$ and
$T=c_0\log\la$ we have the uniform bounds
\begin{equation}\label{high}
\|\tilde S_\la f\|_{L^{q_c}(A_+)}
\le C\la^{\frac1{q_c}},
\end{equation} 
assuming that $f$ is $L^2$-normalized as in 
\eqref{normalize}.
\end{proposition}

\begin{proposition}\label{smallprop}  Let $(M^{n-1},g)$,
$n\ge3$, have nonpositive curvature.  We then can
take $T=c_0\log\la$ as above so that, if $f$ is
$L^2$-normalized,
\begin{equation}\label{22.26}
\|\tilde S_\la f\|_{L^{q_c}(A_-)}
\le C\la^{\frac1{q_c}}T^{\frac1{q_c}\cdot \frac2{q_c}}.
\end{equation}
Furthermore, if all the sectional curvatures of
$(M^{n-1},g)$ are negative,
\begin{equation}\label{22.27}
\|\tilde S_\la f\|_{L^{q_c}(A_-)}
\le C\la^{\frac1{q_c}}T^{\frac{4-q_c}{2q_c}}.
\end{equation}
\end{proposition}

We shall present the proofs of these Propositions in
the next two subsections.

\medskip

 \noindent{\bf 2.2. Estimates for relatively large values: Proof of Proposition~\ref{largeprop}.}

We first note that, by Lemma~\ref{lemmadiff} 
and \eqref{22.10}
we have
$$ \|\tilde S_\la f\|_{L^{q_c}(A_+)} \le \| B S_\la f\|_{L^{q_c}(A_+)} + CT^{\frac1{q_c}-\frac12}\la^{\frac1{q_c}},
$$
and, since $q_c>2$,  \eqref{high} would follow from
\begin{equation}\label{high'}
\| BS_\la f\|_{L^{q_c}(A_+)}\le C\la^{\frac1{q_c}}+\tfrac12 \|\tilde S_\la f\|_{L^{q_c}(A_+)}.
\end{equation}

To prove this we shall adapt an argument of
Bourgain~\cite{BourgainBesicovitch} and more
recent variants in
\cite{BHSsp} and  \cite{sogge2015improved} .
Specifically, choose $g(x,t)$ such that
$$\|g\|_{L^{q_c'}(A_+)}=1
\quad \text{and } \, \,
\|BS_\la f\|_{L^{q_c}(A_+)}
=\iint BS_\la f\cdot
\overline{\bigl(\1_{A_+}\cdot g\bigr)} \, dx dt.
$$
Then, since we are assuming that $\|f\|_2=1$, by
the Schwarz inequality
\begin{align}\label{22.29}
\|BS_\la f\|^2_{L^{q_c}(A_+)}
&= \Bigl( \, \int f(x) \, \cdot \,  
\overline{\bigl(S^*B^*\bigr)(\1_{A_+}
\cdot g\bigr)(x)} \, dx \, \Bigr)^2
\\
&\le \int |S^*_\la B^*(\1_{A_+}\cdot g)(x)|^2 \, dx
\notag
\\
&=\iint \bigl(BS_\la S^*_\la B^*\bigr)(\1_{A_+}\cdot
g)(x,t) \, \overline{(\1_{A_+}\cdot
g)(x,t)} \, dx dt
\notag
\\
&=\iint \bigl(B\circ L_\la\circ B^*\bigr)(\1_{A_+}\cdot
g)(x,t) \, \overline{(\1_{A_+}\cdot
g)(x,t)} \, dx dt \notag
\\
&\qquad+\iint \bigl(B\circ G_\la\circ B^*\bigr)(\1_{A_+}\cdot
g)(x,t) \, \overline{(\1_{A_+}\cdot
g)(x,t)} \, dx dt \notag
\\
&=I + II, \notag
\end{align}
where $L_\la$ is the integral operator with kernel equaling that of $S_\la S^*_\la$ if $|t-s|\le 1$ and 
$0$ otherwise, i.e,
\begin{multline}\label{ok}
L_\la(x,t;y,s)=
\\
\begin{cases}
\bigl(S_\la S_\la^*\bigr)(x,t;y,s)=
\eta(t/T)\eta(s/T) \bigl(\, 
\beta^2(P/\la)e^{-i(t-s)\la^{-1}\Delta_g}\bigr)(x,y), \, \,
\text{if } \, \, |t-s|\le 1,
\\
0 \, \, \, \text{otherwise}.
\end{cases}
\end{multline}

In the final section (see Proposition~\ref{kerprop})
we shall show that for $T$ as above we have
\begin{equation}\label{k}
|(S_\la S_\la^*)(x,t;y,s)|\le C
\la^{\frac{n-1}2} \, |t-s|^{-\frac{n-1}2}
\exp(C_M|t-s|), \quad \text{if } \, \, |t-s|\le 2T.
\end{equation}
Consequently, if we let $L_{\la,t,s}$ be the ``frozen''
operators
$$\bigl(L_{\la,t,s}f\bigr)(x)
=\int L_\la(x,t;y,s) \, f(y) \, dy,
$$
we have that
$$\|L_{\la,t,s}f\|_{L^\infty(M^{n-1})}
\le C\la^{\frac{n-1}2} \, |t-s|^{-\frac{n-1}2}
\, \|f\|_{L^1(M^{n-1})},
$$
and, since $e^{-i(t-s)\la^{-1}\Delta_g}$ is unitary,  we of course have
$$\|L_{\la,t,s}f\|_{L^2(M^{n-1})}
\le C
 \|f\|_{L^2(M^{n-1})}.
$$
Therefore, by interpolation
$$\|L_{\la,t,s}f\|_{L^{q_c}(M^{n-1})}
\le C \la^{\frac2{q_c}} \, |t-s|^{-\frac2{q_c}}
 \|f\|_{L^{q_c'}(M^{n-1})}.
$$
Therefore, by Strichartz's \cite{Strichartz77}
original argument (or, e.g., Theorem 0.3.6 in \cite{SFIO2}),
we can use the classical Hardy-Littlewood fractional integral
estimates to conclude that
$$\|L_\la\|_{L^{q_c'}(M^{n-1}\times \R)
\to L^{q_c}(M^{n-1}\times \R)}=O(\la^{\frac2{q_c}}).$$
If we use this, along with H\"older's inequality and
\eqref{22.10}, we obtain for the term $I$ in \eqref{22.29}
\begin{align}\label{k301}
|I|&\le \|BL_\la B^*(\1_{A_+}\cdot g)\|_{L^{q_c}(M^{n-1}\times \R)}
\cdot \| \1_{A_+}\cdot g \|_{L^{q_c'}(M^{n-1}\times \R)}
\\
&\lesssim \|L_\la B^*(\1_{A_+}\cdot g)\|_{L^{q_c}(M^{n-1}\times \R)} 
\cdot
\|\1_{A_+}\cdot g \|_{L^{q_c'}(M^{n-1}\times \R)}
\notag
\\
&\lesssim \la^{\frac2{q_c}} \|B^*(\1_{A_+}\cdot g) \|_{L^{q_c'}(M^{n-1}\times \R)}
\cdot
\|\1_{A_+}\cdot g \|_{L^{q_c'}(M^{n-1}\times \R)}
\notag
\\
&\lesssim \la^{\frac2{q_c}}\|g\|^2_{L^{q_c'}(A_+)}
=\la^{\frac2{q_c}}. \notag
\end{align}

To estimate  the other term in \eqref{22.29}, $II$, we choose $c_0$ small enough
so that if $C_M$ is the constant
in \eqref{k}
$$\exp(2C_MT)\le \la^{1/8}, \quad \text{if  }
\,
T=c_0\log\la \, \, \text{and } \, \, \la \gg 1.
$$
Then, since $\eta(t)=0$ for $|t|\ge1$, it follows
from \eqref{ok} and \eqref{k} that
$$\|G_\la \|_{L^1(M^{n-1}\times \R) 
\to L^\infty (M^{n-1}\times \R)}\le C
\la^{\frac{n-1}2+\frac18}.
$$
As a result, since, by \eqref{22.10}, the dyadic operators $B$ are bounded
on $L^1$ and $L^\infty$, we can repeat the arguments
to estimate $I$ and use
H\"older's inequality to see that
$$|II|\le C\la^{\frac{n-1}2}\la^{\frac18}
\| \1_{A_+}\cdot g\|^2_1
\le C \la^{\frac{n-1}2}\la^{\frac18}
\|g\|^2_{L^{q_c'}(A_+)} \cdot
\|\1_{A_+}\|_{L^{q_c}}^2
=C \la^{\frac{n-1}2}\la^{\frac18}
\|\1_{A_+}\|_{L^{q_c}}^2.
$$
If we recall the definition of $A_+$ in
\eqref{22.24}, we can estimate the last
factor:
$$\|\1_{A_+}\|_{L^{q_c}}^2
\le \bigl(\la^{\frac{n-1}4+\frac18}\bigr)^{-2}
\|\tilde S_\la f\|^2_{L^{q_c}(A_+)}.$$
Therefore, 
$$|II|\lesssim \la^{-\frac18}
\|\tilde S_\la f\|^2_{L^{q_c}(A_+)}
\le \bigl(\tfrac12 \|\tilde S_\la f\|_{L^{q_c}(A_+)}
\bigr)^2,
$$
assuming, as we may, that $\la$ is large enough.

If we combine this bound with the earlier one,
\eqref{k301} for $I$, we conclude that
\eqref{high'} is valid, which completes the 
proof of Proposition~\ref{largeprop}. \qed

\medskip

\noindent{\bf 2.3.   Estimates for relatively small values: Proof of Proposition \ref{largeprop}.}

We now turn to the proving the $L^{q_c}(A_-)$ estimates  in Proposition~\ref{largeprop}.  To do this we
need to borrow and adapt results from the bilinear harmonic analysis in \cite{LeeBilinear} and
\cite{TaoVargasVega}.

We shall utilize a microlocal decomposition which we shall now describe.
We first recall that the symbol $B(x,\xi)$ of $B$ in \eqref{22.9} is supported in a small
conic neighborhood of some $(x_0,\xi_0)\in S^*M^{n-1}$.  We may assume that its symbol has
small enough support so that we may work in a coordinate chart $\Omega$ and that
$x_0=0$, $\xi_0=(0,\dots,0,1)$ and $g_{jk}(0)=\delta^j_k$ in the local coordinates.
So, we shall assume that $B(x,\xi)=0$ when $x$ is outside a small relatively compact neighborhood
of the origin or $\xi$ is outside of a small conic neighborhood of $(0,\dots,0,1)$.  These reductions
and those that follow will contribute to the number of terms in \eqref{22.8}; however, it will be clear
that the $N$ there will be independent of $\la\gg 1$.  Similarly, the positive numbers $\delta$ and
$\delta_0$ in \eqref{22.7} may depend on $N$, but, at the end we can just take each to be the minimum
of what is required for each $j=1,\dots,N$.

Next, let us define the microlocal cutoffs that we shall use.   We fix a function
$a\in C^\infty_0({\mathbb R}^{2(n-2)})$ supported in $\{z: \, |z_j|\le 1, \, \, 1\le j\le 2(n-2)\}$
 which satisfies
\begin{equation}\label{m1}
\sum_{j\in {\mathbb Z}^{2(n-2)}}a(z-j)\equiv 1
%, \quad
%a(z)=0, \, \, \, |z|\ge2 \quad \text{and } \, \, a(z)a(z-k)\equiv 0, \, \, \, \text{if } \, \, |k|\ge2
.
\end{equation}
We shall use this function to build our microlocal cutoffs.
By the above, we shall focus on defining them 
 for $(y,\eta)\in S^*\Omega$ with    $y$ near the origin
 and  $\eta$ in a small conic neighborhood of $(0,\dots,0,1)$. 
We shall let
$$\Pi=\{y: \, y_{n-1}=0\}$$
be the points in $\Omega$ whose last coordinate vanishes.   Let $y'=(y_1,\dots, y_{n-2})$ and
$\eta=(\eta_1,\dots,\eta_{n-2})$ denote the first $n-2$ coordinates of $y$ and $\eta$, respectively.
 For $y\in \Pi$ near $0$ and $\eta$ near $(0,\dots,0,1)$ we can
just use the functions $a(\theta^{-1}(y',\eta')-j)$, $j\in {\mathbb Z}^{2(n-2)}$ to obtain cutoffs of scale $\theta$.  We will always have
$\theta\in [\la^{-\delta},1]$ with $\delta=1/8$.

We can then extend the definition to a neighborhood of $(0,(0,\dots,0,1))$ by setting for $(x,\xi)\in S^*\Omega$ in this neighborhood
\begin{equation}\label{m2}
a^\theta_j(x,\xi)=a(\theta^{-1}(y',\eta')-j) \quad
\text{if } \, \, \chi_s(x,\xi)=(y',0,\eta',\eta_{n-1}) \, \, \, \text{with } \, \, \, s=d_g(x,\Pi).
\end{equation}
Here $\chi_s$ denotes geodesic flow in $S^*\Omega$.  Thus, $a^\theta_j(x,\xi)$ is constant on all geodesics
$(x(s),\xi(s))\in S^*\Omega$ with $x(0)\in \Pi$ near $0$ and $\xi(0)$ near $(0,\dots,0,1)$.   As a result,
\begin{equation}\label{m3}
a^\theta_j(\chi_s(x,\xi))=a^\theta_j(x,\xi)
\end{equation}
for $s$ near $0$ and $(x,\xi)\in S^*\Omega$ near $(0,(0,\dots,0,1))$.

We then extend the definition of the cutoffs to a conic neighborhood of $(0,(0,\dots,0,1))$  in $T^*\Omega \, \backslash \, 0$ by setting
\begin{equation}\label{m4}
a^\theta_j(x,\xi)=a^\theta_j(x,\xi/p(x,\xi)).
\end{equation}

Notice that if $(y'_j,\eta'_j)=\theta j$ and $\gamma_j$ is the geodesic in $S^*\Omega$ passing through $(y'_j,0,\eta_j)\in S^*\Omega$
with $\eta_j\in S^*_{(y'_j,0)}\Omega$ having $\eta'_j$ as its first $(n-2)$ coordinates then
\begin{equation}\label{m5}
a^\theta_j(x,\xi)=0 \quad \text{if } \, \, \,
\text{dist }\bigl((x,\xi), \gamma_j\bigr)\ge C_0\theta,
\end{equation}
for some fixed constant $C_0>0$.  Also,  $a^\theta_j$ satisfies the estimates
\begin{equation}\label{m6}
\bigl|\partial_x^\sigma \partial_\xi^\gamma a^\theta_j(x,\xi)\bigr| \lesssim \theta^{-|\alpha|-|\gamma|}, \, \, \,
(x,\xi)\in S^*\Omega
\end{equation}
related to this support property.

The $a^\theta_j$ provide ``directional'' microlocalization.  We also need a ``height'' localization since the characteristics of
the symbols of our scaled Schr\"odinger operators lie on paraboloids.  The variable coefficient operators that we shall use of course are 
adapted to our operators and are analogs of ones that are used in the study of Fourier restriction problems involving paraboloids.

To construct these, choose $b\in C^\infty_0(\R)$ supported in $|s|\le1$  satisfying $\sum_{-\infty}^\infty b(s-\ell)\equiv 1$.
 %and $b(s-\ell)b(s)\equiv 0$ if $|\ell|\ge2$.  
 We then simply define the ``height operator'' as follows
\begin{equation}\label{m7}
A_\ell^\theta(P)=
b(\theta^{-1}\la^{-1}(P -\la\kappa^\theta_\ell)) \, 
\Upsilon(P/\la),
\quad \kappa^\theta_\ell = 1+\theta\ell, \quad
|\ell|\lesssim \theta^{-1},
\end{equation}
where if $\tilde \beta$ is as in \eqref{22.7}
\begin{equation}\label{ups}
\Upsilon\in C^\infty_0((1/10,10)) \, \, \text{satisfies } \, \, \,
\Upsilon(r)=1 \, \, \, \text{in a neighborhood of} \, \, \text{supp } \, \tilde \beta.
\end{equation}
Thus, these operators microlocalize $P$ to 
intervals of size $\approx \theta\la$ about ``heights''
$\la\kappa^\theta_\ell\approx \la$.  As we shall see below,
different ``heights'' will give rise to different
``Schr\"odinger tubes" about which the kernels of 
our microlocalization of the $\tilde \sigma_\la$
operators are highly concentrated.
Also, standard arguments as in \cite{SFIO2} show that if $A^\theta_\ell(x,y)$ is the kernel of this operator then
\begin{equation}\label{m8}
A^\theta_\ell(x,y)=O(\la^{-N}) \, \forall \, N, \quad
\text{if } \, d_g(x,y)\ge C_0\theta,
\end{equation}
for a fixed constant if $\theta\in [\la^{-\delta_0},1]$ with, as we are assuming $\delta_0<1/2$.

If $\psi(x)\in C^\infty_0(\Omega)$ equals $1$ in a neighborhood of the $x$-support of the $B(x,\xi)$ and
$A^\theta_j(x,D_x)$ is the operator with symbol
\begin{equation}\label{m9}
A^\theta_j(x,\xi)=\psi(x) a^\theta_j(x,\xi),
\end{equation}
then for $\nu=(\theta j, \theta \ell)\in \theta{\mathbb Z}^{2(n-2)+1}$ we can finally define the cutoffs that we shall use:
\begin{equation}\label{m10}
A^\theta_\nu=A^\theta_j(x,D_x)\circ A^\theta_\ell(P).
\end{equation}
For later use, we note that if 
$A^\theta_\nu(x,\xi)$ and $A^\theta_{\tilde \nu}(x,\xi)$ are 
the symbols of $A^\theta_\nu$ and $A^\theta_{\tilde \nu}$, respectively, then
\begin{equation}\label{sep0}
A^\theta_\nu(x,\xi)A^\theta_{\tilde \nu}(x,\xi)\equiv 0, \quad \text{if } \, \, \, |\nu-\tilde \nu|\ge C_0 \theta,
\end{equation}
for some uniform constant $C_0$.  Also, since $p(x,\xi)$ is invariant under the geodesic flow,
by by \eqref{m3} we have that the principal symbol $a_\nu^\theta(x,\xi)$ of $A^\theta_\nu$ satisfies
\begin{equation}\label{invariant}
a^\theta_\nu(\chi_r(x,\xi))=a^\theta_\nu(x,\xi), \, \, \text{on supp } B(x,\xi) \, \,
\text{if } \, \, |r|\le 2\delta,
\end{equation}
assuming that $\delta>0$ is small, and, as we may assume, the symbol $B(x,\xi)$ is supported in a small
conic neighborhood of $(0, \, (0,\dots,0,1))$.

Note also that, if $\theta\in [\la^{-\delta_0},1]$, then the $A^\theta_\nu$ belong to a bounded subset of
$S^0_{1-\delta_0,\delta_0}(M)$ (pseudo-differential operators of order zero and
type $(1-\delta_0,\delta_0)$).

Also, as operators between any  $L^p\to L^q$, $1\le p,q\le \infty$, spaces we have
\begin{equation}\label{m11}
\tilde \sigma_\la =
\sum_\nu \tilde \sigma_\la A^\theta_\nu  +O(\la^{-N}) \, \, \, \forall N,
\end{equation}
and the $A^\theta_\nu$ are almost orthogonal in the sense that we have 
\begin{equation}\label{m12}
\sum_\nu \|A^\theta_\nu G\|_{L^2_{t,x}}^2\lesssim \|G\|_{L^2_{t,x}}^2,
\end{equation}
with constants independent of $\theta\in [\la^{-\delta_0},1]$, with $\delta_0<1/2$ as above.
The second estimate \eqref{m12}, is standard since the $A_\nu^\theta$ are in $S^{0}_{1-\delta_0,\delta_0}$ 
and \eqref{sep0} is valid.  The other estimate \eqref{m12} follows from the fact, that by \eqref{m1} and \eqref{ups},
$Q(x,D)=I-\sum_\nu A_\nu^\theta\in S^0_{1-\delta_0,\delta_0}$ has symbol supported outside of a neighborhood
of $\text{supp }B(x,\xi)$, if, as we may, we assume that the latter is small, and this leads to \eqref{m11} by
the proof of Lemma~\ref{comprop} below if $\delta$ in \eqref{22.6} is small enough.
Also, for each $x$ the symbols vanish outside of cubes of sidelength $\theta \la$ and 
$| \partial^\gamma_\xi A^\theta_\nu(x,\xi)|=O((\la\theta)^{-|\gamma|})$, we also have that their kernels
are $O((\theta\la)^{n-1}(1+\theta\la d_g(x,y))^{-N})$ for all $N$ and so
\begin{equation}\label{mp}
\|A_\nu^\theta\|_{L^p(M)\to L^p(M)}=O(1) \quad \forall \, 1\le p\le \infty.
\end{equation}

In view of \eqref{m11} we have for $\theta_0=\la^{-1/8}$
\begin{equation}\label{m13}
\bigl( \alpha_m(t)\tilde\sigma_\la H\bigr)^2
=\sum_{\nu,\tilde \nu}
\bigl( \alpha_m(t)\tilde \sigma_\la A^{\theta_0}_\nu H
\bigr)\cdot \bigl( \alpha_m(t)\tilde \sigma_\la A^{\theta_0}_{\tilde\nu} H\bigr)
+O(\la^{-N}\|H\|_2^2),
\end{equation}
for $\alpha_m$ as in \eqref{22.17}.
Recall that in $A_\nu^{\theta_0}$, $\nu \in \theta{\mathbb Z}^{2(n-2)+1}$ indexes a $\la^{-1/8}$-separated set in
${\mathbb R}^{2n-3}$.

We need to organize the pairs of indices $\nu,\tilde \nu$ in \eqref{m13} as in many earlier works (see \cite{LeeBilinear} and 
\cite{TaoVargasVega}).
To this end, consider dyadic cubes,
$\tau^\theta_\mu$ in ${\mathbb R}^{2n-3}$ of sidelength $\theta=2^k\theta_0=2^k\la^{-1/8}$, with $\tau^\theta_\mu$ denoting translations of the cube
$[0,\theta)^{2n-3}$ by $\mu\in \theta{\mathbb Z}^{2n-3}$.
Two such dyadic cubes of sidelength $\theta$ are said
to be {\em close}  if they are not adjacent but have
adjacent parents of length $2\theta$, and, in this case,
we write $\tau^\theta_\mu \sim \tau^\theta_{\tilde \mu}$.
We note that close cubes satisfy $\text{dist}(\tau^\theta_\mu, \tau^\theta_{\tilde \mu})\approx \theta$, and so each fixed cube has $O(1)$ cubes which are ``close'' to it.  Moreover, as noted in 
\cite[p. 971]{TaoVargasVega}, any distinct points $\nu,
\tilde \nu\in {\mathbb R}^{2n-3}$ must like in
a unique pair of close cubes in this Whitney decomposition.  So, there must be a unique triple
$(\theta=\theta_02^k, \mu, \tilde \mu)$ such that
$(\nu,\tilde \nu)\in \tau^\theta_\mu\times \tau^\theta_{\tilde \mu}$ and $\tau^\theta_\mu\sim
\tau^\theta_{\tilde \mu}$.  We remark that by choosing
$B$ to have small support we need only consider 
$\theta=2^k\theta_0\ll 1$.

Taking these observations into account implies
that the bilinear sum \eqref{m13} can be 
organized as follows:
\begin{multline}\label{m14}
\sum_{\{k\in {\mathbb N}: \, k\ge 10 \, \, \text{and } \, 
\theta=2^k\theta_0\ll 1\}}
%\{\theta=2^k\theta_0: \,
%k\in {\mathbb N}, \, 2^k\theta_0\le 1\}}
\sum_{\{(\mu,\tilde \mu): \, \tau^\theta_\mu
\sim \tau^\theta_{\tilde \mu}\}}
\sum_{\{(\nu,\tilde \nu)\in
\tau^\theta_\mu\times \tau^\theta_{\tilde \mu}\}}
\bigl(\alpha_m(t)\tilde \sigma_\la
A^{\theta_0}_\nu H\bigr) 
\cdot \bigl(\alpha_m(t)\tilde \sigma_\la
A^{\theta_0}_{\tilde \nu} H\bigr)
\\
+\sum_{(\tau,\tilde \tau)\in \Xi_{\theta_0}} 
\bigl(\alpha_m(t) \tilde \sigma_\la A^{\theta_0}_\nu H\bigr) 
\cdot \bigl(\alpha_m(t)\tilde \sigma_\la
A^{\theta_0}_{\tilde \nu} 
H\bigr)
,
\end{multline}
where $\Xi_{\theta_0}$ indexes the remaining pairs such
that $|\nu-\tilde \nu|\lesssim \theta_0=\la^{-1/8}$,
including the diagonal ones where $\nu=\tilde \nu$.

The key  estimate that we require, which follows from bilinear harmonic analysis arguments, then is the following.

\begin{proposition}\label{locprop}
If $H=S_\la f$ is as in \eqref{22.2}
then for $m\in {\mathbb Z}$ we have the uniform bounds
\begin{multline}\label{b1}
\|\alpha_m(t)\tilde \sigma_\la H\|_{L^{q_c}(A_-)}
\\
\lesssim \Bigl(\, \sum_\nu \bigl\|
\alpha_m(t) \tilde \sigma_\la A^{\theta_0}_\nu H
\bigr\|_{L^{q_c}_{t,x}(M^{n-1}\times \R)}^{q_c}\, \Bigr)^{1/q_c}
+\la^{\frac1{q_c}-}\|H\|_{L^2_{t,x}(M^{n-1}\times \R)}.
\end{multline}
\end{proposition}

The $\la^{\frac1{q_c}-}$  notation that we are using for the last term in \eqref{b1} denotes $\la^{\frac1{q_c}-\e_0}$ for some
unspecified $\e_0>0$.  Note that since $\|H\|_{L^2_{t,x}}\approx T^{1/2}$ for $H=S_\la f$ and $T\approx \log\la$ the log-loss afforded by having the last term involve this norm is more than overset by the power gain $1/q_c-$ of $\lambda$.  Similarly, when we sum over $m$ and use this estimate, the additional log--loss will be more than compensated by this gain.

We shall postpone the proof of Proposition~\ref{locprop} until the next section.  Let us now
see how we can use it to prove Proposition~\ref{smallprop}.

We first note that if $\alpha_m(t)=\alpha (t-m)$ is as in \eqref{22.17} with $\alpha$ as in \eqref{22.16}, we of course have
$$\| \tilde S_\la f\|^{q_c}_{L^{q_c}(A_-)}\lesssim \sum _m \|\alpha_m(t) \tilde S_\la f\|^{q_c}_{L^{q_c}(A_-)}.$$
Recall that $\tilde S_\la  = \tilde \sigma_\la S_\la$.  Therefore, by \eqref{b1} and \eqref{normalize} we have with
$\theta_0=\la^{-1/8}$
$$\|\tilde S_\la f\|^{q_c}_{L^{q_c}(A_-)} \lesssim  \sum_m\sum_\nu \|\alpha_m(t)\tilde \sigma_\la A^{\theta_0}_\nu
S_\la f\|^{q_c}_{L^{q_c}(A_-)} +\la^{1-}\|S_\la f\|^{q_c}_{L^{q_c}(M^{n-1}\times \R)}.$$
Since the last term is $O(\la^{1-}\log\la)$, in order to prove Proposition~\ref{smallprop}, it suffices to show that when
$M^{n-1}$ has nonpositive curvature
\begin{equation}\label{22.53}
\sum_m \sum_\nu \|\alpha_m(t)\tilde \sigma_\la A^{\theta_0}_\nu
S_\la f\|^{q_c}_{L^{q_c}(M^{n-1}\times [0,T])} \le C \la T^{\frac2{q_c}},
\end{equation}
with, as in the Proposition, $T=c_0\log\la$ for $c_0>0$ sufficiently small, and we obtain the other estimate, \eqref{22.27}, from
\begin{equation}\label{22.54}
\sum_m\sum_\nu 
\|\alpha_m(t)
\tilde \sigma_\la A^{\theta_0}_\nu
S_\la f\|^{q_c}_{L^{q_c}(M^{n-1}\times [0,T])} \le C \la T^{\frac{4-q_c}{2}}.
\end{equation}

If we use Lemma~\ref{qlemma} along with \eqref{22.10} and \eqref{m12} we obtain the following uniform bounds for each fixed $m$
\begin{align}\label{22.55}
\sum_\nu \|\alpha_m(t)
\tilde \sigma_\la A^{\theta_0}_\nu
S_\la &f\|^2_{L^{q_c}(M^{n-1}\times [0,T])}
\\
&\lesssim \sum_\nu
\|\alpha_m(t)
 \sigma_\la A^{\theta_0}_\nu
S_\la f\|^2_{L^{q_c}(M^{n-1}\times [0,T])} \notag
\\
&\lesssim \la^{\frac2{q_c}}
\sum_\nu \|A_\nu^{\theta_0} S_\la f\|^2_{L^2(M^{n-1}\times [m-10,m+10])} +O(\la^{-N}) \notag
\\
&\lesssim  \la^{\frac2{q_c}} \|S_\la f\|_{L^2(M^{n-1}\times [m-10,m+10])}  +O(\la^{-N}) \notag
\\
&\lesssim \la^{\frac2{q_c}}.  \notag
\end{align}
Here, we again used the trivial bound $\|S_\la f\|_{L^2(M^{n-1}\times I)} \lesssim  |I|^{1/2}$ if $I\subset \R$ is an interval.

To use this, for each $m$ choose $\nu(m)$ such that
\begin{equation}\label{22.56}
\max_\nu \|\alpha_m(t)
\tilde \sigma_\la A^{\theta_0}_\nu
S_\la f\|_{L^{q_c}(M^{n-1}\times [0,T])}=
\|\alpha_m(t)
\tilde \sigma_\la A^{\theta_0}_{\nu(m)}
S_\la f\|_{L^{q_c}(M^{n-1}\times [0,T])}.
\end{equation}
Then, by \eqref{22.55} we have
\begin{align} \label{22.57}
&\sum_m 
\sum_\nu \|\alpha_m(t)
\tilde \sigma_\la A^{\theta_0}_\nu
S_\la f\|^{q_c}_{L^{q_c}(M^{n-1}\times \R)}
\\
&\le \sum_m
\bigl(\, \sum_\nu \|\alpha_m(t)
\tilde \sigma_\la A^{\theta_0}_\nu
S_\la f\|^2_{L^{q_c}(M^{n-1}\times [0,T])} 
\bigr)
\cdot
\|\alpha_m(t)
\tilde \sigma_\la A^{\theta_0}_{\nu(m)}
S_\la f\|^{q_c-2}_{L^{q_c}(M^{n-1}\times [0,T])} \notag
\\
& \lesssim  \la^{\frac2{q_c}}
\sum_m
\|\alpha_m(t)
\tilde \sigma_\la A^{\theta_0}_{\nu(m)}
S_\la f\|^{q_c-2}_{L^{q_c}(M^{n-1}\times [0,T])}.
\notag
\end{align}

Since there are $O(T)$ nonzero terms in the last sum, by H\"older's inequality we have
$$\sum_m
\|\alpha_m(t)
\tilde \sigma_\la A^{\theta_0}_{\nu(m)}
S_\la f\|^{q_c-2}_{L^{q_c}(M^{n-1}\times [0,T])} \lesssim T^{\frac2{q_c}}
\|\alpha_m(t)
\tilde \sigma_\la A^{\theta_0}_{\nu(m)}
S_\la f\|^{q_c-2}_{\ell_m^{q_c} L^{q_c}(M^{n-1}\times [0,T])},
$$
and, as $q_c\le 4$,
$$\sum_m
\|\alpha_m(t)
\tilde \sigma_\la A^{\theta_0}_{\nu(m)}
S_\la f\|^{q_c-2}_{L^{q_c}(M^{n-1}\times [0,T])} \lesssim T^{\frac{4-q_c}2}
\|\alpha_m(t)
\tilde \sigma_\la A^{\theta_0}_{\nu(m)}
S_\la f\|^{q_c-2}_{\ell_m^2 L^{q_c}(M^{n-1}\times [0,T])}.
$$
Therefore, by \eqref{22.55}, we would have \eqref{22.53} if we could show that when all the sectional curvatures
of $M^{n-1}$ are nonpositive then for $T=c_0 \log\la$ with $c_0>0$ small enough
\begin{equation}\label{22.58}
\| \alpha_m(t)B\sigma_\la A^{\theta_0}_{\nu(m)} S_\la f\|_{\ell^{q_c}_mL^{q_c}_{t,x}(M^{n-1}\times \R)} \lesssim \la^{\frac1{q_c}},
\end{equation}
and we would have \eqref{22.54} if we could show that when all of the sectional curvatures are negative and $T$ is as above
\begin{equation}\label{22.59}
\| \alpha_m(t)B\sigma_\la A^{\theta_0}_{\nu(m)} S_\la f\|_{\ell^{2}_mL^{q_c}_{t,x}(M^{n-1}\times \R)} \lesssim \la^{\frac1{q_c}},
\end{equation}
since $\tilde \sigma_\la =B\sigma_\la$.

To prove these inequalities we  shall make use of the following simple lemma whose proof we postpone until the 
end of this subsection.

\begin{lemma}\label{comprop}  
%Assume that
%\eqref{cc1} and \eqref{cc2} are valid.  
If $\delta>0$ in \eqref{22.6} is small enough
and
$\theta_0 =\la^{-1/8}$ we have
for %$B=B_\la=B_{j,\la}$ 
$B$ as in \eqref{22.9}
\begin{equation}\label{cc3}
\bigl\| \, 
B \sigma_\la A^{\theta_0}_\nu -B A^{\theta_0}_\nu
\sigma_\la \, \bigr\|_{L^2_{t,x}\to L^{q_c}_{t,x}}
=O(\la^{\frac1{q_c}-\frac14}).
\end{equation}
\end{lemma}

If we use \eqref{cc3} followed by the use of \eqref{22.10} and \eqref{mp}, we see that for
each $m$ we  have
\begin{align}\label{22.60}
\|\alpha_m(t) &B \sigma_\la   A_{\nu(m)}^{\theta_0} S_\la f\|_{L^{q_c}_{t,x}}
\\
&\lesssim \| \alpha_m(t) B A_{\nu(m)}^{\theta_0} \sigma_\la  S_\la f\|_{L^{q_c}_{t,x}}  +\la^{\frac1{q_c}-\frac14} \|S_\la f \|_{L^2_{t,x}} \notag
\\
&\lesssim 
\|\alpha_m(t) BA_{\nu(m)}^{\theta_0} S_\la f\|_{L^{q_c}_{t,x}}
+ \| \alpha_m(t) B A_{\nu(m)}^{\theta_0} (I-\sigma_\la)S_\la f\|_{L^{q_c}_{t,x}} +\la^{\frac1{q_c}-\frac14} (\log\la )^{1/2} \notag
\\
&\lesssim 
\|\alpha_m(t) A_{\nu(m)}^{\theta_0} S_\la f\|_{L^{q_c}_{t,x}}
+\| \alpha_m(t) ( I-\sigma_\la) S_\la f\|_{L^{q_c}_{t,x}} +\la^{\frac1{q_c}-\frac14} (\log\la)^{1/2}. \notag
\end{align}

By \eqref{22.16}--\eqref{22.17} and Lemma~\ref{lemmadiff} we have
\begin{equation}\label{22.61}
\| \alpha_m(t) (I-\sigma_\la)S_\la f\|_{\ell^{q_c}_m L^{q_c}_{t,x}}\le \la^{\frac1{q_c}} \, T^{\frac1{q_c}-\frac12},
\end{equation}
and so, by \eqref{22.60} we would have \eqref{22.58} if
\begin{equation}\label{22.62}
\|\alpha_m(t) A_{\nu(m)}^{\theta_0}S_\la f\|_{\ell_m^{q_c}L^{q_c}_{t,x}(M^{n-1}\times \R)}\lesssim \la^{\frac1{q_c}}.
\end{equation}
Also, by H\'older's inequality in $m$ and \eqref{22.61} we have
$$
\|\alpha_m(t)(I-\sigma_\la)S_\la f\|_{\ell^2_m L^{q_c}_{t,x}(M^{n-1}\times \R)}
\lesssim T^{\frac{q_c-2}{2q_c}} \|\alpha_m(t)(I-\sigma_\la)S_\la f\|_{\ell^{q_c}_m L^{q_c}_{t,x}(M^{n-1}\times \R)}
\lesssim \la^{\frac1{q_c}},$$
which, by \eqref{22.60} means that we would also have \eqref{22.59} if when all the sectional curvatures
of $M^{n-1}$ are negative
\begin{equation}\label{22.63}
\|\alpha_m(t) A_{\nu(m)}^{\theta_0}S_\la f\|_{\ell_m^{2}L^{q_c}_{t,x}(M^{n-1}\times \R)}\lesssim \la^{\frac1{q_c}}.
\end{equation}

In both \eqref{22.62} and \eqref{22.63} we are considering the map
$$f\to \bigl(Wf \bigr)(x,t,m)=\eta(t/T) \alpha_m(t) \bigl(A_{\nu(m)}^{\theta_0}\circ e^{-it\la^{-1}\Delta_g}f\bigr)(x).
$$
By repeating the standard $TT^*$ argument that was used in the proof of Proposition~\ref{largeprop}, we 
would have \eqref{22.62} if
\begin{equation}\label{22.62'}\tag{2.65$'$}
\| WW^*G\|_{\ell_m^{q_c}L^{q_c}_{t,x}(M^{n-1}\times \R)}
\le C\la^{\frac2{q_c}} \|G\|_{\ell^{q_c'}_{m'} L^{q_c'}_{t,x}(M^{n-1}\times \R)},
\end{equation}
and \eqref{22.63} if
\begin{equation}\label{22.63'}\tag{2.66$'$}
\| WW^*G\|_{\ell_m^2 L^{q_c}_{t,x}(M^{n-1}\times \R)}
\le C\la^{\frac2{q_c}} \|G\|_{\ell^{2}_{m'} L^{q_c'}_{t,x}(M^{n-1}\times \R)},
\end{equation}
with
\begin{align}\label{22.64}
&WW^*G(x,t,m)=
\\
&=\alpha_m(t) \eta(t/T) \sum_{m'} \int_{-\infty}^\infty
\alpha_{m'}(s) \eta(s/T) \Bigl[ \bigr( A_{\nu(m)}^{\theta_0} e^{-i(t-s)\la^{-1}\Delta_g} (A^{\theta_0}_{\nu(m')})^* \bigr)G(\, \cdot \, , s,m')\Bigr] (x) \, ds \notag
\\
&=\sum_{m'} \iint K(x,t,m;y,s,m') \, G(y,s,m') \, dyds, \notag
\end{align}
with
\begin{equation}\label{22.65}
K(x,t,m; y,s,m') = \alpha_m(t) \eta(t/T) \bigr( A_{\nu(m)}^{\theta_0} e^{-i(t-s)\la^{-1}\Delta_g} (A^{\theta_0}_{\nu(m')})^* \bigr)(x,y) \, \alpha_{m'}(s)\eta(s/T).
\end{equation}

In \S 4 we shall show (see Proposition \ref{mick}) that for $T=c_0\log \la$
small enough we have for $M^{n-1}$ of nonpositive curvature
\begin{equation}\label{22.65np}
|K(x,t,m;y,s,m')|\le C\la^{\frac{n-1}2} \, |t-s|^{-\frac{n-1}2},
\end{equation}
and, moreover, if all of the sectional curvatures of $M^{n-1}$ are negative
\begin{equation}\label{22.65neg}
|K(x,t,m;y,s,m')|\le C\la^{\frac{n-1}2} \,|t-s|^{-N} \quad \text{if } \, \, |t-s|\ge1.
\end{equation}
As we shall see, it is for these two estimates that we need to assume that $c_0$ is small enough depending
on $(M^{n-1},g)$.
Also, by the support properties of $\alpha$ in \eqref{22.16} we also have
\begin{equation}\label{22.67}
K(x,t,m; y, s, m')=0 \quad \text{if } \, |t-m| \ge 3 \, \, \text{or } \, \, |s-m'| \ge 3.
\end{equation}
Thus, if we define the frozen operators
$$\bigl(W_{t,m;s,m'}h\bigr)(x) =\int_{M^{n-1}} K(x,t,m; y, s, m') \, h(y) \, dy,$$
we have
\begin{equation}\label{22.68}
W_{t,m;s,m'}\equiv 0 \quad \text{if } \, |t-m|\ge 3 \, \, \text{or } \, \, |s-m'|\ge 3,
\end{equation}
and, if $M^{n-1}$ has nonpositve curvature, by \eqref{22.65np},
\begin{equation}\label{22.69}
\|W_{t,m;s,m'}h\|_{L^\infty_x(M^{n-1})}\le C\la^{\frac{n-1}2} \, |t-s|^{-\frac{n-1}2} \|h\|_{L^1(M^{n-1})},
\end{equation}
and, moreover, if the sectional curvatures of $M^{n-1}$ are negative, by \eqref{22.65neg} and \eqref{22.67},
\begin{equation}\label{22.70}
\|W_{t,m;s,m'} h\|_{L^\infty_x(M^{n-1})}\lesssim
\begin{cases}
\la^{\frac{n-1}2} |t-s|^{-\frac{n-1}2} \|h\|_{L^1(M^{n-1})} \, \, \text{if } \, |m-m'|\le 10
\\
\la^{\frac{n-1}2} |m-m'|^{-N} \|h\|_{L^1(M^{n-1})} \, \, \forall \, N \, \, \text{if } \, |m-m'| > 10.
\end{cases}
\end{equation}
Also, by \eqref{mp} and the fact that $e^{-it\la^{-1}\Delta_g}$ is unitary, we of course always have
\begin{equation}\label{22.71}
\| W_{t,m;s,m'} \|_{L^2(M^{n-1})\to L^2(M^{n-1})}=O(1).
\end{equation}

By interpolation \eqref{22.69}, \eqref{22.71} along with \eqref{22.68} yield that if $M^{n-1}$ has nonpositive curvature
\begin{equation}\label{22.72}
\|W_{t,m;s,m'} h\|_{L^{q_c}_x(M^{n-1})}
=
\begin{cases}
O(\la^{\frac2{q_c}}|t-s|^{-\frac2{q_c}}\|h\|_{L^{q_c'}(M^{n-1})}) \, \,
\text{if } \, |t-m| \le 3 \, \, \text{and } \,|s-m'|\le 3
\\
0 \quad \text{if } \, \, |t-m| >3 \, \, \text{or } \, \,|s-m'| >3.
\end{cases}
\end{equation}
while if we also use \eqref{22.70} then this argument implies that if the sectional
curvatures of $M^{n-1}$ are all negative
\begin{multline}\label{22.73}
\|W_{t,m;s,m'} h\|_{L^{q_c}_x(M^{n-1})}
=
\\
\begin{cases}
O(\la^{\frac2{q_c}}|t-s|^{-\frac2{q_c}}\|h\|_{L^{q_c'}(M^{n-1})}) \, \,
\text{if }|m-m'|\le 10, \, |t-m|\le 3 \, \, \text{and }|s-m'| \le 3
\\
O(\la^{\frac2{q_c}}|m-m'|^{-2}  \|h\|_{L^{q_c'}(M^{n-1})}) \, \, 
\text{if } \,|m-m'|>10, \, |t-m|\le 3 \, \text{and } \,|s-m'|\le 3
\\
0 \quad \text{if } \, \, |t-m|>3 \, \, \text{or } \, \,|s-m'| >3.
\end{cases}
\end{multline}

Note that for fixed $t,m$ we have by Minkowski's inequality and \eqref{22.64}
\begin{align}\label{22.74}
\| WW^*G(\, \cdot \, ,t,m)\|_{L^{q_c}_x} &\le \sum_{m'} \int \, \Bigl\| \, \int K(x,t,m;y,s,m') \, G(y,s,m') \, dy \, \Bigr\|_{L^{q_c}_x} \, ds
\\
&=\sum_{m'} \int \, \bigl\| (W_{t,m;s,m'}G(\, \cdot \, ,s,m))(x) \bigr\|_{L^{q_c}_x} \, ds.
\notag
\end{align}

Set
\begin{equation}\label{22.75}
H(t,m;s,m') =
\begin{cases}
\la^{\frac2{q_c}} |t-s|^{-\frac2{q_c}}, \, \, \, \text{if } \, |t-m|\le 3 \, \text{and } \, |s-m'|\le 3
\\
0 \quad \text{if } \, \, |t-m|>3 \, \, \text{or } \, \, |s-m'|>3.
\end{cases}
\end{equation}
Then, by \eqref{22.72} and \eqref{22.74} we have
\begin{align*}
\|WW^*G\|_{\ell_m^{q_c}L^{q_c}_{t,x}} &\lesssim
\Bigl( \sum_m \int \bigl| \, \sum_{m'} \int H(t,m;s,m') \, \|G(\, \cdot \, ,s,m') \|_{L^{q_c'}_x} \, ds \, \bigr|^{q_c} \, dt \, \Bigr)^{1/q_c}
\\
&\lesssim \la^{\frac2{q_c}} \Bigl( \sum_{m'} \int \|G(\, \cdot \, ,s,m')\|^{q_c'}_{L^{q_c'}_x} \, ds \Bigr)^{1/q_c'}
\\
&=\la^{\frac2{q_c}} \|G\|_{\ell^{q_c'}_{m'} L^{q_c'}_{t,x}(M^{n-1}\times \R)},
\end{align*}
since  if
\begin{equation}\label{22.76}
Uf(t,m) =\sum_{m'} \int H(t,m;s,m') \, f(s,m')\, ds,
\end{equation}
we have
$$\|U\|_{\ell^{q_c'}_{m'}L^{q_c'}_s \to \ell^{q_c}_mL^{q_c}_t} =O(\la^{\frac2{q_c}})$$
by a simple variant of Theorem 0.3.6 in \cite{SFIO2}.  Thus, we have obtained
\eqref{22.62'}.

If all the sectional curvatures of $M^{n-1}$ are negative and we set
$$
H(t,m;s,m')=
\begin{cases}
\la^{\frac2{q_c}} |t-s|^{-\frac2{q_c}} \, \, 
\text{if } \,|m-m'|\le 10, \, |t-m|\le 3 \, \, \text{and } \,|s-m'|\le 3
\\
\la^{\frac2{q_c}}|m-m'|^{-2} \, \, 
\text{if } \,|m-m'| > 10, \, |t-m|\le 3 \, \, \text{and } \,|s-m'|\le 3
\\
0 \quad \text{if } \, \, |t-m|>3 \, \, \text{or } \, \,|s-m'|>3,
\end{cases}
$$
and, if $U$ is as in \eqref{22.76}, then the proof of Theorem 0.3.6 in \cite{SFIO2} yields
$$\|U\|_{\ell^2_{m'}L^{q_c'}_s\to \ell^2_mL^{q_c}_t}=O(\la^{\frac2{q_c}}),$$
which yields \eqref{22.63'} by the above argument. \qed

This completes the proof of Proposition~\ref{smallprop} and hence Theorems \ref{nonposthm} and \ref{negthm} up to proving
the crucial local estimates in Proposition~\ref{locprop}, as well as the global kernel estimates \eqref{k},  \eqref{22.65np} and \eqref{22.65neg}
and  that we have used.  We shall prove the former using bilinear harmonic analysis techniques in the next section and the kernel estimates
in the final section.

The other task remaining to complete the proofs Theorems \ref{nonposthm} and \ref{negthm}  is to 
prove the  commutator estimate that we employed:

\begin{proof}[Proof of Lemma~\ref{comprop}]
Recall that 
by \eqref{m7} and \eqref{ups} the symbol $B(x,\xi)=B_\la(x,\xi)
\in S^{0}_{1,0}$ vanishes when $|\xi|$ is not
comparable to $\la$.  In particular, it vanishes if $|\xi|$ is larger than a fixed
multiple of $\la$, and it belongs to a bounded subset of $S^0_{1,0}$.
Furthermore, if $a^{\theta_0}_\nu(x,\xi)$ is the principal
symbol of our zero-order dyadic microlocal operators, we recall  that by \eqref{invariant}
we have that for $\delta>0$ small enough
\begin{equation}\label{cc2}
a^{\theta_0}_\nu(x,\xi)=a^{\theta_0}_\nu(\chi_r(x,\xi))
\quad \text{on supp } \, B_\la \, \, \, 
\text{if } \, \, |r|\le 2\delta,
\end{equation}
where $\chi_r: \, T^*M^{n-1}\, \backslash 0 \to T^*M^{n-1} \, \backslash 0$ denotes
geodesic flow in the cotangent bundle.

By Sobolev estimates for $M^{n-1}\times \R$, in order
to prove \eqref{cc3}, it suffices to show that
\begin{equation}\label{cc4}
\Bigl\|\,
\Bigl(\sqrt{I+P^2+D_t^2} \, \, \Bigr)^{n(\frac12-\frac1{q_c})}
\,
\bigl[ B_\la \sigma_\la A^{\theta_0}_\nu
-B_\la A^{\theta_0}_\nu \sigma_\la \bigr]
\, 
\Bigr\|_{L^2_{t,x}\to L^2_{t,x}}
=
O(\la^{\frac1{q_c}-\frac14}).
\end{equation}

To prove this we recall that
$$\sigma_\la=(2\pi)^{-1}\tilde \beta(D_t/\la)
\int \Hat \sigma(r)
e^{ir\la^{1/2}|D_t|^{1/2}} \, e^{-irP} \, dr,
$$
and, therefore, since $e^{ir\la^{1/2}|D_t|^{1/2}}$
has $L^2\to L^2$ norm one and commutes with $B_\la$,
$A^{\theta_0}_\nu$ and $(\sqrt{I+P^2+D_t^2})^{n(\frac12-\frac1{q_c})}$, and since
$\Hat \sigma(r)=0$, $|r|\ge 2\delta$, by
Minkowski's integral inequality,
we would have
\eqref{cc4} if 
\begin{multline}\label{cc5}
\sup_{|r|\le 2\delta}\, 
\Bigl\|\,
\Bigl(\sqrt{I+P^2+D_t^2}  \, \, \Bigr)^{n(\frac12-\frac1{q_c})}
\, \tilde \beta(D_t/\la) \, 
\bigl[ B_\la e^{-irP} A^{\theta_0}_\nu
-B_\la A^{\theta_0}_\nu e^{-irP} \bigr]
\, 
\Bigr\|_{L^2_{t,x}\to L^2_{t,x}}
\\
=
O(\la^{\frac1{q_c}-\frac14}).
\end{multline}

Next, to be able to use Egorov's theorem, we write
$$\bigl[ B_\la e^{-irP} A^{\theta_0}_\nu
-B_\la A^{\theta_0}_\nu e^{-irP} \bigr]
=B_\la
\, \bigl[
(e^{-irP}A^{\theta_0}_\nu e^{irP}) - B_\la A^{\theta_0}_\nu]
\circ e^{-irP}.
$$
Since $e^{-irP}$ also has $L^2$-operator norm one, we
would obtain \eqref{cc5} from
\begin{multline}\label{cc6}
\Bigl\|\,
\Bigl(\sqrt{I+P^2+D_t^2} \, \, \Bigr)^{n(\frac12-\frac1{q_c})}
\, \tilde \beta(D_t/\la) \, 
B_\la
\, \bigl[
(e^{-irP}A^{\theta_0}_\nu e^{irP}) - A^{\theta_0}_\nu\bigr]
\, 
\Bigr\|_{L^2_{t,x}\to L^2_{t,x}} 
\\
=
O(\la^{\frac1{q_c}-\frac14}).
\end{multline}

By Egorov's theorem (see e.g. Taylor~\cite[\S VIII.1]{TaylorPDO})
$$A^{\theta_0}_{\nu,r}(x,D)= e^{-irP}A^{\theta_0}_\nu e^{irP}$$
is a one-parameter family of 
zero-order  pseudo-differential operators, depending on the parameter $r$, whose principal symbol is 
$a^{\theta_0}_\nu(\chi_{-r}(x,\xi))$.  By \eqref{cc2}
 and the composition calculus
of pseudo-differential operators the principal
symbol of $B_\la A_{\nu,r}^{\theta_0}$ and $B_\la A^{\theta_0}_\nu$ both
equal $B_\la(x,\xi)a^{\theta_0}_\nu(x,\xi)$ if $|r|\le 2\delta$.  If $\theta=1$ then $A^\theta_\nu
\in S^0_{1,0}$, and, so, in this  case
we would have that $B_\la (e^{-irP}A^{\theta}_\nu e^{irP})
-B_\la A^\theta_\nu$ would be a pseudo-differential
operator of order $-1$ with symbol vanishing
for $|\xi|$ larger than a fixed multiple of $\la$ (see e.g., \cite[Theorem 4.3.6]{SoggeHangzhou}).
Since we are assuming that $\theta_0=\la^{-1/8}$, by the way they were constructed, the symbols $A^{\theta_0}_\nu$
belong to a bounded subset of $S^0_{7/8,1/8}$.  So, by \cite[p. 147]{TaylorPDO}, for $|r|\le 2\delta$,
$B_\la (e^{-irP}A^{\theta_0}_\nu e^{irP})
-B_\la A^{\theta_0}_\nu$ belong to a bounded subset of
$S^{-3/4}_{7/8,1/8}$ with symbols vanishing
for $|\xi|$ larger than a fixed multiple of
$\la$ due to the fact that the symbol  $B_\la(x,\xi)$ has this property
(see e.g., \cite[p. 46]{TaylorPDO}).

We also need to take into account the other operators
inside the norm in \eqref{cc6}.  Since
$\tilde \beta(D_t/\la)$ is a zero-order dyadic operator,
by the above, the operators in the left of \eqref{cc6} belong to 
a bounded subset of 
$S^{n(\frac12-\frac1{q_c})-\frac34}_{7/8,1/8}
(M^{n-1}\times \R)$ with symbols vanishing
for $|(\xi,\tau)|$ larger than a fixed multiple of
$\la$.  Consequently, the left side of
\eqref{cc6} is $O(\la^{n(\frac12-\frac1{q_c})-\frac34})
=O(\la^{\frac1{q_c}-\frac14})$.  For,
$q_c=\tfrac{2(n+1)}{n-1}$ and so
$\tfrac1{q_c}=n(\tfrac12-\tfrac1{q_c})-\tfrac12$.
\end{proof}

\noindent{\bf 2.3.  Endpoint Strichartz estimates: Proof of Theorem~\ref{endthm}.}

We now prove our final theorem saying that if all the 
sectional curvatures of $M^{n-1}$ are nonpositive
and, as is necessary, $d=n-1\ge3$ we have the endpoint
Strichartz estimates \eqref{00.15}.  As we pointed
out before, such improvements cannot hold on spheres
$S^d$  since the estimates are
saturated just by taking the initial data in
\eqref{00.1} to be zonal eigenfunctions.

To prove our improvements under our geometric assumptions
we shall use the universal local estimates of
Burq, G\'erard and Tzvetkov~\cite{bgtmanifold} along
with our improvements in Theorem \ref{nonposthm} for
non-endpoint exponents, some of the kernel estimates
we have used and an argument of one of us
\cite{sogge2015improved} that is a variation of an
earlier one of Bourgain~\cite{BourgainBesicovitch}.

To this end, we recall the universal endpoint Strichartz
estimates of Burq, G\'erard and Tzvetkov, which say
that for $\la\gg 1$ one has the uniform 
dyadic small interval bounds
\begin{equation*}
\|e^{-it\Delta_g}\beta(P/\la)f\|_{L^2_tL^{q_e}_x(
M^{n-1}\times [0,\la^{-1}])} \le C\|f\|_2,
\, \,  
\text{if } \, \, q_e=\tfrac{2d}{d-2}=\tfrac{2(n-1)}{n-3},
\, \, \, d=n-1\ge 3.
\end{equation*}
This is of course equivalent to the following estimates
for the scaled Schr\"odinger operators
\begin{equation}\label{2.85}
\|e^{-it\la^{-1}\Delta_g}
\beta(P/\la)f\|_{L^2_tL^{q_e}_x(M^{n-1}\times [0,1])}
\le C\la^{\frac12} \|f\|_2, 
\, \, q_e=\tfrac{2(n-1)}{n-3}, \, \, n\ge4.
\end{equation}

We also point out that by using the Littlewood-Paley arguments  described in the introduction we would 
obtain the bound \eqref{00.15} in Theorem~\ref{endthm}
by showing that whenever all the sectional 
curvatures of $M^{n-1}$ are nonpositive we have
for $q_e$ and $n$ as above
\begin{equation}\label{endest}
\bigl\|
e^{-it\la^{-1}\Delta_g}\beta(P/\la)f
\bigr\|_{L^2_tL^{q_e}_x(M^{n-1}\times 
[0,\log\la])}
\le C \la^{\frac12}
\, (\log\la)^{\frac12}\,
(\log(\log\la))^{-\frac12} \|f\|_2.
\end{equation}

In order to use  our earlier arguments, it turns
out that we need to modify the height splitting
\eqref{22.24} as follows
\begin{multline}\label{22.86}
A_+=\{(x,t)\in M^{n-1}\times [0,\log\la]:
\, |U_\la f(x,t)|\ge \la^{\frac{n-1}4}(\log\la)^{\e_0}
\},
\\
\text{and } \,
A_-=\{(x,t)\in M^{n-1}\times [0,\log\la]:
\, |U_\la f(x,t)|< \la^{\frac{n-1}4}(\log\la)^{\e_0}
\},
\end{multline}
assuming, as we are that $\|f\|_2=1$,
for $\e_0>0$ to be specified in just a moment and
$$U_\la f=e^{-it\la^{-1}\Delta_g}\beta(P/\la)f.$$

Let us now see how we can adapt the proof of
Proposition~\ref{largeprop} to obtain the following.

\begin{proposition}\label{endprop}  Suppose that
all the curvatures of $M^{n-1}$ are nonpositive and
let $\e_0>0$ be fixed and $A_+$ be as in
\eqref{22.86}.  Then, if, as before 
$\|f\|_2=1$ and $\la\gg 1$ we have the following uniform
bounds
\begin{equation}\label{ubound}
\|U_\la f\|_{L^2_tL^{q_e}_x(A_+\cap (M^{n-1}\times I_T))}
\le C\la^{\frac12},
\end{equation}
if $I_T\subset [0,\log\la]$ is an interval of length
$|I_T|\le T$ where
$$T=c_0\log(\log\la),$$
with $c_0>0$ sufficiently small (depending on
$\e_0>0$ and $M^{n-1}$).
\end{proposition}

\begin{proof}  If $I_T$ is as above choose $g$ so that
\begin{multline}\label{22.87}
\|g\|_{L^2_tL^{q_e'}_x(A_+\cap (M^{n-1}\times I_T))}
=1, \, \, \text{and }
\\
\|U_\la f\|_{L^2_tL^{q_e}_x(A_+\cap (M^{n-1}\times
I_T))}
=
\iint U_\la f \, \cdot \,
\overline{  \1_{A_+\cap(M^{n-1}\times I_T)} \cdot g }
\, dx dt.
\end{multline}

Note that $U_\la U^*_\la 
=e^{-i(t-s)\la^{-1}\Delta_g}\beta^2(P/\la)$.  Let us
split
$$U_\la U^*_\la =L_\la +G_\la,$$
where if $\alpha_m$ is as in \eqref{22.17}
$$L_\la =\sum_{\{(j,k): \, |j-k|\le 10\}}
\alpha_j(t) e^{-i(t-s)\la^{-1}\Delta_g}\beta^2(P/\la)
\alpha_k(s).$$
Then, it is straightforward to see that \eqref{2.85}
yields
\begin{equation}\label{22.88}
\|L_\la \|_{L^2_tL^{q_e'}_x \to L^2_t L^{q_e}_x}
=O(\la),
\end{equation}
and, using \eqref{k} again, we have that the kernel
of $G_\la$ satisfies
\begin{equation}\label{22.89}
|G_\la(x,t;y,s)|\le C\la^{\frac{n-1}2} \exp(C_M|t-s|) 
\quad \text{if } \, |t-s|\lesssim \log\la,
\end{equation}
for some constant $C_M$ depending on $M^{n-1}$.

Thus, if we repeat the first part of the proof 
of Proposition~\ref{largeprop}, we find that
$$\|U_\la f\|^2_{L^2_tL^{q_e}_x(A_+\cap
(M^{n-1}\times I_T))}\le |I|+|II|,
$$
where
\begin{align*}
I&=\iint L_\la \bigl( \,
\1_{A_+\cap (M^{n-1}\times I_T)} \cdot g\, \bigr)
\, 
\overline{\1_{A_+\cap (M^{n-1}\times I_T)} \cdot g \, }
\, dx dt
\\
II &= \iint G_\la \bigl( \,
\1_{A_+\cap (M^{n-1}\times I_T)} \cdot g\, \bigr)
\, 
\overline{\1_{A_+\cap (M^{n-1}\times I_T)} \cdot g \, }
\, dx dt.
\end{align*}

By \eqref{22.87} and \eqref{22.88}
\begin{equation}\label{22.90}
|I|\le \bigl\| L_\la ( \,
\1_{A_+\cap (M^{n-1}\times I_T)} \cdot g\, )
\|_{L^2_t L^{q_e}_x(A_+\cap (M^{n-1}\times I_T))}
\le C\la.
\end{equation}
Also, by \eqref{22.89}, if  
$T=c_0(\log(\log\la))$ with $c_0>0$ sufficiently
small we have
$$|G_\la(x,t;y,s)|\le C\la^{\frac{n-1}2} \,
(\log\la)^{\e_0},  \, \,
\text{if } \, \, t,s\in I_T.$$
So, for this choice of $T$ we have by \eqref{22.87}
and H\"older's inequality
\begin{align*}
|II|&\le C\la^{\frac{n-1}2}(\log\la)^{\e_0}
\bigl\| \, \1_{A_+\cap (M^{n-1}\times I_T)} \cdot g\,
\bigr\|_{L^1_{t,x}(A_+\cap (M^{n-1}\times I_T)}^2
\\
&\le  C\la^{\frac{n-1}2}(\log\la)^{\e_0}
\, \| \1_{A_+\cap (M^{n-1}\times I_T)}\|^2_{L^2_t
L^{q_e}_x}.
\end{align*}
Since $1\le |U_\la f(x,t)| \cdot (\la^{\frac{n-1}4}
\, (\log\la)^{\e_0})^{-1}$ on $A_+$, we have
$$\|\1_{A_+\cap (M^{n-1}\times I_T)}\|^2_{L^2_t
L^{q_e}_x}\le \la^{-\frac{n-1}2} (\log\la)^{-2\e_0}
\|U_\la f\|^2_{L^2_tL^{q_e}_x(A_+\cap (M^{n-1}
\times I_T))},
$$ 
and thus for $\la \gg1$
\begin{equation}\label{22.91}
|II|\le \tfrac12 \|U_\la f\|^2_{L^2_tL^{q_e}_x(A_+\cap (M^{n-1}
\times I_T))}.
\end{equation}

Since \eqref{22.90} and
\eqref{22.91} imply \eqref{ubound}, the proof is complete.\end{proof}

Next, let us note that by
Proposition~\ref{endprop}
$$\|U_\la f\|_{L^2_tL^{q_e}_x(A_+)}
\le C\la^{\frac12}
(\log\la/\log(\log\la))^{\frac12}.$$
Thus, we would have \eqref{endest} and hence
\eqref{00.15} if we could show that if 
$\e_0>0$ in \eqref{22.86} is small enough, then
for $\la \gg1$,
\begin{equation}\label{22.94}
\|U_\la f\|_{L^2_tL^{q_e}_x(A_-)}
\le C\la^{\frac12}(\log\la)^{\frac12-\delta_1}, \quad
\text{some } \, \delta_1>0.
\end{equation}

We can use our log power gains for $L^{\qc}$ to prove this since, by \eqref{2.85},
\begin{equation}\label{22.95}
\tfrac12=\tfrac{n-1}2(\tfrac12 -\tfrac1{\qe}) \quad \text{and } \, \, \,
\tfrac1{\qc}=\tfrac{n-1}2(\tfrac12-\tfrac1{\qc}).
\end{equation}
We also note that by H\"older's inequality
since $A_-\subset M^{n-1}\times [0,\log\la]$,
the $L^{q_c}_{t,x}$ estimates \eqref{00.10'}
yield
\begin{equation}\label{22.96}
\|U_\la f\|_{L_t^rL^{q_c}_x(A_-)}
\le C\la^{\frac1{q_c}}(\log\la)^{\frac1r-\delta_0}, \, \, 
\text{if } \, 1\le r<q_c, \, \text{and } \, \delta_0=\tfrac1{q_c}(1-\tfrac2{q_c})>0.
\end{equation}

Note that $q_e>q_c$ and let 
\begin{equation}\label{22.97}
\tilde \e_0=\tfrac{q_e-q_c}{q_e}\e_0<\e_0 \quad \text{and } \, \,
\tilde \delta_0=\tfrac{q_c}{q_e}\delta_0<\delta_0.
\end{equation}
Then by \eqref{22.86} and \eqref{22.97}
\begin{align}\label{e.16}
\|U_\la f\|_{\kt(A_-)}&\le \|S_\la f\|_{L^\infty(A_-)}^{\frac{\qe-\qc}{\qe}} \cdot \|U_\la f\|^{\frac{\qc}{\qe}}_{L^{\frac{2\qc}{\qe}}_tL^{\qc}_x(A_-)}
\\
&\lesssim  (\log\la)^{\tilde \e_0} \la^{\frac{n-1}4 (\frac{\qe-\qc}{\qe})} \|U_\la f\|^{\frac{\qc}{\qe}}_{L^{\frac{2\qc}{\qe}}_tL^{\qc}_x(A_-)} 
\notag
\\
&= (\log\la)^{\tilde \e_0} \la^{\frac{n-1}4} \, \la^{-\frac{n-1}4\frac{\qc}{\qe}} \|U_\la f\|^{\frac{\qc}{\qe}}_{L^{\frac{2\qc}{\qe}}_tL^{\qc}_x(A_-)}  .
\notag
\end{align}
If we let $r=\tfrac{2\qc}{\qe}$, then, since $n\ge4$, we have $r\in [1,\qc)$. Therefore, if we apply \eqref{22.96}
and recall \eqref{22.97}, since $\|f\|_2=1$, we can bound the last factor as follows
\begin{align}\label{e.17}
\|U_\la f\|^{\frac{\qc}{\qe}}_{L^{\frac{2\qc}{\qe}}_tL^{\qc}_x(A_-)} 
&\le C\la^{\frac{n-1}2(\frac12-\frac1{\qc})\cdot \frac{\qc}{\qe}} \, \bigl[ (\log\la)^{\frac{\qe}{2\qc}-\delta_0}\bigr]^{\frac{\qc}{\qe}}
\\
&=C\la^{\frac{n-1}4 \cdot \frac{\qc}{\qe}-\frac{n-1}2\cdot \frac1{\qe}} \, (\log\la)^{\frac12-\tilde \delta_0}.
\notag
\end{align}
If we combine \eqref{e.16} and \eqref{e.17} and use \eqref{22.95} one more time we conclude that
$$\|S_\la f\|_{\kt(A_-)}\le C\la^{\frac{n-1}4-\frac{n-1}2 \frac1{\qe}}
(\log\la)^{\frac12-(\tilde \delta_0-\tilde \e_0)}=
C\la^{\frac12}(\log\la)^{\frac12-(\tilde \delta_0-\tilde \e_0)}.$$
This gives us \eqref{22.94} with $\delta_1=\tilde \delta_0-\tilde \e_0$, if $\e_0>0$ is small enough so that $\tilde \e_0<\tilde \delta_0$, which
finishes the proof of Theorem~\ref{endthm}.  \qed

\bigskip

\noindent{\bf Remarks.}  We note that if $M^{n-1}$ is a torus ${\mathbb T^{n-1}}$ 
of dimension $d=n-1\ge 3$ then we can use the toral estimates of Bourgain and Demeter~\cite{BoDe} to obtain
much stronger results than the ones we have obtained for general manifolds of nonpositive curvature.  Indeed, we
recall that in \cite{BoDe} it was shown that 
$\|\beta(P/\la)e^{-it\Delta_{{\mathbb T}^{n-1}}}\|_{L^2({\mathbb T}^{n-1})\to L^{q_c}({\mathbb T}^{n-1}\times [0,1])}
=O(\la^\e)$, $\forall \e>0$.  Therefore, by Sobolev estimates and H\"older's inequality we have
\begin{align*}
\|\beta(P/\la)e^{-it\Delta_{{\mathbb T}^{n-1}}}\|_{L^2_tL^{q_e}_x({\mathbb T}^{n-1}\times [0,1])}
&\lesssim \la^{(n-1)(\frac1{q_c} -\frac1{q_e})}
\|\beta(P/\la)e^{-it\Delta_{{\mathbb T}^{n-1}}}\|_{L^2_tL^{q_c}_x({\mathbb T}^{n-1}\times [0,1])}
\\
&\le  \la^{(n-1)(\frac1{q_c} -\frac1{q_e})}
\|\beta(P/\la)e^{-it\Delta_{{\mathbb T}^{n-1}}}\|_{L^{q_c}_tL^{q_c}_x({\mathbb T}^{n-1}\times [0,1])}
\\
&\lesssim \la^{(n-1)(\frac1{q_c} -\frac1{q_e})+\e} \|f\|_2 = \la^{\frac2{n+1}+\e} \|f\|_2.
\end{align*}
If $d=n-1\ge 3$, this is a $\la^{\frac2{n+1}-\frac12+\e}\le \la^{-\frac1{10}+\e}$ over the universal bounds
of Burq, G\'erard and Tzvetkov \cite{bgtmanifold}, which is much better than our $(\log\log\la)^{-1/2}$
improvement in Theorem~\ref{endthm}.

On the other hand, it seems likely that we shall be able to obtain no loss for dyadic estimates on tori
${\mathbb T}^n$
 on intervals
of length $\la^{-1+\delta_n}$ for some $\delta_n>0$, which would be the natural analog of \eqref{00.14}
in this setting.  We hope to study this problem as well as possible improved Strichartz estimates for
spheres\footnote{We should point out
that in a recent work S\'anchez and Esquivel~\cite{sanchez2021sharp} stronger results
than those in \cite{bgtmanifold} were stated.  However,
there is a gap in the
arguments in \cite{sanchez2021sharp} 
based on incorrect use of Sobolev estimates, and simple
examples (such as the function $f_\la=\beta(P/\la)(x,x_0)$ discussed in the introduction) show that
some of the results in \cite{sanchez2021sharp}
are invalid.}
 in a later work.

\newsection{Local variable coefficient harmonic analysis: Proof of Proposition~\ref{locprop}}

We are dealing with $A^{\theta_0}_\nu\in S^0_{7/8,1/8}$ which are pseudo-differential cutoffs at the scale $\theta_0=\la^{-1/8}$.  In order to obtain the gains involved in the last term in the right side of \eqref{b1} we shall have to also use cutoffs at the scale $\theta_\ell = 2^\ell \theta_0$ with $\ell<0$.

To prove this we shall use the strategy in Blair
and Sogge \cite{SBLog} and earlier works, especially
Tao, Vargas and Vega~\cite{TaoVargasVega} and Lee~\cite{LeeBilinear}.

We first note that if $\delta$ as in \eqref{22.6} is small enough we have
\begin{equation}\label{b2}
\alpha_m(t)\tilde \sigma_\la -
\sum_\nu \alpha_m(t) \tilde \sigma_\la \Atn=R_\la,
\, \,  \text{where } \, \, 
\|R_\la H\|_{L^\infty_{t,x}}\lesssim \la^{-N}
\|H\|_{L^2_{t,x}} \, \, \forall N.
\end{equation}
Thus, we have
\begin{equation}\label{b3}
\bigl(\, \alpha_m(t)\tilde \sigma_\la H\, \bigr)^2
=\sum_{\nu,\tilde \nu}
\bigl(\alpha_m(t)\tilde \sigma_\la \Atn H\bigr)\cdot
\bigl(\alpha_m(t)\tilde \sigma_\la \Atnt H\bigr)
+O(\la^{-N} \|H\|_{L^2_{t,x}}^2) \, \, \forall \, N.
\end{equation}

As in earlier works, let
\begin{equation}\label{b4}
\diag(H) = \sum_{(\nu,\tilde \nu)\in \xid}
\bigl(\alpha_m(t)\tilde \sigma_\la \Atn H\bigr)\cdot
\bigl(\alpha_m(t)\tilde \sigma_\la \Atnt H\bigr),
\end{equation}
and
\begin{equation}\label{b5}
\far(H)=\sum_{(\nu,\tilde \nu)\notin \xid}
\bigl(\alpha_m(t)\tilde \sigma_\la \Atn H\bigr)\cdot
\bigl(\alpha_m(t)\tilde \sigma_\la \Atnt H\bigr)
+O(\la^{-N} \|H\|_{L^2_{t,x}}^2),\end{equation}
with the last term denoting the error term in \eqref{b3}.
Thus,
\begin{equation}\label{b6}
\bigl(\alpha_m(t)\tilde \sigma_\la H\bigr)^2
=\diag(H)+\far(H).
\end{equation}

Thus, the summation in $\diag(H)$ is over near diagonal pairs $(\nu,\tilde \nu)$.  In particular we
have $|\nu-\tilde \nu|\le C\theta_0$ for some uniform constant
as $\nu,\tilde \nu$ range over $\theta_0{\mathbb Z}^{(2n-3)}$.
The other term $\far(H)$ is the remaining pairs, which include many which are far from the diagonal.  This sum will provide the contribution to the last term in
\eqref{b1}. 

The two types of terms here are treated differently,
as in analyzing parabolic restriction problems or 
spectral projection estimates. 

We can treat the first term in the right of \eqref{b6} as in \cite{BHSsp} and \cite{SBLog} by using a variable coefficient variant of
Lemma 6.1 in 
\cite{TaoVargasVega} (see also Lemma 4.2 in \cite{SBLog}):

\begin{lemma}\label{blemma}  If $\diag(H)$ is as in
\eqref{b6} and $n\ge3$, then we have the uniform bounds
\begin{equation}\label{b7}
\|\diag(H)\|_{L^{q_c/2}_{t,x}}
\lesssim \Bigl(\, \sum_\nu
\|\alpha_m(t) \tilde \sigma_\la
\Atn H\|_{L^{q_c}_{t,x}}^{q_c}\,
\Bigr)^{2/q_c}
+O(\la^{\frac2{q_c}-}\|H\|_{L^2_{t,x}}^2).
\end{equation}
\end{lemma}

We also need the following estimate for
$\far(H)$ which will be proved using bilinear
oscillatory integral estimates of Lee~\cite{LeeBilinear}
and arguments of two of us in \cite{BlairSoggeRefined}, \cite{blair2015refined} and \cite{SBLog}.

\begin{lemma}\label{leelemma}
If $\far(H)$ is as in \eqref{b5}, and, as above
$\theta_0=\la^{-1/8}$, then for all $\e>0$ we have
for $H=S_\la f$
\begin{equation}\label{b8}
\iint |\far(H)|^{q/2}\, dxdt \lesssim_{\e}
\la^{1+\e} \, \bigl(\la^{7/8}\bigr)^{\frac{n-1}2
(q-q_c)} \, \|H\|_{L^2_{t,x}}^{q}, \quad
\text{if } \, q=\tfrac{2(n+2)}n.
\end{equation}
\end{lemma}

Let us postpone the proofs of these two lemmas for a
bit and show how they can be used to obtain
Proposition~\ref{locprop}.

If we let $q=\tfrac{2(n+2)}n$ as in Lemma~\ref{leelemma}, we note that
$q<q_c$ and also
\begin{multline*}|\alpha_m(t) \tilde \sigma_\la H \cdot
\alpha_m(t) \tilde \sigma_\la H|
\\
\le 2^{q/2}
\, |\alpha_m(t) \tilde \sigma_\la H \cdot
\alpha_m(t) \tilde \sigma_\la H|^{\frac{q_c-q}2}
\cdot \bigl( \,
|\diag(H)|^{q/2}+|\far(H)|^{q/2}\, \bigr).
\end{multline*}
Thus,
\begin{align}\label{b9}
\|\alpha_m(t) &\tilde \sigma_\la H
\|_{L^{q_c}(A_-)}^{q_c} 
=\int_{A_-}\bigl| \alpha_m(t) \tilde \sigma_\la H \cdot
\alpha_m(t) \tilde \sigma_\la H\bigr|^{q_c/2} \, dxdt
\\
&\lesssim \notag
\int_{A_-} |\alpha_m(t) \tilde \sigma_\la H \cdot
\alpha_m(t) \tilde \sigma_\la H|^{\frac{q_c-q}2} \, 
|\diag(H)|^{q/2} \, dxdt
\\
&\quad +\int_{A_-} |\alpha_m(t) \tilde \sigma_\la H \cdot
\alpha_m(t) \tilde \sigma_\la H|^{\frac{q_c-q}2} \, 
|\far(H)|^{q/2} \, dxdt \quad =I+II.
\notag
\end{align}

To estimate $II$ we use \eqref{b8}, the ceiling
for $A_-$, and the fact that $\tilde \sigma_\la
H=\tilde S_\la f$ if $H=S_\la f$ to see that
\begin{multline*}
II \lesssim \|\alpha_m(t)\tilde S_\la f\|_{L^\infty(A_-)}^{q_c-q} \cdot 
\la^{1+\e} \, \bigl(\la^{7/8}\bigr)^{\frac{n-1}2
(q-q_c)} \, \|H\|_{L^2_{t,x}}^{q}
\\
\le \la^{(\frac{n-1}4+\frac18)(q_c-q)}
\cdot \la^{-(q_c-q)(\frac78\cdot \frac{n-1}2)}
\cdot \la^{1+\e} \|H\|^q_{L^2_{t,x}}
=O(\la^{1-\delta_n+\e}\|H\|^{q_c}_{L^2_{t,x}}),
\, \, \text{some } \, \delta_n>0.
\end{multline*}
We have $\delta_n>0$ since $(q_c-q)(\tfrac{3(n-1)}{16}-\tfrac18)>0$, and also 
$\|H\|_{L^2_{t,x}}^{q_c}$ dominates 
$\|H\|_{L^2_{t,x}}^{q}$ since $q_c>q$ and
$\|H\|_{L^2_{t,x}}\approx T$ since $H=S_\la f$, $\|f\|_2=1$
and $e^{-it\la^{-1}\Delta_g}$ is a unitary operator on $L^2_x$.  

Since we may take $\e<\delta_n$, $II^{1/q_c}$ is dominated by the last term in \eqref{b1},  Consequently, we just need to see that $I^{1/q_c}$ is
dominated by the other term in the right side of this inequality. To estimate this term we
use H\"older's inequality followed by Young's 
inequality and Lemma~\ref{blemma} to see that
\begin{align*}
I&\le \|\alpha_m(t) \tilde \sigma_\la H
\cdot \alpha_m(t)\tilde \sigma_\la H\|_{L^{q_c/2}(A_-)}^{\frac{q_c-q}2} \cdot 
\|\diag(H)\|^{q/2}_{L^{q_c/2}_{t,x}}
\\
&\le \tfrac{q_c-q}{q_c}
\|\alpha_m(t) \tilde \sigma_\la H
\cdot \alpha_m(t)\tilde \sigma_\la H\|_{L^{q_c/2}(A_-)}^{q_c/2}
+\tfrac{q}{q_c} \|\diag(H)\|^{q_c/2}_{L^{q_c/2}_{t,x}}
\\
&\le \tfrac{q_c-q}{q_c} \|\alpha_m(t)\tilde \sigma_\la H
\|_{L^{q_c}(A_-)}^{q_c} 
+C\sum_\nu \|\alpha_m(t)\tilde \sigma \Atn H\|_{L^{q_c}_{t,x}}^{q_c}
+O(\la^{1-}\|H\|_{L^2_{t,x}}^{q_c}).
\end{align*}
Since $\tfrac{q_c-q}{q_c}<1$, the first term
in the right can be absorbed in the left
side of \eqref{b9}, and this, along with
the estimate for $II$ above yields \eqref{b1}.

Thus, if we can prove Lemma~\ref{blemma} and
Lemma~\ref{leelemma}, the proof of Proposition~\ref{locprop} will be complete.

\medskip

\noindent {\bf Proof of Lemma~\ref{blemma}.}

Let us first define slightly wider microlocal
cutoffs by setting
$$\tAtn =\sum_{|\mu-\nu|\le C_0 \theta_0}A^{\theta_0}_\mu.$$
We can fix $C_0$ large enough so that
\begin{equation}\label{b10}
\|\Atn-\Atn\tAtn\|_{L^p_x\to L^p_x}=O(\la^{-N})
\, \, \forall \, N\, \, \text{if } \,
\, 1\le p\le \infty.
\end{equation}
Also, like the original $\Atn$ operators
the $\tAtn$ operators are almost orthogonal
\begin{equation}\label{b11}
\sum_\nu \|\tAtn h\|^2_{L^2_x}\lesssim
\|h\|_{L^2_x}^2.
\end{equation}

Since
$$\|\alpha_m(t)\tilde \sigma_\la F\|_{L^{q_c}_{t,x}}\le C\la^{\frac1{q_c}}\|F\|_{L^2_{t,x}},
$$
we conclude that, in order to prove \eqref{b7}, we may replace $\diag(H)$ by $\tdiag(H)$ where
the latter is defined by the analog of 
\eqref{b4} with $\Atn$ and $\Atnt$ 
replaced by $\Atn \tAtn$ and $\Atnt
\tAtnt$, respectively.

So, it suffices to prove
\begin{multline}\label{b12}
\bigl\|\sum_{(\nu,\tilde \nu)\in \xid}
(\alpha_m(t)\tilde \sigma_\la
\Atn\tAtn H)\cdot 
(\alpha_m(t)\tilde \sigma_\la
\Atnt\tAtnt H)\bigr\|_{L^{q_c/2}_{t,x}}
\\
\le C\Bigl(\sum_\nu
\|\alpha_m(t)\tilde \sigma_\la \Atn H\|^{q_c}_{L^{q_c}_{t,x}}
\Bigr)^{2/q_c}
 +O(\la^{\frac2{q_c}-}\|H\|_{L^2_{t,x}}^2).
\end{multline}
We shall need the following variant of  \eqref{cc3},
\begin{equation}\label{b13}
\|\alpha_m(t)[\, \tilde \sigma_\la
\Atn-\Atn\tilde \sigma_\la\, ]F\|_{L^{q_c}_{t,x}}
\lesssim \la^{\frac1{q_c}-\frac14}\|F\|_{L^2_{t,x}}.
\end{equation}
This follows from the proof of Lemma~\ref{comprop}, or, alternately from
Lemma~\ref{qlemma}, \eqref{cc3} and the fact that the commutator
$[B,A^{\theta_0}_\nu]$ is bounded on $L^{q_c}_x(M^{n-1})$ with norm
$O(\la^{-7/8})$. 
Since the $\Atn$ commute with the
$\alpha_m(t)$ time-localizations, by \eqref{b11} and \eqref{b13} we 
would have \eqref{b12} if we could show that
\begin{multline}\label{b14}
\bigl\|\sum_{(\nu,\tilde \nu)\in \xid}
(\Atn(\alpha_m(t)\tilde \sigma_\la
\tAtn H)\cdot 
\Atnt(\alpha_m(t)\tilde \sigma_\la
\tAtnt H)\bigr\|_{L^{q_c/2}_{t,x}}
\\
\le C\Bigl(\sum_\nu
\|\alpha_m(t)\tilde \sigma_\la \Atn H\|^{q_c}_{L^{q_c}_{t,x}}
\Bigr)^{2/q_c}
 +O(\la^{\frac2{q_c}-}\|H\|_{L^2_{t,x}}^2).
\end{multline}

Note that the functions in the norm in the 
left side of \eqref{b14} vanish if $
t\notin [m-1,m+1]$.  Therefore, if we take
$r=(q_c/2)'$ so that $r$ is the conjugate
exponent for $q_c/2$, it suffices to show that
\begin{multline}\label{b15}
\Bigl| \sum_{(\nu,\tilde \nu)\in \xid}
\iint \Atn(\alpha_m(t)\tilde \sigma_\la
\tAtn H)\cdot 
\Atnt(\alpha_m(t)\tilde \sigma_\la
\tAtnt H) \, \cdot G \, dtdx\, \Bigr|
\\
\le C\Bigl(\sum_\nu
\|\alpha_m(t)\tilde \sigma_\la \Atn H\|^{q_c}_{L^{q_c}_{t,x}}
\Bigr)^{2/q_c}
 +O(\la^{\frac2{q_c}-}\|H\|_{L^2_{t,x}}^2),
\\ \text{if } \, \|G\|_{L^r_{t,x}}=1
\, \, \text{and } \,
G(t,x)=0 \, \, \,
\text{if } \, \, t\notin [m-1,m+1].
\end{multline}

Note that if $x$ and $\nu$ are fixed and 
$\xi \to \Atn(x,\xi)$ does not vanish identically, then this function of $\xi$ is supported in a cube $Q^{\theta_0}_\nu(x)\subset {\mathbb R}^{n-1}_\xi$ of sidelength $\approx \la^{7/8}$.
The cubes can be chosen so that, if $\eta_\nu(x)$ is its center, then $\partial^\gamma_x\eta_\nu(x)=O(\la)$ for all 
multi-indices $\gamma$.  Keeping this in mind it
is straightforward to construct for every
pair $(\nu,\tilde \nu)\in \xid$ symbols
$b_{\nu,\tilde \nu}(x,\xi)$ belonging to 
a bounded subset of $S^0_{7/8,1/8}$ satisfying
\begin{equation}\label{p1}
b_{\nu,\tilde \nu}(x,\eta)=1 \, \, \text{if } \, \,
\text{dist}\bigl(\eta, \, \text{supp}_\xi \Atn(x,\xi) \, +\, \text{supp}_\xi \Atnt(x,\xi)\bigr)
\le \la^{7/8},
\end{equation}
with ``$+$'' denoting the algebraic sum.  Using this and a simple integration by parts argument shows that for every pair $(\nu,\tilde \nu)\in \xid$
\begin{equation}\label{p2}
\bigl\| (I-b_{\nu,\tilde \nu}(x,D))\bigl[
\Atn h \cdot \Atnt h]\bigr] \bigr\|_{L^\infty_x}
\le C_N \la^{-N}\|h\|^2_{L^1_x}, \quad
\forall \, N.
\end{equation}
The symbols can also be chosen so that
$b_{\nu_1,\tilde \nu_1}(x,\xi)$ and
$b_{\nu_2,\tilde \nu_2}(x,\xi)$ have disjoint
supports if $(\nu_j,\tilde \nu_j)\in \xid$,
$j=1,2$ and $\min(|(\nu_1-\nu_2,\tilde \nu_1
-\tilde \nu_2)|, \, |(\nu_1-\tilde \nu_2,
\tilde \nu_1-\nu_2)|)\ge C_2 \theta_0$ with 
$C_2$ being a fixed constant independent of 
$\la$ since all pairs in $\xid$ are nearly
diagonal.  Due to this, the adjoints, 
$b^*_{\nu,\tilde \nu}(x,D)$ are almost orthogonal 
in the sense that we have the uniform bounds
\begin{equation}\label{p3}
\sum_{(\nu,\tilde \nu)\in \xid}
\|b^*_{\nu,\tilde \nu}(x,D)h\|^2_{L^2_x}\lesssim
\|h\|^2_{L^2_x}.
\end{equation}
Since $\text{supp}_\xi \Atn(x,\xi) \, +\, \text{supp}_\xi \Atnt(x,\xi)$ is contained
in a cube of sidelength $\approx \la^{7/8}$
and can be chosen to have center $\eta_{\nu,\tilde \nu}(x)$ satisfying 
$\partial^\gamma_x \eta_{\nu,\tilde \nu}(x)
=O(\la)$, we can furthermore assume that
we have the uniform bounds
\begin{equation}\label{p4}
\sup_{(\nu,\tilde \nu)\in \xid}
\|b^*_{\nu,\tilde \nu}(x,D)h\|_{L^\infty_x}
\lesssim \|h\|_{L^\infty_x}.
\end{equation}

We have now set up our variable coefficient 
version of the simple argument in 
\cite{TaoVargasVega} that will allow us
to obtain \eqref{b15}.  First, by \eqref{p2},
modulo $O(\la^{-N}\|H\|_{L^2_{t,x}}^2)$ errors,
the left side of \eqref{b15} is dominated by
\begin{multline}\label{b16}
\Bigl| \sum_{(\nu,\tilde \nu)\in \xid}
\iint (\Atn(\alpha_m(t)\tilde \sigma_\la
\tAtn H)\cdot 
\Atnt(\alpha_m(t)\tilde \sigma_\la
\tAtnt H \, \cdot \bigl( b^*_{\nu,\tilde \nu}(x,D)G \bigr)\, dtdx\, \Bigr|
\\
\le
\Bigl(
\sum_{(\nu,\tilde \nu)\in \xid}
\| \Atn(\alpha_m(t)\tilde \sigma_\la
\tAtn H)\cdot 
\Atnt(\alpha_m(t)\tilde \sigma_\la
\tAtnt H)\|_{L^{q_c/2}_{t,x}}^{q_c/2}\,
\Bigr)^{2/q_c}
\\
\cdot \Bigl(
\sum_{(\nu,\tilde \nu)\in \xid}
\| b^*_{\nu,\tilde \nu}(x,D)G\|_{L^{r}_{t,x}}^r
\Bigr)^{1/r},
\end{multline}
since $r=(q_c/2)'$.

Note that $r\in [2,\infty)$ since
$q_c\in (2,4]$.  So, if we use \eqref{p3},
\eqref{p4} and an interpolation argument
we conclude that
$$\Bigl(
\sum_{(\nu,\tilde \nu)\in \xid}
\| b^*_{\nu,\tilde \nu}(x,D)G\|_{L^{r}_{t,x}}^r
\Bigr)^{1/r}=O(1),$$
for $G$ as in \eqref{b15}.  As a result,
we conclude that modulo $O(\la^{\frac2{q_c}-}
\|H\|_{L^2_{t,x}})$ errors, the left
side of \eqref{b14} is dominated by
\begin{multline*}
\Bigl(
\sum_{(\nu,\tilde \nu)\in \xid}
\| \Atn(\alpha_m(t)\tilde \sigma_\la
\tAtn H)\cdot 
\Atnt(\alpha_m(t)\tilde \sigma_\la
\tAtnt H)\|_{L^{q_c/2}_{t,x}}^{q_c/2}\,
\Bigr)^{2/q_c}
\\
\lesssim 
\Bigl(\sum_\nu
\|\alpha_m(t)\Atn \tilde \sigma_\la
\tAtn H\|_{L^{q_c}_{t,x}}^{q_c}\Bigr)^{2/q_c}.
\end{multline*}

If we repeat earlier arguments and use
\eqref{b10} again, we conclude that the
right side of the preceding inequality is dominated by the right side of \eqref{b7}, and
this finishes the proof of Lemma~\ref{blemma}.

\medskip

\noindent{\bf Bilinear oscillatory integral estimates: Proof of Lemma~\ref{leelemma}}

To prove \eqref{b8} we note that for a given $\theta=2^k\theta_0$, $k\ge 10$ we have
for each fixed $c_0>0$
\begin{equation}\label{l1}
\alpha_m(t) \tilde \sigma_\la A^{\theta_0}_\nu H=
\sum_{\mu'\in {c_0\theta\,\mathbb Z}^{2n-3}}
\tilde \sigma_\la A^{c_0\theta}_{\mu'} A^{\theta_0}_\nu H
+O(\la^{-N}\|H\|_2).
\end{equation}
As in \cite{BlairSoggeRefined} it will be convenient to choose $c_0=2^{-m_0}<1$ so that we are working
at scales $c_0\theta$ rather than $\theta$ to ensure that we easily have the separation to apply bilinear 
oscillatory integral bounds.

With this in mind we note that if we fix $k\ge10$ in the first sum in \eqref{m14}, we then have for a 
given fixed $c_0=2^{-m_0}$, $m_0\in {\mathbb N}$, and pair of dyadic cubes $\tau^\theta_\mu$, $\tau^\theta_{\tilde \mu}$
with $\tau^\theta_\mu \sim \tau^\theta_{\tilde \mu}$
and $\theta=2^k\theta_0$ 
\begin{multline}\label{l2}
\sum_{(\nu,\tilde \nu)\in \tau^\theta_\mu\times \tau^\theta_{\tilde \mu}}
(\alpha_m(t)\tilde \sigma_\la A^{\theta_0}_\nu H)
(\alpha_m(t)\tilde \sigma_\la A^{\theta_0}_{\tilde \nu} H)
\\
=\sum_{(\nu,\tilde \nu)\in \tau^\theta_\mu\times \tau^\theta_{\tilde \mu}} \, 
\sum_{\substack{\tau^{c_0\theta}_{\mu'} \cap \overline{\tau}^\theta_\mu \ne \emptyset
\\ \tau^{c_0\theta}_{\tilde \mu'} \cap \overline{\tau}^\theta_{\tilde \mu} \ne \emptyset}}
(\alpha_m(t)\tilde \sigma_\la  A^{c_0\theta}_{\mu'} A^{\theta_0}_\nu H)
(\alpha_m(t)\tilde \sigma_\la  A^{c_0\theta}_{\tilde \mu'} A^{\theta_0}_{\tilde \nu} H)
+O(\la^{-N}\|H\|_2^2),
\end{multline}
if $\overline{\tau}^\theta_\mu$ and $\overline{\tau}^\theta_{\tilde \mu}$ the cubes with the same centers but $11/10$ times the
sidelength
of $\tau^\theta_\mu$ and $\tau^\theta_{\tilde \mu}$, respectively, so that we have
$\text{dist}(\overline{\tau}^\theta_\mu, \overline{\tau}^\theta_{\tilde \mu})\ge \theta/2$ when
$\tau^\theta_\mu  \sim \tau^\theta_{\tilde \mu}$.  
This follows from the fact that for $c_0$ small enough the product of the symbol of $A_{\mu'}^{c_0\theta}$ and $A_\nu^{\theta_0}$ 
vanishes identically if $\tau_{\mu'}^{c_0\theta}\cap \overline{\tau}^\theta_{\mu}=\emptyset$ and $\nu \in \tau^\theta_\mu$, since
$\theta=2^k\theta_0$ with $k\ge 10$.  
Also notice that we then have 
for fixed $c_0=2^{-m_0}$ small enough
\begin{equation}\label{sep}
\text{dist}(\tau^{c_0\theta}_{\mu'}, \tau^{c_0\theta}_{\tilde \mu'})\in [4^{-1}\theta, 4^n  \theta],
\quad \text{if } \, \, \tau^{c_0\theta}_{\mu'} \cap \overline{\tau}^\theta_\mu \ne \emptyset, \, \, \,
\text{and } \, \, \tau^{c_0\theta}_{\tilde \mu'} \cap \overline{\tau}^\theta_{\tilde \mu} \ne \emptyset.
\end{equation}
Also, of course, for each $\mu$ there are $O(1)$ terms $\mu'$  with $\tau^{c_0\theta}_{\mu'} \cap \overline{\tau}^\theta_\mu \ne \emptyset$,
if $c_0$ is fixed.

Note also, that if we fix $c_0$ then for our fixed pair $\tau^\theta_\mu\sim \tau^\theta_{\tilde \mu}$ of
$\theta$-cubes there are only $O(1)$ summands involving $\mu'$ and $\tilde \mu'$ in the
right side of \eqref{l2}.

Keeping this in mind, we claim that we would have favorable bounds for the $L^{q/2}_{t,x}$-norm, $q=\tfrac{2(n+2)}n$, 
of the first term in \eqref{m14} and hence $\far(H)$
%for each fixed $k=1,2,\dots$ 
if we could prove the following:

\begin{proposition}\label{key}
Let $\theta=2^k\theta_0=2^k\la^{-1/8}\ll 1$ with $k\in {\mathbb N}$.  Then we can fix $c_0=2^{-m_0}$ small enough so that whenever
% dyadic close cubes $\tau^\theta_\mu\sim \tau^\theta_{\tilde \mu}$
% are fixed 
\begin{equation}\label{sep'}\text{dist}(\tau^{c_0\theta}_{\nu}, \tau^{c_0\theta}_{\tilde \nu})\in [4^{-1}\theta, 4^n  \theta],
\end{equation}
 one has the uniform bounds for $0\le m\le C\log\la$
 \begin{multline}\label{l3}
 \iint \bigl| (\alpha_m(t)\tilde \sigma_\la A^{c_0\theta}_{\nu}H_1) \, (\alpha_m(t)\tilde \sigma_\la A^{c_0\theta}_{\tilde \nu}H_2) \bigr|^{q/2} \, dt dx
 \\
 \lesssim_\e \la^{1+\e}
\, \bigl(2^k\la^{7/8}\bigr)^{\frac{n-1}2 (q-q_c)} \, \|H_1\|^{q/2}_{L^2_{t,x}} \, \|H_2\|^{q/2}_{L^2_{t,x}},
%\, \, \text{if } \, \, \tau^{c_0\theta}_{\mu'} \cap \overline{\tau}^\theta_\mu \ne \emptyset, \, \, \,
%\text{and } \, \, \tau^{c_0\theta}_{\tilde \mu'} \cap \overline{\tau}^\theta_{\tilde \mu} \ne \emptyset,
\end{multline}
with, as in \eqref{b7}, $q=\tfrac{2(n+2)}n$, assuming that $H_k(y,s)=0$, $k=1,2$, for $|s|\ge C\log\la$.
 \end{proposition}

 Before using Lee's \cite{LeeBilinear} oscillatory integral estimates to prove this Proposition, let us verify the above claim.
 
 We first note that if
 $$H_1=\sum_{\nu\in \tau^\theta_\mu} A^{\theta_0}_\nu H
 \quad \text{and } \, \, \, H_2=\sum_{\tilde \nu\in \tau^\theta_{\tilde \mu}} A^{\theta_0}_{\tilde \nu}H,$$
 then by the almost orthogonality of the $A^\theta_\nu$ operators, there is a fixed constant $C$ so that
 $$\|H_1\|_{L^2_{t,x}}\le C
 \bigl(\sum_{\nu \in \tau^\theta_\mu}\|A^{\theta_0}_\nu H\|^2_{L^2_{t,x}}\bigr)^{1/2}
 \quad \text{and } \, \, 
 \|H_2\|_{L^2_{t,x}}\le C
\bigl( \sum_{\tilde \nu \in \tau^\theta_{\tilde \mu}}\|A^{\theta_0}_\nu H\|^2_{L^2_{t,x}}\bigr)^{1/2}.$$
 Thus, \eqref{l1}, \eqref{sep}, \eqref{l3} and Minkowski's inequality yield
 the following estimates for the first term in \eqref{m14} with $k\ge10$, $\theta=2^k\theta_0$
 and $q=\tfrac{2(n+2)}n$:
 \begin{align}\label{l4}
 &\bigl\| \sum_{(\mu,\tilde \mu): \, \tau^\theta_\mu\sim \tau^\theta_{\tilde \mu}} \sum_{(\nu,\tilde \nu)\in \tau^\theta_\mu
 \times \tau^\theta_{\tilde \mu}}
 (\alpha_m(t)\tilde \sigma_\la A^{\theta_0}_{\nu}H)
 \, 
 (\alpha_m(t)\tilde \sigma_\la A^{\theta_0}_{\tilde \nu}H) \bigr\|_{L^{q/2}_{t,x}}
 \\
 &\le
 \sum_{(\mu,\tilde \mu): \, \tau^\theta_\mu\sim \tau^\theta_{\tilde \mu}}
 \bigl\| 
 \sum_{\substack{\tau^{c_0\theta}_{\mu'} \cap \overline{\tau}^\theta_\mu \ne \emptyset
\\ \tau^{c_0\theta}_{\tilde \mu'} \cap \overline{\tau}^\theta_{\tilde \mu} \ne \emptyset}}
\bigl(\alpha_m(t)\tilde \sigma_\la A^{c_0\theta}_{\mu'}(\sum_{\nu\in \tau^\theta_\mu}A^{\theta_0}_\nu H)\bigr) \cdot
\bigl(\alpha_m(t)\tilde \sigma_\la A^{c_0\theta}_{\tilde \mu'}(\sum_{\tilde \nu\in \tau^\theta_{\tilde \mu}}A^{\theta_0}_\nu H)\bigr)
\|_{L^{q/2}_{t,x}} \notag
\\ &\qquad \qquad \qquad \qquad\qquad \qquad
+O(\la^{-N}\|H\|^2_{L^2_{t,x}}) \notag
\\
&\lesssim_\e \la^{(1+\e)\frac2q} \, \bigl(2^k\la^{7/8}\bigr)^{\frac{n-1}q \, (q-q_c)} 
\sum_{(\mu,\tilde \mu): \, \tau^\theta_\mu\sim \tau^\theta_{\tilde \mu} }
\bigl(\sum_{\nu\in \tau^\theta_\mu} \|A^{\theta_0}_\nu H\|_{L^2_{t,x}}^2 \bigr)^{1/2}
\bigl(\sum_{\tilde \nu\in \tau^\theta_{\tilde \mu}} \|A^{\theta_0}_\nu H\|_{L^2_{t,x}}^2\bigr)^{1/2}
\notag
\\ &\qquad \qquad \qquad \qquad\qquad \qquad
+O(\la^{-N}\|H\|^2_{L^2_{t,x}}) \notag
\\
&\lesssim_\e \la^{(1+\e)\frac2q} \, \bigl(2^k\la^{7/8}\bigr)^{\frac{n-1}q \, (q-q_c)} 
\sum_\nu \sum_{\nu\in \tau^\theta_\nu} \|A^{\theta_0}_\nu H\|_{L^2_{t,x}}^2 
+O(\la^{-N}\|H\|^2_{L^2_{t,x}}) \notag
\\
&\lesssim_\e \la^{(1+\e)\frac2q} \, \bigl(2^k\la^{7/8}\bigr)^{\frac{n-1}q \, (q-q_c)} 
\|H\|^2_{L^2_{t,x}} +O(\la^{-N}\|H\|^2_{L^2_{t,x}}) \notag.
\end{align}
In the above we used the fact that for each $\tau^\theta_\mu$ there are $O(1)$ $\tau^{c_0\theta}_{\mu'}$
with $\tau^{c_0\theta}_{\mu'}\cap \overline{\tau}^\theta_\mu \ne \emptyset$, and
$O(1)$ $\tau^\theta_{\tilde \mu}$ with $\tau^\theta_\mu\sim \tau^\theta_{\tilde \mu}$, as well as \eqref{m12}.

Since $q-q_c<0$, we can clearly show that if we replace $\far(H)$ by the first term in \eqref{m14}, then the
resulting expression satisfies the bounds in \eqref{b8}.  Since by \eqref{b5} the additional part of $\far(H)$ is pointwise
bounded by $O(\la^{-N}\|H\|^2_{L^2_{t,x}})$, we conclude that we have reduced matters to proving Proposition~\ref{key}.

\medskip
\noindent{\bf Proof of Proposition~\ref{key}: Schr\"odinger curves and coordinates, and using  bilinear oscillatory integral estimates}

We first need to collect some facts about the kernels of the operators $\tilde \sigma_\la A^{c_0\theta}_\nu$ in \eqref{l3}.
As we shall momentarily see, they are highly concentrated near certain ``Schr\"odinger curves''.

To describe these, let us recall \eqref{m10}, which says that $A_\nu^{c_0\theta}=A^{c_0\theta}_j(x,D_x)\circ A^{c_0\theta}_\ell(P)$,
if $\nu=(c_0\theta j,c_0\theta\ell)\in c_0\theta{\mathbb Z}^{2(n-2)}\times c_0\theta{\mathbb Z}$.  We also recall that, by \eqref{m5}, the symbols of the
``directional'' operators $A^{c_0\theta}_j$ are each highly concentrated near a unit speed geodesic
\begin{equation}\label{direct}
\gamma_j(s)=(x_j(s),\xi_j(s))\in S^*\Omega, \quad \text{with } \, (x_j(s),\xi_j(s))\in \text{supp }A^{c_0\theta}_j(x,\xi).
\end{equation}
Since $\gamma_j$ is of unit speed, we have $d_g(x_j(s_1),x_j(s_2))=|s_1-s_2|$ for points on the geodesic in $\Omega$.
On the other hand, as described in \cite{HSSchro}, due to the role of the ``height operators'' $A_\ell^{c_0\theta}(P)$, the space-time
Schr\"odinger curves associated to the operators in \eqref{l3} will necessarily have to involve speeds that are
associated with the heights $\kappa_\ell^{c_0\theta}$ in \eqref{m7} that define the operators $A_\ell^{c_0\theta}(P)$ (see also
\cite{AnanNon} and \cite{GGH}).

To be more specific, we claim that, if we define the ``{\em Schr\"odinger curves}'' corresponding to $\nu$,
\begin{equation}\label{q1}
\iota_{s_0,\nu}(s)=(x_j(2\kappa s), -(s-s_0))\in \Omega\times \R,
\quad \nu=(c_0\theta j,c_0\theta\ell),\, \, \kappa=\kappa^{c_0\theta}_\ell,
\end{equation}
then the kernels $K_\nu^{c_0\theta}(x,t;y,s)$ of the operators $\tilde \sigma_\la A^{c_0\theta}_\nu$
must be highly concentrated in ``Schr\"odinger tubes'' of radius $\approx \theta$ about the curves
$\iota_\nu$ in \eqref{q1}.  Note that $s\to x_j(2\kappa^{c_0\theta}_\ell s)$ is a geodesic of speed $2\kappa_\ell^{c_0\theta}$,
meaning that $d_g(x_j(2\kappa^{c_0\theta}_\ell s_1), x_j(2\kappa^{c_0\theta}_\ell s_2))=2\kappa_\ell^{c_0\theta}|s_1-s_2|$.
Also, the minus sign in the time variable in \eqref{q1} is just based on the minus sign in \eqref{22.5}, which, as we shall see, we have chosen to be able to 
use the local analysis in \cite{SBLog}, \cite{sogge88}, etc., without unnecessary  sign-confusion.  This minus sign also occurs because
of our sign convention in \eqref{00.1} of course.

\begin{remark}In Figure~\ref{fig:tubes} three Schr\"odinger tubes passing through a common point $(x_0,t_0)$ are depicted.  The two on the right
have a common spatial orientation, meaning that each comes from a common geodesic $\gamma=\gamma_j$ as in \eqref{direct}; however, their
speeds come from different heights and thus do not coincide, which accounts for the separation of the two Schr\"odinger tubes away
from $(x_0,t_0)$ on the right.  The left and right tubes in the figure have a common speed but different spatial components, which accounts for their separation.
  We also point out that in parabolic restriction problems, curves of the form
\eqref{q1} necessarily arise in the analysis due to Knapp phenomena.  In the translation invariant setting, these Schr\"odinger curves
are simply lines in directions pointing in normal directions to relevant portions of paraboloids
as depicted in Figure~\ref{fig:euclidean}.  For variable coefficient Schr\"odinger problems, the analogous 
Knapp phenomenon was discussed in \cite[\S4]{HSSchro}, and additional variable coefficient local analysis that we have exploited was developed there.
\end{remark}

%\begin{figure}
%\centering
%\begin{minipage}{.5\textwidth}
%  \centering
%  \includegraphics[width=.5\linewidth]{tube.png}
%  \caption{Schr\"odinger tubes}
%  \label{fig:tubes}
%\end{minipage}%
%\begin{minipage}{.6\textwidth}
%  \centering
%  \includegraphics[width=.5\linewidth]{PARABOLOID.png}
%  \caption{Euclidean case}
%  \label{fig:euclidean}
%\end{minipage}
%\end{figure}

\begin{figure} 
\begin{minipage}{0.49\textwidth} \centering \includegraphics[width=\linewidth]{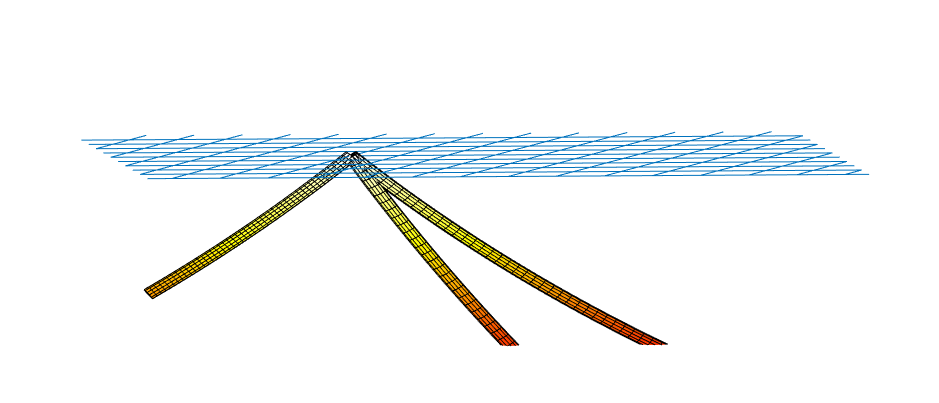} 
\caption{Schr\"odinger tubes}
\label{fig:tubes}
\end{minipage}
\begin{minipage}{0.43\textwidth} \centering \includegraphics[width=\linewidth]{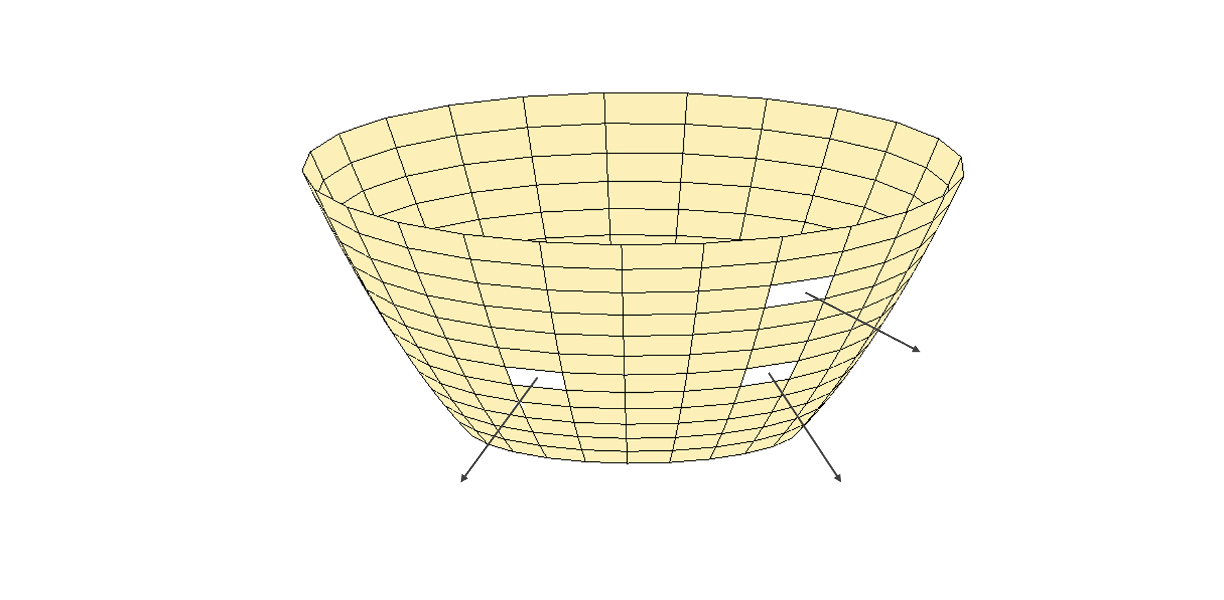} 
\caption{Euclidean case} 
\label{fig:euclidean} 
\end{minipage}
\end{figure}

% \begin{figure} \begin{minipage}{0.49\textwidth} \centering \includegraphics[width=\linewidth]{tube.png} \caption{Schr\"odinger tubes} 
% \label{fig:tubes}\end{minipage} \begin{minipage}{0.43\textwidth} \centering \includegraphics[width=\linewidth]{PARABOLOID.png} \caption{Euclidean case}\label{fig:euclidean} \end{minipage} \end{figure}
%\begin{figure}
%  \includegraphics[width=\linewidth]{tube.png}
%  \caption{Schr\"odinger tubes}
%  \label{fig:tube}
%\end{figure}
%
%\begin{figure}
%  \includegraphics[width=\linewidth]{PARABOLOID.png}
%  \caption{Euclidean case}
%  \label{fig:euclidean}
%\end{figure}
 Let us now state the properties of the kernels $K_\nu^{c_0\theta}(x,t;y,s)$  that we shall require.  To simplify the statements and to 
 also most easily apply Lee's \cite{LeeBilinear} results, let us work in Fermi normal coordinates about the geodesic $x_j(s)$ in 
 \eqref{direct} (see \cite{Tubes}).  
In these coordinates the geodesic becomes part of the last coordinate axis, i.e., $(0,\dots,0,s)$ in ${\mathbb R}^{n-1}$,
 with, as in the constructions of the symbols of the $A^{c_0\theta}_\nu$, $s$ close to the origin.  For the remainder of the section, abusing
 notation a bit, $x=(x_1,\dots,x_{n-1})$ denotes these Fermi normal coordinates.  We then have  \begin{equation}\label{q2}
 d_g((0,\dots,0,x_{n-1}), (0,\dots,0,y_{n-1}))=|x_{n-1}-y_{n-1}|,
 \end{equation}
 and, moreover, on our spatial geodesic in \eqref{direct} we also have that the metric is simply
 $g_{jk}(x)=\delta^k_j$ if $x=(0,\dots,x_{n-1})$ and all the Christoffel symbols vanish there.  
 Thus, $g_{jk}$ agrees with the standard flat rectangular metric to second order along this geodesic.  See \cite{Fermi1}, \cite{Fermi2}.
 Note that in these coordinates we have $(0,(0,\dots,0,1))\in \text{supp }A^{c_0\theta}_j(x,\xi)$ and so for small enough $\theta$ we have
 \begin{multline}\label{ch}
 A^{c_0\theta}_j(0,\xi/|\xi|)=0 \, \, \, \text{when } \,
 |\xi/|\xi|-(0,\dots,0,1)|\ge Cc_0\theta, 
 \\ \text{and } \, \,
 \chi_t(0,(0,\dots,0,1))=(t,(0,\dots,0,1)),
 \end{multline}
 with, as before, $\chi_t$ being geodesic flow, and $C$ here a uniform constant.
 
 We can now formulate the required properties of the kernels.
 
 \begin{lemma}\label{ker}  Fix $0<\delta\ll \tfrac12 \text{Inj }M$.  Assume further that $\mu=\nu,\nu'$ are as in \eqref{sep'},
 and let  $K_{\la,\mu}^{c_0\theta}$ be the kernel of $\tilde \sigma_\la A^{c_0\theta}_\mu$.  In the above coordinates if $c_0\ll 1$ we have
 \begin{equation}\label{q3}
 K_{\la,\mu}^{c_0\theta}(x,t;y,s)=\la^{\frac{n-1}2} e^{-i\la (d_g(x,y))^2/4(t-s)} a_{\la,\mu}(x,t;y,s)+O(\la^{-N}), \quad
 \mu =\nu,\nu',
 \end{equation}
 where, if $\nu=(c_0\theta j,c_0\theta\ell)$, and $\kappa_\ell^{c_0\theta}$ is as in \eqref{q1},
 \begin{multline}\label{q4}
 \bigl| \, \bigl(2\kappa_\ell^{c_0\theta}\tfrac\partial{\partial x_{n-1}}-\tfrac\partial{\partial t})^{m_1}
 \bigl(2\kappa_\ell^{c_0\theta}\tfrac\partial{\partial y_{n-1}}-\tfrac\partial{\partial s})^{m_2} D^\beta_{x,t,y,s}a_{\la,\mu} \bigr|
 \\
 \le C_{m_1,m_2,\beta} \theta^{-|\beta|}, \quad \mu=\nu,\nu'.
 \end{multline}
 Furthermore, for small $\theta$ and $c_0$ there is a constant $C_0$ so that the above $O(\la^{-N})$ errors can be chosen so that the amplitudes
   have the following support properties:
 If $\overline\gamma_j$ denotes the projection onto $M^{n-1}$ of the geodesic in \eqref{m5} and $\overline\gamma_{j'}$ 
 when $j$ is replaced by $j'$, 
\begin{multline}\label{q6}
a_{\la,\mu}(x,t;y,s)=0,
%O(\la^{-N})
 \quad \text{if } \, \, d_g(x,\overline\gamma_k)+d_g(y,\overline\gamma_k)\ge C_1c_0\theta, 
 \\ \text{if } k=j \, \, \text{when } \nu=(c_0\theta j,c_0\theta\ell) \, \, \text{and if} \, \, 
 k=j' \, \, \text{when} \, \, \, \nu'=(c_0\theta j',c_0\theta\ell'),
\end{multline}
 %Additionally, if $c_0$ is  small  we have for a fixed constant $C_1$
 \begin{multline}\label{q5}
 a_{\la,\mu}(x,t;y,s)=0
 %O(\la^{-N}) 
 \quad \text{if } \, \, | d_g(x,y)
 +2\kappa\,(t-s)|\ge C_0 c_0\theta,
 \\
  \text{when } \, \mu=\nu \, \, \text{with} \quad \kappa=\kappa_\ell^{c_0\theta}, \, \, \, 
 \text{or } \, \, \nu=\nu' \, \, \text{with} \quad \kappa=\kappa_{\ell'}^{c_0\theta},
 \end{multline}
 as well as
 \begin{multline}\label{qtube}
 a_{\la,\mu}(x,t;y,s)=0, \, \, \mu=\nu, \nu' ,
 \\  \text{if } \, \,
 |(x_1,\dots,x_{n-2})|+
 |(y_1,\dots,y_{n-2})|+
 |(x_{n-1}-y_{n-1})+2\kappa^{c_0\theta}_\ell(t-s) |\ge C_0 \theta.
 \end{multline}
Finally, %there is a fixed constant $C_0$ so that 
for small $\delta_0>0$ in \eqref{22.6}, the $O(\la^{-N})$ errors can be chosen so that  we also have
\begin{equation}\label{q7}
a_{\mu,\la}(x,t;y,s)=0 \quad \text{if } \, \, \, \bigl|\, d_g(x,y)-\delta\, \bigr|\ge 2\delta_0\delta, \, \, 
\text{or if }\, \, x_{n-1}-y_{n-1}<0, \, \, \mu=\nu, \nu'
\end{equation}
with $\delta$ and $\delta_0$ as in \eqref{22.6}.
 \end{lemma}

 This lemma is just a small variation on Lemma 4.3 in \cite{SBLog} (see also Lemma 3.2 in \cite{BlairSoggeRefined}), and we shall use the 
 aforementioned result
 from \cite{SBLog}
 and the nature of the $\sigma_\la$ operators to obtain the above estimates.  We shall 
 postpone the proof  until  the final section in which we prove  all  the kernel estimates we have used.
 
 Let us show now how Lemma~\ref{ker} along with results from Lee~\cite{LeeBilinear}
 can be used to obtain
 Proposition~\ref{key}.   

 \begin{proof}[Proof of Proposition~\ref{key}]
 To be able to prove \eqref{l3} using Lee's bilinear estimates we need to make one more change of variables to isolate what amounts
 to a ``linear direction" for the phase functions in Lemma~\ref{ker}.  In our earlier works on improved spectral projection estimates this was
 done simply by choosing Fermi normal coordinates about the spatial geodesic in \eqref{direct}.  Since the kernels in Lemma~\ref{ker}
 also involve a time variable, we have to deal with our time management problem by working in what amounts to ``Fermi-Schr\"odinger''
 coordinates adapted to the Schr\"odinger tubes that we have described before.  As we shall see, when we use these coordinates we use a simple
 parabolic scaling argument allowing us to apply the main estimate in \cite{LeeBilinear}.  We should also point out that the coordinate
 system we are about to describe is associated to the tube $\iota_\nu$ in \eqref{q1} that is associated with the amplitude
 $a_{\la,\nu}$ of the
 kernel $K^{c_0\theta}_{\lambda,\nu}$ but not the amplitude
 other kernel $K^{c_0\theta}_{\la,\nu'}$ in the lemma.
 
 To describe these coordinates we first recall that, by \eqref{q2}, the last spatial coordinate $x_{n-1}$ measures distance along  the spatial geodesic
 partially defining $\iota_\nu$.  The ``Fermi-Schr\"odinger'' coordinates will preserve the first $(n-2)$ spatial coordinates but involve a linear change of
 variables in the last two coordinates $(x_{n-1},t)$ that takes into account the speeds of the spatial geodesics in \eqref{q1}, i.e., $2\kappa_\ell^{c_0\theta}$, with $\nu=(c_0\theta j,c_0\theta\ell)$ as before and $\kappa_\ell^{c_0\theta}$ as in \eqref{m7}.  
 The ``Schr\"odinger coordinates''  that we employ are the quantum analog of the ``free-fall coordinates" in relativity theory
  described in Manasse and Misner~\cite{Fermi2}.

 We note that if 
 \begin{equation}\label{q08}
 \varphi(x,t;y,s)=\frac{-(d_g(x,y))^2}{4(t-s)},
 \end{equation}
 is the phase function of the kernels in \eqref{q3}, then since we are working in Fermi normal coordinates, we have along our spatial 
 geodesic
 \begin{equation}\label{q8}\frac\partial{\partial x_j}\varphi, \, \, \frac\partial{\partial y_j}\varphi =0
 \, \, \, \text{if } \, \, x=(0,\dots,0,x_{n-1}), \, \,  y=(0,\dots,0,y_{n-1}) 
 \, \, \, \text{and } \, \, j=1,\dots,n-2.
 \end{equation}
 This is not valid, though, for either of the two remaining coordinates $x_{n-1}$ or $t$ that we are currently using.  We need to
 change coordinates so that, in the new variable, we will have the analog of \eqref{q8} for the $(n-1)$-th variable, and, simultaneously,
 have that the phase function is linear in the other remaining variable when restricted to $\iota_\nu$.
 
 Fortunately, this is easy to do.  We simply define new variables $(\tilde x_{n-1},\tilde t)$ via
 \begin{equation}\label{q9}
 (x_{n-1},t)=\tilde t (2\speed, -1)+\tilde x_{n-1}(\kappa_\ell^{c_0\theta},-1)=(2\speed \tilde t+\speed \tilde x_{n-1},
 -\tilde t-\tilde x_{n-1}).
 \end{equation}
 Note then, for later use that
 \begin{equation}\label{q99}
 (\tilde x_{n-1},\tilde t)=-(\speed)^{-1}\cdot \bigl(x_{n-1}+2\speed t, - x_{n-1}-\speed t),
 \end{equation}
 which means that the $\tilde x_{n-1}$ is related to the concentration in \eqref{q5} with $\kappa=\speed$.
 As mentioned before, we shall not change the first $(n-2)$ variables and so to be consistent with our notation, we let
 \begin{equation}\label{q999}
 \tilde x_j=x_j, \quad 1\le j\le n-2.
 \end{equation}

Note that $(\tilde x,\tilde t)$ is on the Schr\"odinger curve $\iota_\nu$ in \eqref{q2} if and only if $\tilde x=0$.  Moreover,
 we claim that our new coordinates fulfill the two additional goals for the behavior of the phase function $\varphi$ in \eqref{q08} on
 $\iota_\nu$. 
 
 So, we need to check that we have the analog of \eqref{q8} for all $j=1,\dots,n-1$, i.e.,
 \begin{equation}\label{q10}
 \nabla_{\tilde x}\varphi, \, \, \nabla_{\tilde y}\varphi =0 \quad \text{if } \, \, \tilde x=\tilde y=0,
 \end{equation}
 as well as
 \begin{equation}\label{q11}
 \varphi(0,\tilde t,0,\tilde s)=(\speed)^2 \cdot (\tilde t-\tilde s).
 \end{equation}
 
 To verify \eqref{q10}, we note that since $\tilde x_j=x_j$, $1\le j \le n-2$, \eqref{q8} yields 
 $\partial \varphi/\partial x_j=0$ and $\partial \varphi/\partial y_j=0$ when $\tilde x=\tilde y=0$ and $1\le j\le n-2$.
 To see that this remains true for $j=n-1$, which gives us the remaining part of \eqref{q10}, we note that, by
 \eqref{q2} and \eqref{q9},
 \begin{equation}\label{q12}
 \varphi(0,\dots,\tilde x_{n-1},\tilde t,0,\dots,0, \tilde y_{n-1},\tilde s)=
 \frac{(\speed)^2}4 \cdot \frac{(2(\tilde t-\tilde s)+(\tilde x_{n-1}-\tilde y_{n-1}))^2}{\tilde t-\tilde s+(\tilde x_{n-1}-\tilde y_{n-1})},
 \end{equation}
 and, consequently, by calculus, we also obtain $\partial\varphi/\partial \tilde x_{n-1}, \, 
 \partial \varphi/\partial \tilde y_{n-1}=0$ when $\tilde x=\tilde y=0$.  Finally, of course \eqref{q12}  yields \eqref{q11} as well, meaning
 that our goals are fulfilled.

 Next, we need to make a couple of more minor modifications to prove \eqref{l3}, which, in the notation of 
 Lemma~\ref{ker}, after a little bit of arithmetic, can be rewritten as follows:
 \begin{equation}\label{q13}
 \bigl\| (T_1H_1)(T_2H_2)\bigr\|_{L^{q/2}_{t,x}}\lesssim_\e
 \la^{-\frac{2n}q +\e} \,  \theta^{-\frac2{n+2}}\, \|H_1\|_{L^2_{t,x}}\|H_2\|_{L^2_{t,x}}, \, \, q=\tfrac{2(n+2)}n,
 \end{equation}
 assuming $H_k(y,s)=0$, $k=1,2$, if $|s|\ge C\log \la$, where
 \begin{equation}\label{q14}
 (T_1H_1)(\tilde x,\tilde t)=\alpha_m(t)\iint 
 %\la^{-\frac{n-1}2} K^{c_0\theta}_\nu(\tilde x,\tilde t; \tilde y, \tilde s) \
e^{i\la \varphi(\tilde x,\tilde t,\tilde y, \tilde s)} a_{\la,\nu}(\tilde x,\tilde t; \tilde y,\tilde s)
\, H_1(\tilde y,\tilde s)\, d\tilde y d\tilde s
\end{equation}
% \\
% =\iint e^{i\la \varphi(\tilde x,\tilde t,\tilde y, \tilde s)} a_{\la,\nu}(\tilde x,\tilde t, \tilde y,\tilde s) \, H_1(\tilde y,\tilde s) \, d\tilde y d\tilde s
% ,
% \end{multline}
 and
%\begin{multline}\label{q15}
% (T_2H_2)(\tilde x,\tilde t)=\alpha_m(t)\iint 
% \la^{-\frac{n-1}2} K^{c_0\theta}_{\nu'}(\tilde x,\tilde t; \tilde y, \tilde s) \, H_2(\tilde y,\tilde s)\, d\tilde y d\tilde s
% \\
% =\iint e^{i\la \varphi(\tilde x,\tilde t,\tilde y, \tilde s)} a_{\la,\nu}(\tilde x,\tilde t, \tilde y,\tilde s) \, H_1(\tilde y,\tilde s) \, d\tilde y d\tilde s.
% \end{multline}
 \begin{equation}\label{q15}
 (T_2H_2)(\tilde x,\tilde t)=\alpha_m(t)\iint 
 %\la^{-\frac{n-1}2} K^{c_0\theta}_\nu(\tilde x,\tilde t; \tilde y, \tilde s) \
e^{i\la \varphi(\tilde x,\tilde t,\tilde y, \tilde s)} a_{\la,\nu'}(\tilde x,\tilde t; \tilde y,\tilde s)
\, H_2(\tilde y,\tilde s)\, d\tilde y d\tilde s.
\end{equation}
We may neglect the $O(\la^{-N})$ errors in Lemma~\ref{ker} since in \eqref{l3} we are supposing that $H_j(s,\cdot)=0$ if $|s|\ge C\log\la$.

 We also of course have
 \begin{multline}\label{q155}
 (T_1H_1 \cdot T_2H_2)(\tilde x,\tilde t)
 = \bigl(\alpha_m(t)\bigr)^2 \times
 \\
 \int e^{i\la (\varphi(\tilde x,\tilde t,\tilde y,\tilde s)+\varphi(\tilde x,\tilde t,\tilde y',\tilde s'))}
 a_{\la,\nu}(\tilde x,\tilde t, \tilde y,\tilde s)
  a_{\la,\nu'}(\tilde x,\tilde t, \tilde y',\tilde s')
  \, H_1(\tilde y,\tilde s) \, H_2(\tilde y',\tilde s') \, d\tilde y d\tilde s d\tilde y' d\tilde s'.
  \end{multline}

  Note that by \eqref{q6}, \eqref{qtube} and \eqref{q99} we have that $a_{\la,\mu}(\tilde x,\tilde t, \tilde y,\tilde s)=0$, $\mu=\nu,\nu'$ if
  $|(\tilde x_1,\dots,\tilde x_{n-2})|\ge C_1\theta$, $|(\tilde y_1,\dots,\tilde y_{n-2})|\ge C_1\theta$ 
  or $|\tilde x_{n-1}-\tilde y_{n-1}|\ge C_1\theta$.
As a result, in order to prove \eqref{q13}, it suffices to control the left side when the norm
  is taken over sets where $|\tilde x-(0,\dots,0,r)|\le C_2\theta$, with $C_2$ fixed, and so, since we may take $r$ to be $0$, we have reduced matters to showing that
  for sufficiently small $\theta$ we have with $C_3\approx C_2$,
 \begin{multline}\label{q13'}
 \bigl\| (T_1H_1)(T_2H_2)\bigr\|_{L^{q/2}_{t,x}(\{|\tilde x|\le C_3\theta\}\times [-1,1])}
 \\
 \lesssim_\e
 \la^{-\frac{2n}q +\e} \,  \theta^{-\frac2{n+2}}\, \|H_1\|_{L^2_{\tilde t,\tilde x}}\|H_2\|_{L^2_{\tilde t,\tilde x}}, \, \, q=\tfrac{2(n+2)}n,
 \end{multline}
 assuming, as above, that  $H_k(y,s)=0$, $k=1,2$, if $|s|\ge C\log \la$.
 
 Next, we note that by \eqref{q4}, \eqref{q9} and \eqref{q99} we have that if we use the parabolic scaling
 $(\tilde x,\tilde t)\to (\theta\tilde x,\tilde t)$ then
 \begin{equation}\label{q16}
 D^\beta_{\tilde x,\tilde t,\tilde y,\tilde s} a_{\la,\mu}(\theta \tilde x, \tilde t,\theta \tilde y, \tilde s)=O_\beta(1).
 \end{equation}
 This is clear for $\mu=\nu$ since then $2\kappa^{c_0\theta}_\ell \tfrac\partial{\partial x_{n-1}}-\tfrac\partial{\partial t}$
 corresponds to $\tfrac\partial{\partial \tilde t}$, and the bounds also hold for $\mu=\nu'$ since $\kappa_\ell^{c_0\theta}-\kappa_{\ell'}^{c_0\theta}
  =O(\theta)$.  
  Also note that the dilated amplitude in \eqref{q16}
  is $O(\la^{-N})$ when $|\tilde x|$ or $|\tilde y|$ is larger than a fixed constant.
  
  The phase function $\varphi(\tilde x,\tilde t,\tilde y, \tilde s)$ does not quite satisfy the bounds in \eqref{q16}; however, it is straightforward to 
  remedy this if we recall that we constructed our Fermi-Schr\"odinger coordinates so that \eqref{q10} and \eqref{q11} would be valid.  As a result
  \begin{equation}\label{q17}
  \tilde \varphi(\tilde x,\tilde t, \tilde y,\tilde s)=\varphi(\tilde x,\tilde t,\tilde y,\tilde s)-\bigl(\kappa^{c_0\theta}_\ell\bigr)^2 \, (\tilde t-\tilde s)
  \end{equation}
  vanishes to second order when $\tilde x=0$ and $\tilde y=0$.  This means that, after the above parabolic scaling, we actually have
  \begin{equation}\label{q18}
   D^\beta_{\tilde x,\tilde t,\tilde y,\tilde s}\bigl( \theta^{-2}\tilde \varphi(\theta\tilde x,\tilde t,\theta\tilde y,\tilde s)\bigr)=O_\beta(1)
   \quad \text{if } \, \, |\tilde x|, |\tilde y|=O(1).
   \end{equation}
   
   Clearly, in order to prove \eqref{q13'} we may replace $\varphi$ by $\tilde \varphi$.  Also, by Minkowski's inequality and the Schwarz inequality,
   if we define the ``frozen'' bilinear oscillatory integral operators
   \begin{multline}\label{q19}
   B_{\la,\nu,\nu'}^{\tilde s, \tilde s'}(h_1,h_2)(x,t)
   =(\alpha_m(t))^2 \times
   \\
    \iint e^{i\la (\tilde \varphi(\tilde x,\tilde t,\tilde y,\tilde s)+\tilde\varphi(\tilde x,\tilde t,\tilde y',\tilde s'))}
 a_{\la,\nu}(\tilde x,\tilde t, \tilde y,\tilde s)
  a_{\la,\nu'}(\tilde x,\tilde t, \tilde y',\tilde s')
  \, h_1(\tilde y)\, h_2(\tilde y') \, d\tilde y  d\tilde y',
  \end{multline}
  then it suffices to prove that
   \begin{multline}\label{q20}
 \bigl\|    B_{\la,\nu,\nu'}^{\tilde s, \tilde s'}(h_1,h_2)\bigr\|_{L^{q/2}_{t,x}(\{|\tilde x|\le C_3\theta\}\times [-1,1])}
 \\
 \lesssim_\e
 \la^{-\frac{2n}q +\e} \,  \theta^{-\frac2{n+2}}\, \|h_1\|_{L^2_{\tilde x}}\|h_2\|_{L^2_{\tilde x}}, \, \, q=\tfrac{2(n+2)}n.
 \end{multline}
 
 Note that $B_{\la,\nu,\nu'}^{\tilde s, \tilde s'}(h_1,h_2)$ factors as the product of two oscillatory integral operators involving the $(\tilde x,\tilde t, \tilde y)$
 variables.  The two phase functions are
 \begin{equation}\label{q21}
 \phi_{\tilde s}(\tilde x,\tilde t;\tilde y)=\tilde\varphi(\tilde x,\tilde t,\tilde y,\tilde s) \, \, \text{and } \, \,
  \phi_{\tilde s'}(\tilde x,\tilde t;\tilde y)=\tilde\varphi(\tilde x,\tilde t,\tilde y,\tilde s').
  \end{equation}
  
  In order to apply the bilinear results in \cite{LeeBilinear} we need to collect a few facts about the support
  properties of the amplitudes of the bilinear oscillatory integrals in \eqref{q19} which are straightforward
  consequences of Lemma~\ref{ker}.
  
  \begin{lemma}\label{bila}
  Let $\delta<1/8$ as in \eqref{22.6} be given.  Then we can fix $c_0>0$ in \eqref{l2} so that there are 
  constants $c_\delta, C_\delta\in (0, \infty)$ so that for sufficiently small $\theta$ and $|\tilde x|\le C_0\theta$,
  with $C_0$ fixed, we have
  \begin{equation}\label{q22}
  \text{if } \,  \, a_{\la,\nu}(\tilde x,\tilde t; \tilde y,\tilde s)\cdot a_{\la,\nu'}(\tilde x,\tilde t; \tilde y', \tilde s')\ne 0,
  \\
  \text{then } \, \, |\tilde y|, \, |\tilde y'|\le C_\delta\theta, \, \, \text{and } \, \, 
  |\tilde y-\tilde y'|\ge c_\delta\theta.
  \end{equation}
  Additionally, if $\delta_0<1/8$ as in \eqref{22.6} is small enough, then for sufficiently small  $\theta$ we have
  \begin{multline}\label{q22'}
  \text{if } \, a_{\la,\nu}(\tilde x,\tilde t; \tilde y, \tilde s)\ne 0 \, \, 
  \text{then } \, |\delta-2\kappa_\ell^{c_0\theta}(\tilde t-\tilde s)|\le 4\delta_0\delta, 
  \\
   \text{and } \, \, 
     \text{if } \, a_{\la,\nu'}(\tilde x,\tilde t; \tilde y, \tilde s')\ne 0 \, \, 
  \text{then } \, |\delta-2\kappa_\ell^{c_0\theta}(\tilde t-\tilde s')|\le 4\delta_0\delta,
  \end{multline}
  \end{lemma}

\begin{proof}
 The first assertion in \eqref{q22} about the size of $\tilde y$ and $\tilde y'$ follows trivially from \eqref{qtube}, \eqref{q99} and \eqref{q999}.
     To see the assertion regarding the important separation of the $\tilde y$-variables, recall that $\nu, \nu' \in c_0\theta\cdot\mathbb{Z}^{2n-3}$, and by \eqref{sep'}, $|\nu-\nu'|\in [\frac14\theta, 4^n\theta]$. Thus, if we write
   $\nu=(c_0\theta j,c_0\theta\ell)$ and $\nu'=(c_0\theta j',c_0\theta\ell')$, we can divide into the following two cases:

   (i) $|j-j'|\ge \frac18$. In this case,  the spatial parts, $\overline\gamma_j$ and $\overline\gamma_{j'}$ of the 
   Schr\"odinger curves $\iota_\nu$ and $\iota_{\nu'}$ have angle $\approx \theta$.  By \eqref{q6} if the product of the amplitudes in
   \eqref{q22} is nonzero then we must have in our original Fermi normal coordinates that, for a fixed constant $C_1'$,
   $x\in {\mathcal T}_{C_1'c_0\theta}(\overline\gamma_j) \cap{\mathcal T}_{C_1'c_0\theta}(\overline\gamma_{j'})$,
   $y\in {\mathcal T}_{C_1'c_0\theta}(\overline\gamma_j)$ and $y'={\mathcal T}_{C_1'c_0\theta}(\overline\gamma_{j'})$.
   Here, of course, ${\mathcal T}_{r}(\overline\gamma)$ denotes an $r$-tube about $\overline\gamma$ in $M^{n-1}$.
   By \eqref{q7} we must also have $d_g(x,y),d_g(x,y')\in [\delta-\delta_0\delta,\delta+\delta_0\delta]$ for our small $\delta_0>0$ if the 
   product is nonzero.  Since we are assuming \eqref{sep'} the two tubes of width $\approx c_0\theta$ intersect at
   angle $\approx \theta$ at $(x,t)$, which implies that in our original Fermi normal
   coordinates $|(y_1,\dots,y_{n-2})-(y_1',\dots,y_{n-2}')|\approx \theta$ if the above product is nonzero and 
   $c_0$ and $\theta$ are small.  By \eqref{q999}, this yields  the assertion in \eqref{q22} about the separation of $\tilde y$ and $\tilde y'$
   under our assumption that $j\ne j'$.  Note that the smaller $\delta$ becomes we have to choose $c_0$ to be correspondingly small, but we are
   assuming here that $\delta$ is fixed (as we shall do later).

 (ii)  $|\ell-\ell'|\ge \frac18$. In this case we have $|\ell-\ell'|\approx 1$.
%  In this case the two Sch\"odinger tubes both involve the common spatial geodesic $\overline\gamma_j$ about
% which we constructed our Fermi normal coordinates but with different speeds.  So, if 
% $\theta$ is small enough, by \eqref{q6} and the triangle inequality,
% we must have in these coordinates  
Recall that in our Fermi normal coordinates, we have
 \begin{multline}\label{taylor1}
     d_g(x,y)=|x_{n-1}-y_{n-1}|,
 \, \, \,  \frac\partial{\partial x_j}d_g(x,y), \, \, \frac\partial{\partial y_j}d_g(x,y)=0,  \\
\text{if} \, \, x=(0,\dots,0,x_{n-1}), \, \,  y=(0,\dots,0,y_{n-1}), \text{and } \, \, \, j=1,\dots,n-2.
 \end{multline}
Also we know that by \eqref{qtube}
 \begin{equation}\label{taylor2}
|(x_1,\dots,x_{n-2})|+
 |(y_1,\dots,y_{n-2})|
\le C_0 \theta,\quad \text{if } \, a_{\la,\nu} a_{\la,\nu'}\ne 0.
 \end{equation}
Since the function $d_g(x,y)$ is smooth when $d_g(x,y)\approx \delta$, by \eqref{taylor1}, \eqref{taylor2} and Taylor's expansion, we have
$$|d_g(x,y)-(x_{n-1}-y_{n-1})|, \, |d_g(x,y')-(x_{n-1}-y_{n-1}')| \le C_\delta\theta^2, \quad \text{if } \, a_{\la,\nu} a_{\la,\nu'}\ne 0,$$
since by \eqref{q7} both of  the amplitudes vanish if $x_{n-1}-y_{n-1}<0$. If we let $\theta$ to be small enough, $C_\delta\theta^2$ is much smaller than $c_0\theta$, 
consequently, if $|(x_{n-1}-y_{n-1})+2\kappa_\ell^{c_0\theta}(t-s)|\ge C'_0 c_0\theta$ with
$C_0'$ large enough we must have that
$|d_g(x,y)+2\kappa^{c_0\theta}_\ell(t-s)|\ge C_0c_0\theta$ with $C_0$ as in \eqref{q5}, which means
that $a_\nu=0$ if $|(x_{n-1}-y_{n-1})+2\kappa_\ell^{c_0\theta}(t-s)|\ge C'_0 c_0\theta$ for this choice of $C'_0$ (which is independent of $c_0$).
We similarly have $a_{\nu'}=0$ if
 $|(x_{n-1}-y'_{n-1})+2\kappa_{\ell'}^{c_0\theta}(t-s)|\ge C'_0 c_0\theta$.  By \eqref{q99} this means that
 for a uniform constant $C_1$ if $a_\nu\ne0$ we must have $|\tilde x_{n-1}-\tilde y_{n-1}|\le C_1c_0\theta$, and if
 $a_{\nu'}\ne 0$ we must have $|( \tilde x_{n-1}-\tilde y_{n-1}')-2(\kappa_{\ell}^{c_0\theta})^{-1}(\kappa^{c_0\theta}_{\ell'}-\kappa^{c_0}_{\ell})(t-s)|\le
 C_1c_0\theta$.  Since \eqref{q5} and \eqref{q7} imply that $-(t-s)\sim \delta$ on the support of the amplitudes and thus
 $|2(\kappa_{\ell}^{c_0\theta})^{-1}(\kappa^{c_0\theta}_\ell-\kappa^{c_0\theta}_{\ell'})(t-s)|$ must be larger than a fixed multiple of $\theta$ if
$|\ell-\ell'|\approx 1$ and $a_{\la,\nu} \cdot a_{\la,\nu'}\ne 0$.   So, in this case, if $c_0$ is small enough, we must have $|\tilde y_{n-1}-\tilde y_{n-1}'|\approx \theta$ if 
 $a_{\la,\nu}\cdot a_{\la,\nu'}\ne0$,
 which finishes the proof of the first assertion regarding the separation of $\tilde y$ and $\tilde y'$ in \eqref{q22}
 .
% , since by \eqref{qtube} we have $|(\tilde y_1,\dots, \tilde y_{n-2})|
% +|(\tilde y_1',\dots, \tilde y_{n-2}')|=O(\theta)$ if the product is nonzero.

% To prove the assertion in \eqref{q22} regarding $|\tilde s-\tilde s'|$ we note that, by \eqref{qtube} and our observation that $-(t-s)\approx \delta$
% on the support of the amplitudes, we must have that $x_{n-1}-y_{n-1}\approx 1$ if they are nonzero.  Thus, if $\theta$ is small 
% \eqref{qtube} also yields
% $$|(x_{n-1}-y_{n-1})-d_g(x,y)|+|(x_{n-1}-y'_{n-1})-d_g(x,y')|=O(\theta)
% \quad \text{if } \, \, a_{\la,\nu} a_{\la,\nu'}\ne 0.$$
%   Thus, by \eqref{q7},
% $|(x_{n-1}-y_{n-1})-\delta|, \, |(x_{n-1}-y_{n-1}')-\delta|\le 2\delta_0\delta+C\theta$ if the product is nonzero, and, hence if $\theta$ is small enough and
% $a_{\la,\nu} a_{\la,\nu'}\ne0$ we must have $|y_{n-1}-y_{n-1}'|\le 5\delta_0\delta$.  Therefore, by the last part of \eqref{qtube} we must also
% have $|s-s'|=O(\delta_0)$ if $\theta$ is small and the product is non-zero, which leads to the last part of \eqref{q22} by \eqref{q99}.
It remains to prove \eqref{q22'}.  Since we are assuming $|\tilde x|\le C_0\theta$, it follows from the first part of \eqref{q22} that 
both of the amplitudes in \eqref{q22'} will vanish if we do not have $|\tilde y|, \, |\tilde y'|=O(\theta)$ and hence
\begin{equation}\label{qh}
a_{\la,\nu}(\tilde x,\tilde t; \tilde y,\tilde s)=a_{\la,\nu'}(\tilde x,\tilde t; \tilde y, \tilde s)=0
\quad \text{if } \, \, \, |\tilde x_{n-1}-\tilde y_{n-1}|\ge C'\theta,
\end{equation}
for some constant $C'$.  By \eqref{q99}
$$(\tilde x_{n-1}-\tilde y_{n-1},\tilde t-\tilde s)=
-(\kappa_\ell^{c_0\theta})^{-1}\bigl(x_{n-1}-y_{n-1}+2\kappa^{c_0\theta}_\ell(t-s), \, -(x_{n-1}-y_{n-1})-\kappa^{c_0\theta}_\ell(t-s))\bigr),
$$
and so $(\tilde x_{n-1}-\tilde y_{n-1})+(\tilde t-\tilde s)=-(t-s)$.  Since $|\ell-\ell'|\approx 1$ implies $\kappa^{c_0\theta}_\ell-
\kappa^{c_0\theta}_{\ell'}=O(\theta)$, by \eqref{q5} and \eqref{qh}, we conclude that both amplitudes vanish if we do not have
$$|d_g(x,y)-2\kappa_\ell^{c_0\theta}(\tilde t-\tilde s)| \le C'\theta,$$
for some uniform constant $C'$.  By the first part of \eqref{q7} we obtain \eqref{q22'} if $\theta$ is sufficiently small.
\end{proof}

Let us now prove the bilinear oscillatory integral estimates \eqref{q20} which will finish the proof of 
Proposition~\ref{key}.

To prove \eqref{q20}, in addition to following the proof of \cite[Theorem 1.3]{LeeBilinear}, we shall also follow related arguments
of two of us \cite{BlairSoggeRefined} which proved analogous bilinear estimates in the $1+2$ dimensional setting (one lower dimension than here)
using the simpler classical bilinear oscillatory integral estimates implicit in H\"ormander~\cite{HormanderFLP}.  Similar arguments were in
the paper \cite{blair2015refined} by these two authors.

Just as in \cite{LeeBilinear} we first perform a parabolic scaling as in \eqref{q16} and \eqref{q18} to be able to apply the main
estimate, Theorem 1.1, in Lee~\cite{LeeBilinear}.  So for small $\la^{-1/8}\le \theta \ll 1$, we let 
\begin{equation}\label{q244}
\phi^\theta_{\tilde s}(\tilde x,\tilde t; \tilde y)=\theta^{-2}\tilde \varphi(\theta \tilde x,\tilde t; \theta \tilde y, \tilde s),
\quad
\text{and } \, \, 
\phi^\theta_{\tilde s'}(\tilde x,\tilde t; \tilde y')=\theta^{-2}\tilde \varphi(\theta \tilde x,\tilde t; \theta \tilde y', \tilde s').
\end{equation}
and corresponding amplitudes
\begin{equation}\label{q255}
a^\theta_{\la,\nu}(\tilde x,\tilde t; \tilde y,\tilde s)=a_{\la,\nu}(\theta x,\tilde t; \theta \tilde y,\tilde s)
\quad \text{and } \, \, 
a^\theta_{\la,\nu'}(\tilde x,\tilde t; \tilde y,\tilde s)=a_{\la,\nu}(\theta x,\tilde t; \theta \tilde y',\tilde s').
\end{equation}
Then, as we noted before
$$D^\beta_{\tilde x,\tilde t,\tilde y}a^\theta_{\la,\mu}=O_\beta(1), \, \mu=\nu,\nu' \quad \text{and } \, \,
D^\beta_{\tilde x,\tilde t,\tilde y}\phi_j=O_\beta(1), \, \, \phi_1=\phi^\theta{\tilde s}, \, \, \phi_2=\phi^\theta_{\tilde s'}.
$$

By \eqref{q22} and \eqref{q22'} we also have the key separation properties for small enough $\theta$
\begin{multline}\label{q27} 
\text{if } \, \, a^\theta_{\la,\nu}(\tilde x,\tilde t; \tilde y,\tilde s)a^\theta_{\la,\nu'}(\tilde x,\tilde t; \tilde y',\tilde s')\ne 0
\\
\text{then } \, |\tilde y|, |\tilde y'|=O(1), \, \, |\tilde y-\tilde y'|\ge c_\delta \, \,
\text{and } \, \, |\tilde s-\tilde s'|\le 8\delta_0\delta,
\end{multline}
%for some fixed constant $C$, 
with $\delta$ and $\delta_0$ as in \eqref{22.6}.

Additionally, by a simple scaling argument, our remaining task of the section, \eqref{q20}, is equivalent to the following
for small enough $\theta$:
\begin{equation}\label{q28}
\bigl\| B^{\theta,\tilde s,\tilde s'}_{\la,\nu,\nu'}(h_1,h_2)\bigr\|_{L^{q/2}_{t,x}(\{|\tilde x_0|\le C_1\}\times [-1,1])}
\lesssim_\e \bigl( \la\theta^2 \bigr)^{-\frac{2n}q+\e}\|h_1\|_{L^2_x}\|h_2\|_{L^2_x}, \, q=\tfrac{2(n+2)}n,
\end{equation}
where we have the scaled version of \eqref{q19}, i.e.,
\begin{multline}\label{q29}
B^{\mu,\tilde s,\tilde s'}_{\la,\nu,\nu'}(h_1,h_2)(x,t)=(\alpha_m(t))^2\times
\\
\iint e^{i(\la\theta^2)[\phi^\theta_{\tilde s}(\tilde x,\tilde t;\tilde y)+\phi^\theta_{\tilde s'}(\tilde x,\tilde t; \tilde y')]}
a^\theta_{\la,\nu}(\tilde x,\tilde t; \tilde y,\tilde s)a^\theta_{\la,\nu'}(\tilde x,\tilde t; \tilde y',\tilde s') \, h_1(\tilde y) 
h_2(\tilde y')\, d\tilde y d\tilde y'.
\end{multline}

To prove this, let us see how we can use our earlier observation that \eqref{q10} and \eqref{q17}
implies that $\tilde \varphi$ vanishes to second order when $(\tilde x,\tilde y)=(0,0)$ to see that the
scaled phase functions in \eqref{q244} closely resemble Euclidean ones if $\theta$ is small which will allow us to 
verify the hypotheses of Lee's bilinear oscillatory integral theorem \cite[Theorem 1.3]{LeeBilinear} if $\delta,\delta_0>0$
in \eqref{22.6} are fixed small enough.

To be more specific, let  
\begin{multline}\label{v1}
A(\tilde t,\tilde s)=\frac{\partial^2\tilde \varphi}{\partial \tilde y_j \partial \tilde y_k}(0,\tilde t;0,\tilde s), \, \, \,
B(\tilde t,\tilde s)=\frac{\partial^2\tilde \varphi}{\partial \tilde x_j \partial\tilde y_k}(0,\tilde t;0,\tilde s),
\\ 
\text{and } \, \, C(\tilde t,\tilde s)=\frac{\partial^2\tilde \varphi}{\partial \tilde x_j \partial \tilde x_k}(0,\tilde t;0,\tilde s).
\end{multline}
Then the Taylor expansion about $(\tilde x,\tilde y)=(0,0)$  of $\tilde \varphi$ is
\begin{equation}\label{v2}
\tilde \varphi(\tilde x,\tilde t; \tilde y, \tilde s)=
\tfrac12 \tilde y^TA(\tilde t,\tilde s)\tilde y+ \tilde x^TB(\tilde t,\tilde s)\tilde y+\tfrac12\tilde x^TC(\tilde t,\tilde s)\tilde x
+r(\tilde x,\tilde t; \tilde y,\tilde s),
\end{equation}
where $r(\tilde x,\tilde t; \tilde y, \tilde s)$ vanishes to third order at $(\tilde x,\tilde y)=(0,0)$.  So,
\begin{equation}\label{v3}
D^\beta_{\tilde x,\tilde y,\tilde t,\tilde s}r^\theta(\tilde x,\tilde t;\tilde y,\tilde s)=O(\theta), \quad 
\text{if } \, \, r^\theta(\tilde x, \tilde t; \tilde y,\tilde s) = \theta^{-2} r(\theta\tilde x, \tilde t; \theta\tilde y,\tilde s),
\end{equation}
which means that $r^\theta\to 0$ in the $C^\infty$ topology as $\theta\to 0$.

To use \eqref{v2} we shall use parabolic scaling and  the following lemma, whose proof we
postpone until the end of this subsection, which says that if $\delta, \delta_0>0$ in \eqref{22.6}
are small enough then the phase functions $\phi_{\tilde s}$ and $\phi_{\tilde s'}$ in \eqref{q21}
satisfy the Carleson-Sj\"olin condition (see \cite[\S2.2.2]{SFIO2} and \cite{steinbeijing}).

\begin{lemma}\label{ABlemma}  Let $A(\tilde t,\tilde s)$ and $B(\tilde t,\tilde s)$ be as in \eqref{v1}.
Then if $\delta,\delta_0>0$ in \eqref{22.6} are small enough
\begin{equation}\label{v4}
\det  B(\tit,\tis)=\det \frac{\partial^2 \tilde \varphi(0,\tilde t; 0,\tilde s)}{\partial \tix_j\partial \tiy_k}\ne 0, \quad
\text{if } \, \, a^\theta_{\la,\nu}\cdot a^\theta_{\la,\nu'}\ne 0.
\end{equation}
Furthermore, on the support of $ a^\theta_{\la,\nu}\cdot a^\theta_{\la,\nu'}$, $-(\tfrac\partial{\partial \tilde t}A(\tilde t,\tilde s))^{-1}
=-(\tfrac\partial{\partial \tilde t} \tfrac{\partial^2 \tilde \varphi  }{\partial \tilde y_j\partial \tilde y_k}(0,\tilde t; 0,\tilde s))^{-1}$ is positive definite, i.e.,
\begin{equation}\label{v13}
\xi^t\Bigl(-\frac{\partial}{\partial \tit}A(\tit,\tis)\Bigr)^{-1}\xi, \, \, \xi^t\Bigl(-\frac{\partial}{\partial \tit}A(\tit,\tis') \Bigr)^{-1}\xi 
\ge c_\delta|\xi|^2, \quad \text{if } \, a^\theta_{\la,\nu}\cdot a^\theta_{\la,\nu'}\ne 0,
\end{equation}
and also
\begin{equation}\label{v13'}
\Bigl|\frac\partial{\partial \tilde t}A(\tilde t,\tilde s)\xi \Bigr|\ge c_\delta |\xi| , \, \Bigl|\frac\partial{\partial \tilde t}A(\tilde t,\tilde s')\xi \Bigr|\ge c_\delta |\xi|
\, \,  \, \text{if } \, \, a^\theta_{\la,\nu}\cdot a^\theta_{\la,\nu'}\ne 0,
\end{equation}
\end{lemma}

Let us use \eqref{v2} and \eqref{v3} and this lemma to see that we can obtain our remaining estimate \eqref{q28} via Lee's
\cite[Theorem 1.1]{LeeBilinear}.  As we shall see, it is crucial for us that $-\frac{\partial}{\partial \tilde t}A(\tilde t,\tilde s)$
is positive definite.

Note that, in addition to the $\theta$ parameter, \eqref{q28} also involves the
$(\tis,\tis')$ parameters.  For simplicity, let us first see how Lee's result yields \eqref{q28} in the case where these
two parameters agree, i.e. $\tis=\tis'$.  We then will  argue that if $\delta_0$ in \eqref{22.6} and hence \eqref{q22} is fixed small enough we can also handle the case
where $\tis\ne \tis'$. 

To do this we first note that the parabolic scaling in \eqref{v3}, which agrees with that in \eqref{q244}, preserves the first three terms in
the right of \eqref{v2} since they are quadratic.  Also, in proving \eqref{q29}, we may subtract $\tfrac12 \tix^tC(\tit,\tis)\tix$ from
$\phi^\theta_{\tis}$ and $\tfrac12 \tix^tC(\tit,\tis')\tix$ from $\phi^\theta_{\tis'}$ as these quadratic terms do not involve $\tiy$.  We point out that
this trivial reduction also works if $\tis\ne \tis'$.

Next, we note that, by \eqref{v4} and our temporary assumption that $\tilde s=\tilde s'$, after making a linear change of variables in
$\tix$ (depending on $\tit,\tis$), we may reduce to the case where $B(\tit,\tis)=I_{n-1}$, the $(n-1)\times (n-1)$ identity matrix.
This means for the case where $\tis =\tis'$ we have reduced matters to showing that \eqref{q28} is valid where
\begin{multline}\label{v5}
\phi^\theta_{\tis}(\tix,\tit;\tiy)=\langle \tix,\tiy\rangle + \frac12 \sum_{j,k=1}^{n-1} \frac{\partial^2\tilde \varphi}{\partial \tiy_j \partial \tiy_k}(0,\tit; 0, \tis)
\tiy_j \tiy_k + \tilde r^\theta(\tix,\tit; \tiy, \tis)
\\
=\langle \tilde x,\tilde y\rangle +\tilde y^t A(\tilde t,\tilde s) \tilde y+\tilde r^\theta(\tilde x,\tilde t; \tilde y,\tilde s)
,
\end{multline}
with $\tilde r^\theta$ denoting $r^\theta$ rewritten in the new $\tix$ variables coming from $B(\tit,\tis)$.  In view of 
\eqref{v4}, \eqref{v3} remains valid for $\tilde r^\theta$.
For later use, we note that if we change variables according to $\tilde s$ as above, then for $\tis'$ near $\tis$ we have
for
\begin{equation}\label{v5.1}
B(\tit,\tis,\tis')=(B(\tit,\tis'))^t \, ((B(\tit,\tis)^{-1})^t=I_{n-1}+O(|\tis-\tis'|),
\end{equation}
\begin{multline}\label{v5.2}
\phi^\theta_{\tis'}(\tix,\tit,\tiy)=\langle x, B(\tit,\tis,\tis')\tiy\rangle+ \frac12 \sum_{j,k=1}^{n-1} \frac{\partial^2\tilde \varphi}{\partial \tiy_j \partial \tiy_k}(0,\tit; 0, \tis')
\tiy_j \tiy_k + \tilde r^\theta(\tix,\tit; \tiy, \tis)
\\
=\phi^\theta_{\tis}(\tix,\tit;\tiy)+O(|\tis-\tis'|).
\end{multline}

We fix $\delta$ and $\delta_0$ in \eqref{22.6} so that the conclusions of Lemmas \ref{ker} and \ref{ABlemma} are valid.
We can  also now fix also finally fix $c_0$ so that the results of Lemma~\ref{ker} and 
Lemma~\ref{bila} are valid.   If we only had to treat the case where $\tilde s=\tilde s'$ in \eqref{q28} then the above
choice of $\delta_0$ would suffice; however, as we shall momentally see, to handle the cases where
$\tilde s\ne \tilde s'$ we shall need to choose $\delta_0$ small enough so that we can exploit the last part
of \eqref{q27}.

Let us now verify that we can apply \cite[Theorem 1.3]{LeeBilinear} to obtain \eqref{q28} for sufficiently small $\theta$.
%with $\phi^\theta_{\tilde s}, \phi^{\theta}_{\tilde s'}$ as in \eqref{v5}
This would complete the proof of Proposition~\ref{key}.
We recall that we are assuming for now that $\tis=\tis'$ and that we have reduced to the case
where $B(\tit,\tis)=I_{n-1}$ and $C(\tit,\tis)=0$ in \eqref{v2} and so
\begin{equation}\label{v14}
\phi^\theta_{\tis}(\tix,\tit;\tiy)=\langle \tix,\tiy\rangle +\tfrac12 \tiy^t A(\tit,\tis)\tiy +\tilde r^\theta(\tix,\tit; \tiy,\tis), 
\end{equation}
with $\tilde r^\theta$ satisfying the bounds in \eqref{v3}.   

By \eqref{v3} and \eqref{v14} we have
\begin{equation}\label{v15}
\frac{\partial \phi^\theta_{\tis}}{\partial \tix}(\tilde x,\tilde t; \tilde y)=\tilde y
+\frac{\partial \tilde r^\theta}{\partial \tix}(\tix,\tit; \tiy, \tis)
=\tiy +\e(\theta,\tix,\tit,\tis;\tiy), 
\end{equation}
where $\tiy\to \e(\cdot; \, \tiy)$ and its derivatives are $O(\theta)$.  Thus, for small enough $\theta$,
the inverse function also satisfies
\begin{equation}\label{v16}
\tiy\to \bigl(\frac{\partial \phi^\theta_{\tis}}{\partial \tix}(\tix, \tit; \, \cdot\, )\bigr)^{-1}(\tiy)
=\tiy +\tilde \e(\theta,\tix,\tis;\tiy), 
\end{equation}
where
\begin{equation}\label{v17}
D^\beta_{\tiy}\tilde \e(\theta,\tix,\tit,\tis;\tiy)=O_\beta(\theta).
\end{equation}

Define, in the notation of \cite{LeeBilinear},
\begin{multline}\label{v18}
q^\theta_s(\tix,\tit;\tiy)=\tfrac{\partial}{\partial \tit}\phi^\theta_s\bigl(\tix,\tit; \bigl( \tfrac{\partial \phi^\theta_s}{\partial \tix}
(\tix,\tit; \, \cdot\, )\bigr)^{-1}(\tiy)\bigr)
\\
=\bigl(\tfrac\partial{\partial \tit}\phi^\theta_s\bigr)\bigl(\tix,\tit; \tiy+\tilde \e(\theta,\tix,\tit,s;\tiy)\bigr), \quad
s=\tis, \tis',
\end{multline}
as well as
\begin{equation}\label{v19}
\delta^\theta_{\tis,\tis'}(\tix,\tit; \tiy, \tiy')
= \partial_{\tiy}q^\theta_{\tis}(\tix,\tit; \partial_{\tix}\phi^\theta_{\tis}(\tix,\tit;\tiy))
- \partial_{\tiy}q^\theta_{\tis'}(\tix,\tit; \partial_{\tix}\phi^\theta_{\tis'}(\tix,\tit;\tiy')).
\end{equation}
Even though we are assuming for the moment that $\tis=\tis'$ these two quantities will be needed for 
$\tis\ne\tis'$ as well to be able to use \cite[Theorem 1.1]{LeeBilinear} to obtain \eqref{q28}.

Then \cite[(1.4)]{LeeBilinear}, the conditions to ensure the bounds \eqref{q28}, are
\begin{multline}\label{v20}
\bigl| \bigl\langle 
\partial^2_{\tix\tiy}\phi^\theta_{\tis}(\tix,\tit;\tiy) \delta_{\tis,\tis'}^\theta, \, 
%\delta^\theta_{\tis,\tis'}(\tix,\tit; \tiy, \tiy'), 
\bigl[ \, \partial^2_{\tix,\tiy}\phi^\theta_{\tis}(\tix,\tit; \tiy)\, \bigr]^{-1}
\bigl[ \, \partial^2_{\tiy\tiy}q^\theta_{\tis}(\tix,\tit; \partial_{\tix}\phi^\theta_{\tis}(\tix,\tit;\tiy))\, \bigr]^{-1}
%\delta^\theta_{\tis,\tis'}(\tix,\tit; \tiy, \tiy')
\delta_{\tis,\tis'}^\theta
\bigr\rangle\bigr| 
>0,
\\
 \quad \delta_{\tis,\tis'}^\theta=\delta^\theta_{\tis,\tis'}(\tix,\tit; \tiy, \tiy'), \quad
 \text{on } \, \, \, \text{supp }(a_{\la,\nu}\cdot a_{\la,\nu'}),
\end{multline}
as well as
\begin{multline}\label{v21}
\bigl| \bigl\langle 
\partial^2_{\tix\tiy}\phi^\theta_{\tis'}(\tix,\tit;\tiy') \delta_{\tis,\tis'}^\theta, 
%\delta^\theta_{\tis,\tis'}(\tix,\tit; \tiy, \tiy'), 
\bigl[  \partial^2_{\tix,\tiy}\phi^\theta_{\tis'}(\tix,\tit; \tiy')\, \bigr]^{-1}
\bigl[ \, \partial^2_{\tiy\tiy}q^\theta_{\tis'}(\tix,\tit; \partial_{\tix}\phi^\theta_{\tis'}(\tix,\tit;\tiy'))\bigr]^{-1}
%\delta^\theta_{\tis,\tis'}(\tix,\tit; \tiy, \tiy')
\delta_{\tis,\tis'}^\theta
\bigr\rangle\bigr| 
>0,
\\
 \quad \delta_{\tis,\tis'}^\theta=\delta^\theta_{\tis,\tis'}(\tix,\tit; \tiy, \tiy'), \quad
 \text{on } \, \, \, \text{supp }(a_{\la,\nu}\cdot a_{\la,\nu'}).
\end{multline}

Note that by \eqref{v3}, \eqref{v5}, \eqref{v12}, \eqref{v17} and \eqref{v18} for small $\theta$ we have
%\begin{equation}\label{v22}
%\bigl(q^\theta_s(\tix,\tit;\tiy)\bigr)^{-1}=
%\bigl(\tfrac{\partial A}{\partial \tit}(\tit,s)\bigr)^{-1}+O(\theta), \quad s=\tis, \tis',
%\end{equation}
\begin{equation}\label{v22}
\bigl(q^\theta_{\tis}(\tix,\tit;\tiy)\bigr)^{-1}=
\bigl(\tfrac{\partial A}{\partial \tit}(\tit,\tis)\bigr)^{-1}+O(\theta),
\end{equation}
and also, by \eqref{v15} and \eqref{v16},
\begin{equation}\label{v23}
\partial^2_{\tix,\tiy}\phi^\theta_{\tis}(\tix,\tit;\tiy)=I_{n-1}+O(\theta), 
\end{equation}
as well as
\begin{equation}\label{v24}
\bigl(\partial^2_{\tix,\tiy}\phi^\theta_{\tis}(\tix,\tit;\tiy)\bigr)^{-1}=I_{n-1}+O(\theta).
\end{equation}

 By \eqref{v5}, \eqref{v13'}, \eqref{v18} and the separation condition in \eqref{q22}, if
 $\tis=\tis'$ we have
 \begin{equation}\label{v25}
 |\delta^\theta_{\tis,\tis}(\tix,\tit; \tiy,\tiy')|>0 \quad \text{on } \, \, \text{supp }(a_{\la,\nu}\cdot a_{\la,\nu'}),
 \end{equation}
 if $\theta$ is small enough.
 Thus, in this case the quantities inside the absolute values in \eqref{v20} and \eqref{v21} equal
 \begin{equation}\label{v26}
 \bigl\langle \delta^\theta_{\tis,\tis}(\tix,\tit; \tiy,\tiy'), 
 \, \bigl(\tfrac{\partial A}{\partial t} (\tit,s)\bigr)^{-1}\delta^\theta_{\tis,\tis}(\tix,\tit;\tiy,\tiy')\bigr\rangle
 +O(\theta) \, \, \, \text{on } \, \, \text{supp }(a_{\la,\nu}\cdot a_{\la,\nu'}), 
 %\, \, s=\tis,
 \end{equation}
 and, therefore, by     \eqref{v13} and \eqref{v25} the conditions \eqref{v20} and \eqref{v21} are
 valid for small enough $\theta$ when $\tis=\tis'$.
 So, by \cite[Theorem 1.1]{LeeBilinear}, we obtain \eqref{q28}, we obtain \eqref{q28} in this case.
 
 If $\tis\ne \tis'$ in \eqref{q28}, we must replace $\delta^\theta_{\tis,\tis}$ by $\delta^\theta_{\tis,\tis'}$.
 In order to accommodate this, 
 we first need to use the fact that, by the last part of \eqref{q22},
 $$\delta^\theta_{\tis,\tis'}(\tix,\tit; \tiy,\tiy')
 =\delta^\theta_{\tis,\tis}(\tix,\tit;\tiy,\tiy')+O(|\tis-\tis'|)
 \quad \text{on } \, \text{supp }(a_{\la,\nu}\cdot a_{\la,\nu'}).$$
 Thus, by the last part of \eqref{q27},
 $$\delta^\theta_{\tis,\tis'}(\tix,\tit; \tiy,\tiy')
 =\delta^\theta_{\tis,\tis}(\tix,\tit;\tiy,\tiy')+O(\delta_0)
 \quad \text{on } \, \text{supp }(a_{\la,\nu}\cdot a_{\la,\nu'}).$$
 This means that, if we replace $O(\theta)$ by $O(\theta+\delta_0)$ in \eqref{v26}, then the quanitity
 in \eqref{v20} is of this form.
 
 The other condition, \eqref{v21} involves the phase function $\phi^\theta_{\tis'}$ and the associated
 $q^\theta_{\tis'}$.  However if $B=B(\tit,\tis,\tis')$ is as in \eqref{v5.1}, then we have the analog
 of \eqref{v15} where we replace the first term in the right side of \eqref{v15} by $B\tiy$ and the 
 first term in the right side of \eqref{v16} by $B^{-1}\tiy$.  Also, of course $\tfrac{\partial A}{\partial \tit}(\tit,\tis')=
\tfrac{ \partial A}{\partial \tit}(\tit,\tis)+O(|\tis-\tis'|)$.
 Consequently, $q^\theta_{\tis'}$ will agree with $q^\theta_{\tis}$ when $a_{\la,\nu}\cdot a_{\la,\nu'}\ne0$ up to
 a $O(|s-s'|)=O(\delta_0)$ error, and by \eqref{v5.1} the analogs of \eqref{v23} and \eqref{v24} remain valid if
 $\tis$ is replaced by $\tis'$ if $O(\theta)$ there is replaced by $O(\theta+\delta_0)$.
  So, like \eqref{v20}, if we replace $O(\theta)$ by $O(\theta+\delta_0)$ in \eqref{v26},
 then the quantity in \eqref{v21} is of this form.

Thus, if $\delta_0$ in \eqref{22.6} is (finally) fixed
 small enough, and, as above, $\theta$ is small enough we conclude that the condition (1.4) in 
 \cite{LeeBilinear} is valid, which yields \eqref{q28} and completes
 the proof of Proposition~\ref{key}.    \end{proof}
 
 \begin{proof}[Proof of Lemma~\ref{ABlemma}]
 
Let us first prove \eqref{v13} and \eqref{v13'} since they are slightly more difficult than the other estimate, \eqref{v4}, in
the lemma.
%We first recall that the Schr\"odinger coordinates $(\tilde x,\tilde t)$  in
%\eqref{q99} and \eqref{q999} come from the Fermi normal coordinates $(x_1,\dots, x_{n-1})$
%about the spatial geodesic
% $\overline{\gamma}_j$.  We shall also use that, in these coordinates $\overline{\gamma_j}=(0,\dots,0,t)$
% and on this  geodesic
%the metric is $\delta_{j,k}$ (rectangular) and the Christoffel symbols vanish there as well.

If we recall \eqref{q17} we see that 
\begin{equation}\label{v6}
A(\tilde t,\tilde s)=
\frac{\partial^2 \tilde\varphi}{\partial \tiy_j \partial \tiy_k}(0,\tit; 0, \tis) = \frac{\partial^2 \varphi}{\partial \tiy_j \partial \tiy_k}(0,\tit; 0, \tis),
\end{equation}
where $\varphi$ is as in \eqref{q08}.

By \eqref{q12} we have
\begin{equation}\label{v7}
\frac{\partial^2 \varphi}{\partial \tiy_{n-1}^2}(0,\tit;0,\tis)=\frac{(\kappa_\ell^{c_0\theta})^2}{2(\tit-\tis)}.
\end{equation}
Additionally, by \eqref{q9} and \eqref{q99} we have
\begin{equation}\label{v8}
\varphi(0,\tit;\tiy,\tis)=
\frac{\big[ d_g\bigl((0,\dots,0,2\kappa_\ell^{c_0\theta}\tit), \, (\tiy_1,\dots,\tiy_{n-2}, 2\kappa^{c_0\theta}_\ell \tis
+\kappa_\ell^{c_0\theta} \tiy_{n-1}) \bigr)        \bigr]^2}{4(\tit-\tis-\tiy_{n-1})}.
\end{equation}
By \eqref{v8} we have
\begin{equation}\label{v9}
\frac{\partial^2\varphi}{\partial \tiy_j\partial \tiy_{n-1}}(0,\tit; 0,\tis)\equiv 0, \quad \text{if } \, \, j=1,\dots,n-2.
\end{equation}

The remaining part of the Hessian in \eqref{v6} is
\begin{multline}\label{v10}
\frac{\partial^2\tilde \varphi}{\partial \tiy_j \partial\tiy_k}(0,\tit;0,\tis)=
\frac{\partial^2}{\partial \tiy_j\partial \tiy_k}
\frac{\bigl[ d_g\bigl((0,\dots,0,,2\kappa_\ell^{c_0\theta}\tit), (\tiy_1,\dots, 2\kappa_\ell^{c_0\theta}\tis\bigr)\bigr]^2}
{4(\tit-\tis)}, 
\\
\text{when } \, \, \tiy_1=\dots=\tiy_{n-2}=0 \, \, \,
\text{and } \, \, 1\le j,k\le n-2.
\end{multline}
To compute this, we recall that the Schr\"odinger coordinates $(\tilde x,\tilde t)$  in
\eqref{q99} and \eqref{q999} come from the Fermi normal coordinates $(x_1,\dots, x_{n-1})$
about the spatial geodesic
 $\overline{\gamma}_j$, and that in these coordinates $\overline{\gamma_j}=(0,\dots,0,t)$
 and on this  geodesic
the metric is $\delta_{j,k}$ (rectangular) and the Christoffel symbols vanish there as well

As a result, in the Fermi normal coordinates,
we must have that the full Hessian of the 
square of the distance function satisfies 
%(in these coordinates)
$$\frac{\partial^2}{\partial y_j\partial y_k}\bigl[d_g\bigl( (0,\dots,0,2\kappa_\ell^{c_0\theta}\tilde t) ,y\bigr)\bigr]^2=2I_{n-1},
\quad \text{if } \, 
y=(0,\dots,0,2\kappa_\ell^{c_0\theta}\tilde t).
$$
This along with \eqref{q999} means that
\eqref{v10}, the remaining piece of the Hessian in \eqref{v6}, must be of the form
$$\frac{\partial^2\varphi}{\partial \tiy_j\partial\tiy_k}
=\frac1{2(\tit-\tis)}\delta_{j,k} +O(1), \quad 
\text{if } \, \, 1\le j,k\le n-2.$$
Note that since
$$(\tit,\tiy_1,\dots,\tiy_{n-2})\to d_g((0,\dots,0,2\kappa^{c_0\theta}_\ell \tit), (\tiy_1,\dots,\tiy_{n-2},2\kappa^{c_0\theta}_\ell \tis'))$$
is smooth we also obtain
\begin{equation}\label{v11}
\frac{\partial}{\partial \tit}\frac{\partial^2 \varphi}{\partial \tiy_j \partial \tiy_k}(0,\tit; 0,\tis)=\frac{-1}{2(\tit-\tis)^2}\delta_{j,k}+O((\tit-\tis)^{-1}), 
\, \, 1\le j,k\le n-2.
\end{equation}

Therefore, by \eqref{v1}, \eqref{v6}, \eqref{v7}, \eqref{v9} and \eqref{v11}, we have
\begin{multline}\label{v12}
-2(\tit-\tis)^2 \frac{\partial}{\partial \tit}A(\tit,\tis)=
-2(\tit-\tis)^2
\frac{\partial}{\partial \tit} \frac{\partial^2\tilde \varphi}{\partial \tiy_j\partial \tiy_k}(0,\tit;0,\tis)= J_{\kappa_\ell^{c_0\theta}} +O(|\tit-\tis|),
\\
\text{if } \quad J_{\kappa_\ell^{c_0\theta}}=\text{diag }(1,\dots,1,(\kappa_\ell^{c_0\theta})^2).
\end{multline}
Note that $\kappa^{c_0\theta}_\ell \in [1/10,10]$.
Therefore by \eqref{q22'}
if $\delta$ is fixed small enough in \eqref{22.6} and if $\delta_0$ there is smaller than $1/8$, by \eqref{q22'}, we have that 
$-(\partial A(\tit,\tis)/\partial \tit)^{-1}$ is positive definite on the support of the amplitudes in \eqref{q29}.
Thus, we obtain \eqref{v13} and \eqref{v13'}
for some $c_\delta>0$.  Indeed, one may take $c_\delta\sim \delta^{-2}$,

The proof of the other \eqref{v4} is very similar.  If we use \eqref{q08} we see that since,
by \eqref{q7} and \eqref{q22'}, 
$d_g(x,y)\approx |t-s|\approx \delta$ on $\text{supp } a^\theta_{\la,j}\cdot a^\theta_{\la,\nu'}$,
we may assume that $|x|, \, |y|\le C\delta$.  We then have
$$(d_g(x,y))^2 =|x-y|^2 +r(x,y), \quad
\text{where } \, r(x,y)= O\big((|x|+|y|) |x-y|^2 \big) \, \, \text{and } \, r\in C^\infty.$$
Therefore, by \eqref{q9}, \eqref{q17} and \eqref{q99}, if $\bar x=(\tilde x_1, \tilde x_2, \dots, \tilde x_{n-2})$ and $\bar y=(\tilde y_1, \tilde y_2, \dots, \tilde y_{n-2})$
\begin{equation}\label{v4''}
\tilde \varphi(\tilde x,\tilde t; \tilde y, \tilde s)=
  \frac{(\speed)^2(\tilde x_{n-1}-\tilde y_{n-1})^2+|\bar x-\bar y|^2 +r(x,y)}{4\big(\tilde t-\tilde s+(\tilde x_{n-1}-\tilde y_{n-1})\big)}  -(\kappa_\ell^{c_0\theta})^2 (\tilde t-\tilde s).
\end{equation}
Consequently, by the proof of \eqref{v13} and \eqref{v13'} we have that for
$J_{\kappa_\ell^{c_0\theta}}$ as above
$$
B(\tilde t,\tilde s)=\frac{\partial^2\tilde \varphi}{\partial \tilde x_j\partial\tilde y_k}(0,\tilde t;0,\tilde s)=-\frac1{2(\tit-\tis)} J_{\kappa_\ell^{c_0\theta}} +O(1),
$$
which yields \eqref{v4} if $\delta$ is small enough.
 \end{proof}

\newsection{Kernel estimates}

In this section we finish up matters by proving the various kernel estimates that we have utilized.

\medskip

\noindent{\bf 4.1. Basic kernel estimates on manifolds of nonpositive curvature}

Let us prove the kernel estimates that we used on $A_+$.

\begin{proposition}\label{kerprop}  Let $S_\la(x,t;y,s)$ denote the kernel
$$\eta(t/T)\eta(s/T) \beta^2(P/\la)\bigl(e^{-i(t-s)\la^{-1}\Delta_g}\bigr)(x,y).$$
Then if $M=M^{n-1}$ has nonpositive sectional curvatures and $T=c_0\log\la$ with 
$c_0=c_0(M)>0$ sufficiently small, we have for $\la\gg 1$
\begin{equation}\label{j1}
|S_\la(x,t;y,s)|\le C\la^{\frac{n-1}2}|t-s|^{-\frac{n-1}2} \exp(C_M|t-s|).
\end{equation}
\end{proposition}

To prove this we note that for fixed $t$ and $s$,
$\beta^2(P/\la)e^{-i(t-s)\la^{-1}\Delta_g}=\beta^2(P/\la)e^{i(t-s)\la^{-1}P^2}$ is the  Fourier multiplier operator
on $M^{n-1}$ with
\begin{equation}\label{j2}
     m(\la,t-s;\tau)=\beta^2(|\tau|/\la)e^{i(t-s)\la^{-1}\tau^2}.
\end{equation}
We have extended $m$ to be an even function of $\tau$ so that we can write
\begin{equation}\label{j3}
\beta^2(P/\la)e^{-i(t-s)\la^{-1}\Delta_g}=(2\pi)^{-1}\int_{-\infty}^\infty 
\Hat m(\la,t-s;r) \, \cos r\sqrt{-\Delta_g} \, dr,
\end{equation}
where
\begin{equation}\label{j4}
\Hat m(\la,t-s;r)=\int_{-\infty}^\infty e^{-i\tau r}
\beta^2(|\tau|/\la) \, e^{i(t-s)\la^{-1}\tau^2} \, d\tau.
\end{equation}

We note that, by a simple integration by parts argument,
\begin{equation}\label{j5}
\partial^k_r\Hat m(\la,t-s;r)=O(\la^{-N}(1+|r|)^{-N}) \, \forall \, N, 
\, \, \text{if } \, |t-s|\le 1 \, \, \text{and } \, \, \, |r|\ge C_0,
\end{equation}
with $C_0$ sufficiently large.
Similarly
\begin{multline}\label{j6}
\partial^k_r \Hat m(\la,t-s;r)=O(\la^{-N}(1+|r|)^{-N}) \, \forall \, N, 
\\
 \text{if } \, |t-s|\in [2^{j-1},2^j ], \, \, \text{and } \, \, |r|\ge C_02^j, \, \, 
j=1,2,\dots,
\end{multline}
with $C_0$ fixed large enough.  Since $\beta(|\tau|/\la)=0$ if $|\tau|\notin [\la/4,2\la]$ one may take 
$C_0=100$, as we shall do.

To use this fix an even function $a\in C^\infty_0(\R)$ satisfying
$$a(r)=1, \, \, |r|\le 100 \quad \text{and } \, \, a(r)=0\, \, \text{if } \, \,  |r|\ge200.$$
Then by crude eigenfunction estimates and the Weyl formula, if we let
\begin{equation}\label{j7}
\tilde S_{\la,0}(x,t;y,s)=(2\pi)^{-1}\int a(r)\Hat m(\la,t-s,r) \cos rP\, dr
\end{equation}
we have
\begin{equation}\label{j8}
\tilde S_{\la,0}(x,t;y,s)-\bigl(\beta^2(P/\la) e^{-i(t-s)\la^{-1}\Delta_g}\bigr)(x,y)=O(\la^{-N}) \, \forall \, N
\quad \text{if } \, |t-s|\le 1,
\end{equation}
and if 
\begin{equation}\label{j9}
\tilde S_{\la,j}(x,t;y,s)=(2\pi)^{-1}\int a(2^{-j}r)\Hat m(\la,t-s,r) \cos rP\, dr
\end{equation}
we have
\begin{multline}\label{j10}
\tilde S_{\la,j}(x,t;y,s)-\bigl(\beta^2(P/\la) e^{-i(t-s)\la^{-1}\Delta_g}\bigr)(x,y)=O(\la^{-N}) \, \forall \, N
\\
\text{if } \, |t-s| \in [2^{j-1},2^j], \, \, j=1,2,\dots.
\end{multline}

Consequently, we would have \eqref{j1} if we could show that
\begin{equation}\label{j11}
|\tilde S_{\la,0}(x,t;y,s)|\le \la^{\frac{n-1}2}|t-s|^{-\frac{n-1}2} \quad
\text{when } \, \, |t-s|\le 1,
\end{equation}
as well as
\begin{multline}\label{j12}
|\tilde S_{\la,j}(x,t;y,s)|\le \la^{\frac{n-1}2} \exp(C2^j), \quad \text{if } \, |t-s|\in [2^{j-1},2^j]
\\
\text{with } \, \, j=1,2,\dots \, \, \text{and } \, \,  \, 2^j \le c_0\log\la
\end{multline}
with $c_0=c_0(M)$ fixed small enough.

To prove \eqref{j11} and \eqref{j12} we shall argue as in B\'erard~\cite{Berard} (see also \cite[\S3.6]{SoggeHangzhou}).   Just as in \cite{Berard}, \cite{BSTop}, \cite{SoggeZelditchL4} and other works we
shall want to use the Hadamard parametrix and the 
Cartan-Hadamard theorem to lift the 
calculations that will be needed up to the
universal cover $({\mathbb R}^{n-1},\tilde g)$
of $(M^{n-1},g)$.

We therefore let $\{\alpha\}=\Gamma$ denote the group
of deck transfermations preserving the associated
covering map $\kappa: {\mathbb R}^{n-1}\to M^{n-1}$
coming from the exponential map at the point
in $M^{n-1}$ with coordinates $0$ in $\Omega$ in
\S4 above.  The metric $\tilde g$ on ${\mathbb R}^{n-1}$
is the pullback of the metric $g$ on $M^{n-1}$
via $\kappa$.  Choose a Dirichlet domain $D\simeq
M^{n-1}$ for $M^{n-1}$ centered at the lift of the 
point with coordinates $0$.

As in earlier works (see \cite{SoggeHangzhou}) we recall
that if $\tilde x$ denotes the lift of $x\in M^{n-1}$
to $D$, then we have the following formula
\begin{equation}\label{j13}(\cos tP)(x,y)
=(\cos t\sqrt{-\Delta_g})(x,y)
=\sum_{\alpha \in \Gamma}
\bigl(\cos t\sqrt{-\Delta_{\tilde g}}\bigr)(\tilde x,
\alpha(\tilde y)).
\end{equation}
As a result, if we set
\begin{equation}\label{j14}
K_{\la,0}(\tilde x,t;\tilde y,s)=
(2\pi)^{-1}\int a(r)\Hat m(\la,t-s;r)
\, \bigl( \cos r\sqrt{-\Delta_{\tilde g}}\bigr)(\tilde x,
\tilde y) \, dr,
\end{equation}
we have the formula
\begin{equation}\label{j15}
\tilde S_{\la,0}(x,t;y,s)=
\sum_{\alpha \in \Gamma}
K_{\la,0} (\tilde x,t; \alpha(\tilde y),s)).
\end{equation}
Similarly, if we set
\begin{equation}\label{j16}
K_{\la,j}(\tilde x,t;\tilde y,s)=
(2\pi)^{-1}\int a(2^{-j}r)\Hat m(\la,t-s;r)
\, \bigl( \cos r\sqrt{-\Delta_{\tilde g}}\bigr)(\tilde x,
\tilde y) \, dr,
\end{equation}
we have
\begin{equation}\label{j17}
\tilde S_{\la,j}(x,t;y,s)
=
\sum_{\alpha \in \Gamma}
K_{\la,j} (\tilde x,t;\alpha(\tilde y),s).
\end{equation}

Also, by Huygen's principle and the support properties
of $a$, we have that
\begin{equation}\label{j18}
K_{\la,0}(\tilde x,\tilde y)=0 \quad \text{if } \, \,
d_{\tilde g}(\tilde x,\tilde y)\ge C_1, \, \, \,
\text{and } \, \, K_{\la,j}(\tilde x,\tilde y)=0 \, \, \text{if } \, \, 
d_{\tilde g}(\tilde x,\tilde y)\ge C_12^j
\end{equation}
for a uniform constant $C_1$.  Based on this, we conclude that the number
of non-zero summands in the right side of \eqref{j15} is $O(1)$ since $\alpha(D)\cap \alpha'(D)=\emptyset$ if $\alpha\ne \alpha'$.
Also, by simple volume estimates, the number of $\alpha\in \Gamma$ for which $d_{\tilde g}(D,\alpha(D))\le \mu$ is
$O(\exp(C\mu))$ for a uniform constant $C$ if $\mu=2^j$, $j=1,2,\dots$, and so the number of nonzero summands
in the right side of \eqref{j16} is $O(\exp(C2^j))$.  As a result, we would obtain \eqref{j11} if we could show that
\begin{equation}\label{j19}
|K_{\la,0}(\tilde x,t;\tilde y,s)|\le C\la^{\frac{n-1}2}|t-s|^{-\tfrac{n-1}2}, \quad \text{if } \, \, |t-s|\le 1,
\end{equation}
while \eqref{j12} would follow from the estimate
\begin{multline}\label{j20}
|K_{\la,j}(\tilde x,t;\tilde y,s)|\le C\la^{\frac{n-1}2}\exp(C2^j),
\\ \text{if } \, \, 
|t-s| \in [2^{j-1},2^j], \, \, j=1,2,\dots, \, 2^j\le  c_0\log\la,
\end{multline}
with $c_0=c_0(M)$ sufficiently small.

To prove these two estimates, we can use the Hadamard
parametrix for $\partial_r^2-\Delta_{\tilde g}$
since $({\mathbb R}^{n-1},\tilde g)$ is a Riemannian
manifold without conjugate points, i.e., its
 injectivity radius is infinite.  Thus, we can
use the Hadamard parametrix to write for
$\tilde x\in D$, $\tilde y\in {\mathbb R}^{n-1}$
and $|r|\ge c_0>0$
\begin{equation}\label{k13}
\bigl(\cos r\sqrt{-\Delta}_{\tilde g}\bigr)(\tilde x,
\tilde y)=
\sum_{\nu=0}^N
w_\nu(\tilde x,\tilde y)W_\nu(r,\tilde x,\tilde y)
+R_N(r,\tilde x,\tilde y)
\end{equation}
where $w_\nu \in C^\infty({\mathbb R}^{n-1}
\times {\mathbb R}^{n-1})$,
\begin{equation}\label{k14}
W_0(r,\tilde x,\tilde y)=(2\pi)^{-(n-1)}
\int_{{\mathbb R}^{n-1}} e^{id_{\tilde g}(\tilde x,
\tilde y)\xi_1} \cos r|\xi| \, d\xi,
\end{equation}
while for $\nu=1,2,\dots$, $W_\nu(t,\tilde x,\tilde y)$
is a finite linear combination of Fourier integrals 
of the form
\begin{equation}\label{k15}
\int_{{\mathbb R}^{n-1}}
e^{id_{\tilde g}(\tilde x,\tilde y)\xi_1}
e^{\pm ir|\xi|}
\alpha_\nu(|\xi|)\, d\xi,
\, \, \text{with } \, \,
\alpha_\nu(\tau)=0, \, \text{for } \, \tau\le 1
\, \, \text{and } \, \, \partial^j_\tau \alpha_\nu(\tau)
\lesssim \tau^{-\nu-j},
\end{equation}
and, if $N_0$ is given, then if $N$ is large enough,
\begin{equation}\label{k16}
|\partial_r^j R_N(r,\tilde x,\tilde y)|\le C\exp(Cr),
\quad
0\le j\le N_0
\end{equation}
for a fixed constant $C$.  Furthermore, the leading
coefficient $w_0(\tilde x,\tilde y)$ reflects the
geometry of $({\mathbb R}^{n-1},\tilde g)$.
Specifically, in geodesic normal coordinates
about $\tilde x$
$$w_0(\tilde x,\tilde y)=
\bigl(\text{det }\tilde g_{ij}(\tilde y)\bigr)^{-1/4}.$$
Thus, if in geodesic polar coordinates the
volume element is given by 
$$dV_{\tilde g}(\tilde y)
=\bigl({\mathcal A}(r,\omega)\bigr)^{n-2}dr d\omega,
\, \, r=d_{\tilde g}(\tilde x,\tilde y),$$
then
$$w_0(\tilde x,\tilde y)=\bigl(r/{\mathcal A}(r,\omega)
\bigr)^{\frac{n-2}2}.$$
By the classical G\"unther comparison theorem from
Riemannian geometry 
(see \cite[\S III.4]{ChavelRiemannianGeometry})
\begin{equation}\label{k17}
w_0(\tilde x,\tilde y)\le 1,
\end{equation}
and, moreover, for later use,
${\mathcal A}(r,\omega)\ge \tfrac1K\sinh(Kr)$
if all the sectional curvatures are $\le -K^2<0$,
and so
\begin{multline}\label{k18}
w_0(\tilde x,\tilde y)\le C_{K,N}\mu^{-N}
\\ \text{if } \, \,
d_{\tilde g}(\tilde x,\tilde y) \approx \mu
\, \, \text{and } \, \, 
\text{all the sectional curvatures of }\, M^{n-1}
\, \, \text{are} \, \, \le -K^2<0.
\end{multline}
The other coefficients in \eqref{k13} are not as
well behaved; however, B\'erard~\cite{Berard}
showed that if $N_0$ is fixed
\begin{equation}\label{k19}
|\partial^\beta_x w_\nu(\tilde x,\tilde y)|
\le C\exp(Cr), \, \,
|\beta|, \nu\le N_0, \quad r=d_{\tilde g}(\tilde x,
\tilde y),
\end{equation}
for some uniform constant $C$ (depending on 
$\tilde g$ and $N_0$).

The facts that we have just recited are well known.
One can see, for instance,  \cite{Berard} or \cite[\S 1.1, \S 3.6]{SoggeHangzhou} for 
background regarding the Hadamard parametrix,
and \cite{SoggeZelditchL4} for a discussion
of properties of $w_0$.

Let us next use the Hadamard parametrix to prove \eqref{j19}.   By \eqref{k13}, it suffices
to see that if we replace $(\cos r\sqrt{-\Delta_{\tilde g}})(\tilde x,\tilde y)$ in \eqref{j14} by each
of the terms in the right side of \eqref{k13} then each such expression will satisfy the bounds
in \eqref{j19}.

Let us start with the contribution of the main term in the Hadamard parametrix which is the
$\nu=0$ term in \eqref{k13}.  In view of \eqref{k14} and \eqref{k17} it would give rise to these
bounds if
\begin{multline}\label{j28}
(2\pi)^{-n}\int_{-\infty}^\infty \int_{{\mathbb R}^{n-1}} e^{id_{\tilde g}(\tilde x,\tilde y)\xi_1}
\cos (r|\xi|) \, a(r) \, \Hat m(\la,t-s;r) \, dr d\xi
\\
=O(\la^{\frac{n-1}2}|t-s|^{-\frac{n-1}2}) \quad
\text{when } \, \, \, |t-s|\le 1.
\end{multline}
However, by \eqref{j2} and \eqref{j5} and the support properties of $a$, if $|t-s|\le 1$
\begin{align}\label{j29}
(2\pi)^{-1}&\int_{-\infty}^\infty \int_{{\mathbb R}^{n-1}} e^{id_{\tilde g}(\tilde x,\tilde y)\xi_1}
\cos (r|\xi|) \, a(r) \, \Hat m(\la,t-s;r) \, dr d\xi
\\
&=
(2\pi)^{-1}\int_{-\infty}^\infty \int_{{\mathbb R}^{n-1}} e^{id_{\tilde g}(\tilde x,\tilde y)\xi_1}
\cos (r|\xi|) \Hat m(\la,t-s;r) \, dr d\xi+O(\la^{-N}) \notag
\\
&=\int_{{\mathbb R}^{n-1}} e^{id_{\tilde g}(\tilde x,\tilde y)\xi_1} \beta^2(|\xi|/\la) e^{i(t-s)\la^{-1}|\xi|^2} \, d\xi
+O(\la^{-N}). \notag
\end{align}
A simple stationary phase argument shows that the last integral is $O(\la^{\frac{n-1}2}|t-s|^{-\frac{n-1}2})$, and so  
we conclude that the main term in the Hadamard parametrix leads to the desired bounds.

To estimate the contributions of the higher order terms $\nu=1,2,\dots$, we note that by the first part of \eqref{j18}
we may assume that $d_{\tilde g}(\tilde x,\tilde y)$ is bounded.  So, by \eqref{k15} the higher order terms would
lead to the desired bounds since
\begin{multline}\label{j29'}
(2\pi)^{-1} \iint e^{id_{\tilde g}(\tilde x,\tilde y)\xi_1} e^{\pm ir|\xi|} \alpha_\nu(|\xi|) \, a(r)\Hat m(\la,t-s;r) \, dr d\xi
\\
=\int_{{\mathbb R}^{n-1}} e^{id_{\tilde g}(\tilde x,\tilde y)\xi_1} \beta^2(|\xi|/\la) e^{i(t-s)\la^{-1}|\xi|^2} \, \alpha_\nu(|\xi|) \, d\xi
+O(\la^{-N}), \, \, \text{if } \, \, |t-s|\le 1,
\end{multline}
and, by \eqref{k15}, together with a stationary phase argument, the last integral is
$O(\la^{\frac{n-1}2-\nu}|t-s|^{-\frac{n-1}2})$.

We also need to see that the remainder term in \eqref{k13} leads to the bounds
\begin{multline}\label{j30}
\int^\infty_{-\infty} a(r)\Hat m(\la,t-s;r) R(r,\tilde x,\tilde y) \, dr
\\
=\int^\infty_{-\infty} \beta^2(|\tau|/\la) e^{i(t-s)\la^{-1}\tau^2} \, 
\bigl[ a(\, \cdot \, )R(\, \cdot\, ,\tilde x,\tilde y)\bigr]\, \widehat{}\, \,(\tau) \, d\tau =O(\la^{\frac{n-1}2}), \, \, 
\text{if } \, \, d_{\tilde g}(\tilde x,\tilde y)=O(1).
\end{multline}
By \eqref{k16}, the last factor in the integral in the right, which is the Fourier transform of
$r\to a(r)R(r,\tilde x,\tilde y)$, is $O(1)$ if $d_{\tilde g}(\tilde x,\tilde y)=O(1)$.  So, by the support properties of $\beta$, the last integral
in \eqref{j30} is $O(\la)=O(\la^{\frac{n-1}2})$, as desired, since $n\ge3$.

Since each term in the Hadamard parametrix has the desired contribution, the proof of \eqref{j19} is complete.

Similar arguments will yield \eqref{j20}.  We need to see that if we replace $(\cos r\sqrt{-\Delta_{\tilde g}})(\tilde x,\tilde y)$ in
\eqref{j16}
by each of the terms in the right side of \eqref{k13}, then each will satisfy the bounds in \eqref{j20} if
$d_{\tilde g}(\tilde x,\tilde y)\le C2^j$ and $|t-s|\in [2^{j-1},2^j]$ with $j=1,2,\dots$ and
$2^j\le c_0\log \la$ as above.

By \eqref{j6} and \eqref{k17} and the above argument, the $\nu=0$ term in the Hadamard parametrix will lead to 
a contribution of
$$(2\pi)^{-(n-1)}\int_{{\mathbb R}^{n-1}} e^{id_{\tilde g}(\tilde x,\tilde y)\xi_1} e^{i(t-s)\la^{-1}|\xi|^2}
\beta^2(|\xi|/\la)\, d\xi +O(\la^{-N}) =O(\la^{\frac{n-1}2}),$$
by stationary phase and the fact that we are assuming $|t-s|\ge1$.  By \eqref{k19} and the above arguments each of the
$\nu=1,2,\dots$ terms will have contributions which are $O(\la^{\frac{n-1}2-\nu}\cdot \exp(C2^j))=O(\la^{\frac{n-1}2})$
if $2^j\le c_0\log\la$ with $c_0>0$ small enough.  If we repeat the argument above for the contribution of the 
remainder term, we see that the contributions here will be of the form
$$\int^\infty_{-\infty} \beta^2(|\tau|/\la) e^{i(t-s)\la^{-1}\tau^2}
\bigl[ a(2^{-j}\, \cdot \, )R(\, \cdot\, ,\tilde x,\tilde y)\bigr]\, \widehat{}\, \,(\tau) \, d\tau,$$
which, by \eqref{k16} and the support properties of $a$, is $O(\la \exp(C2^j))=O(\la^{\frac{n-1}2}\exp(C2^j))$,
as desired.

Since each term in the Hadamard parametrix has the desired contribution, the proof of \eqref{j20} is complete,
which finishes the proof of Proposition~\ref{kerprop}.

We should point out that the small $|t-s|$ estimates are universally true as in \cite{bgtmanifold}.

\medskip

 \noindent{\bf 4.2. Estimates for kernels of microlocalized operators}
 
 Let us prove the kernel estimates, \eqref{22.65np} and \eqref{22.65neg}, that we used in the proof of $L^{q_c}(A_-)$-estimates.

\begin{proposition}\label{mick} For each $m\in {\mathbb Z}$ pick $\nu(m)\in {\mathbb Z}^{2n-3}$ as in
\eqref{m10} and
 let
\begin{multline}\label{z1}
K_\la(x,t,m;y,s,m') =
\\
\alpha_m(t)\alpha_{m'}(s)
\Bigl( A_{\nu(m)}^{\theta_0} \circ \bigl( \, \beta^2(P/\la) e^{-i\la^{-1}(t-s)\Delta_g} \, \bigr)
\circ (A^{\theta_0}_{\nu(m')})^* \Bigr)(x,y).
\end{multline}
Then if $M=M^{n-1}$ has nonpositive curvature
\begin{equation}\label{z2}
|K_\la(x,t,m;y,s,m')|\le C\la^{\frac{n-1}2} \, |t-s|^{-\frac{n-1}2},
\end{equation}
provided that $|t-s|\le c_0\log\la$ with $c_0=c_0(M)>0$
sufficiently small.
Moreover,  for such $|t-s|$ we have
\begin{multline}\label{z3}
|K_\la(x,t,m;y,s,m')|\le C\la^{\frac{n-1}2} \, |t-s|^{-N} \, \, \forall \, N, 
\\ \text{if }\, |t-s|\ge1 \, \,
\text{and all the sectional curvatures of } \, M^{n-1} \, \text{are negative}.
\end{multline}
The uniform constants $C=C(M^{n-1})$ do not depend on the particular choice of the 
$\nu(m)$.
\end{proposition}
 
\begin{proof}
Since, as we mentioned before, the kernels
of the $A_\nu^{\theta_0}$ operators satisfy the uniform
bounds
\begin{equation}\label{z4}
\int|A_\nu^{\theta_0}(x,y)|\, dx, \, \, \,
\int|A_\nu^{\theta_0}(x,y)| \, dy \, \le C,
\end{equation}
by Proposition~\ref{kerprop}, we obtain
the above bounds when $|t-s|\le 1$.

Also, by \eqref{z4}, if we let
\begin{multline}\label{z5}
\tilde K_\la(x,t,m;y,s,m')=
\\
\alpha_m(t)\alpha_{m'}(s)
\Bigl(A_{\nu}^{\theta_0}\circ \bigl( \, \beta^2(P/\la) e^{-i\la^{-1}(t-s)\Delta_g} \, \bigr)
\Bigr)(x,y), \quad \nu=\nu(m),
\end{multline}
it suffices to see that this kernel, which does not
include the microlocal cutoffs in the right satisfies
the bounds in \eqref{z2} and \eqref{z3} when
$|t-s|\ge1$.

Let us start by proving that \eqref{z5} satisfies the 
bounds in \eqref{z3} for $|t-s|\ge 1$
if $M^{n-1}$ has nonpositive curvatures.

To do this recall that, by  \eqref{m7} and \eqref{m10} with $\theta=\theta_0=\la^{-1/8}$ , if $\nu(m)=(\theta_0k, \theta_0\ell)$ then
\begin{multline}\label{z6}
A_\nu^{\theta_0}(x,D)=\tilde A_k^{\theta_0}(x,D)
\circ b(\la^{-7/8}(P-\la \kappa_\ell^{\theta_0})), \quad
\kappa^{\theta_0}_\ell=1+\theta_0\ell, \, \, |\ell|\lesssim \theta_0^{-1},
\\
\text{if } \, \, \tilde A_k^{\theta_0}(x,D)
=A_k^{\theta_0}(x,D)\circ \tilde \Upsilon(P/\la),
\end{multline}
with $b\in C^\infty_0((-1,1))$ and $\Upsilon$ as in \eqref{ups}.
Here the $A_k^{\theta_0}$ operators localize at
scale $\theta_0=\la^{-1/8}$ about a geodesic
$\overline{\gamma}_k$ in $\Omega$ due to 
\eqref{m2}--\eqref{m6}.

By \eqref{z6} we now have the following variant of 
\eqref{j3}
\begin{align}\label{z7}
A_\nu^{\theta_0} \circ \bigl(\beta^2(P/\la) e^{-i(t-s)\la^{-1}\Delta_g}\bigr)
&= \tilde A_k^{\theta_0}
\circ\bigl(\beta^2(P/\la) \, b(\la^{-7/8}(P-\la
\kappa^{\theta_0}_\ell))\, 
e^{-i(t-s)\la^{-1}\Delta_g}\bigr)
\\
&=
\Bigl( \, (2\pi)^{-1}\int
\Hat m_\ell(\la,t-s;r) \, \bigl(
 \tilde A_k^{\theta_0} \circ \cos r\sqrt{-\Delta_g}\bigr) \, dr
\Bigr),
\notag
\end{align}
if now \eqref{j4} is replaced by
\begin{equation}\label{z8}
\Hat m_\ell(\la,t-s;r)
=\int_{-\infty}^\infty e^{-i\tau r}
\beta^2(|\tau|/\la) \, 
b(\la^{-7/8}(|\tau|-\la \kappa_\ell^{\theta_0}))
e^{i(t-s)\la^{-1}\tau^2}
\, d\tau.
\end{equation}

By a simple integration by parts argument we have the
following analog of \eqref{j6}
\begin{multline}\label{z9}
\partial_r^k \Hat m_\ell(\la,t-s;r)=O(\la^{-N}(1+|r|)^{-N}) \, \, \forall \, N 
\\
 \text{if } \, \,
|t-s|\le c_0\log\la \, \, \,
\text{and } \, \, |r|\ge 100 c_0\log\la.
\end{multline}
Thus, if $a$ is as in the proof of Proposition~\ref{kerprop}, since the dyadic operators
$\tilde A_k^{\theta_0}$ have kernels as in \eqref{z4},
if we insert a factor of $(1-a(r/c_0\log\la))$ into
the integral in the last term in \eqref{z7} the resulting kernels will be $O(\la^{-N})$ for all $N$.  So, we have
reduced the proof of \eqref{z2} to showing that we have
the kernel estimates
\begin{equation}\label{z10}
\bigl(W_{\la,\ell,t-s}\bigr)(x,y)
=O(\la^{\frac{n-1}2}|t-s|^{-\frac{n-1}2}),
\quad \text{if } \, \, 1\le |t-s|\le c_0\log\la,
\end{equation}
for small enough $c_0>0$ if
\begin{equation}\label{z11}
W_{\la,\ell,t-s}=
(2\pi)^{-1}\int a(r/c_0\log\la) \, 
\Hat m_\ell(\la,t-s;r)\, \bigl(\tilde A_k^{\theta_0} \cos r\sqrt{-\Delta_g}\bigr) \, dr.
\end{equation}

To estimate \eqref{z10}, we shall argue as in the last
subsection.  We first lift the calculation up to
the universal cover exactly as before by rewriting
\begin{equation}\label{z12}W_{\la,\ell, t-s}(x,y)
=\sum_{\alpha\in \Gamma}W^\alpha_{\la,\ell,t-s}(\tilde
x, \tilde y),
\end{equation}
where
\begin{multline}\label{z13}
W^\alpha_{\la,\ell,t-s}(\tilde
x, \tilde y)=
\\
(2\pi)^{-1}\int_{-\infty}^\infty
a(r/c_0\log\la) \, \Hat m_\ell(\la,t-s;r)
\, \bigl(\tilde A_k^{\theta_0} \cos r\sqrt{-\Delta_{\tilde g}}\bigr)(\tilde x, 
\alpha(\tilde y))\, dr,
\end{multline}
and $\tilde x$, $\tilde y$ denote the lift of $x,y$, respectively, to the universal cover $({\mathbb R}^{n-1},
\tilde g)$.  By the support properties of $a$ and
Huygens principle
\begin{equation}\label{z14}
W^\alpha_{\la,\ell,t-s}(\tilde
x, \alpha(\tilde y))=0 \quad \text{if } \, \, 
d_{\tilde g}(\tilde x,\tilde \alpha(\tilde y))\ge C
c_0\log\la,
\end{equation}
with $C$ being a fixed constant.

In the last subsection we had to deal with the fact that the sums that arose after lifting the calculations
up to the universal cover involved $O(\exp(Cc_0\log\la))$ terms.  Here, because of the $\nu=(k,\ell)$ localizations,
it will turn out that, given $|t-s|\ge1$,
 there are only $O(1)$ summands above which are nontrivial.
 
 Let us start by exploiting the localization coming from the $\tilde A_k^{\theta_0}$ operators which localize
 about the geodesic $\overline{\gamma}_k$ in $\Omega$.  If we argue exactly in \cite{BSTop}, just by
 using this operator and elementary arguments involving the calculus of Fourier integral operators
 and Toponogov's triangle comparison theorem,
 we shall be able to see that the overwhelming majority of the terms in \eqref{z12} are $O(\la^{-1/2})$,
 which is much better than the bounds posited above.
 
 To do this, just as in earlier works we start by modifying the coordinates in $\Omega$ so that the
 $0\in \overline\gamma_k$.    Then, as in \cite{BSTop}, we let $\tilde \gamma(t)$, $t\in \R$ denote the lift
 of the geodesic $\overline \gamma_k$ to the universal cover and
 $${\mathcal T}_R(\tilde \gamma)=\{x: \, d_{\tilde g}(\tilde \gamma, \tilde z)\le R\}.$$
Then, just as in \cite{BSTop}, if $R$ is fixed large enough and 
$\alpha(D)\cap {\mathcal T}_R(\tilde \gamma)=\emptyset$, with, as before, $D\simeq M^{n-1}$ being our
fundamental domain, then the summand in \eqref{z12} involving $\alpha$ must by $O(\la^{-1})$
by Toponogov's theorem and microlocal arguments.  This is exactly how (3.9) in \cite{BSTop} was
proved, and one can simply repeat the arguments there to obtain this bound.

Since there are $O(\la^{1/2})$ non-zero terms
in \eqref{z12} if $c_0$ in \eqref{z13} is fixed small enough, we obtain in this case
$$\sum_{\{\alpha: \, \alpha(D)\cap{\mathcal T}_R(\tilde \gamma)=\emptyset\} }
W_{\la,\ell,t-s}^\alpha(\tilde x,\tilde y)=O(\la^{-1/2}), \quad \text{if } \, 1\le |t-s|\le c_0\log\la,
$$
which is much better than the bounds in \eqref{z10}.

In order to obtain \eqref{z10} we still have to deal with the terms for which
$ \alpha(D)\cap{\mathcal T}_R(\tilde \gamma) \ne \emptyset$; however, fortunately for us, 
by \eqref{z14}  there
are only $O(\log |t-s|)$ such non-zero terms in \eqref{z12}.  Having reduced out task
to only considering such summands we no longer need to use the microlocal cutoff
$\tilde A_k^{\theta_0}$.  Since it satisfies the bounds in \eqref{z4} we would have
\eqref{z10} if
\begin{multline}\label{z15}
\sum_{\{\alpha: \, \alpha(D)\cap{\mathcal T}_R(\tilde \gamma) \ne \emptyset\} }
(2\pi)^{-1}\int_{-\infty}^\infty
a(r/c_0\log\la) \, \Hat m_\ell(\la,t-s;r)
\, \bigl( \cos r\sqrt{-\Delta_{\tilde g}}\bigr)(\tilde x, 
\alpha(\tilde y))\, dr
\\
 =O(\la^{\frac{n-1}2}|t-s|^{-\frac{n-1}2}), 
\quad
 \text{if } \, \, 1\le |t-s|\le c_0\log\la.
 \end{multline}

To do this, just like before, we shall use the Hadamard parametrix
\eqref{k13}.   We need to see that the contribution
of each term gives a contribution satisfying the
these bounds.

If we argue as before, and use 
\eqref{k15} the contribution
of the  higher order terms to \eqref{z13} will be
a linear combination of terms of the form
\begin{multline}\label{z151}
(2\pi)^{-1}w_\nu(\tilde x,\alpha(\tilde y))
\int_{-\infty}^\infty \int e^{id_{\tilde g}(\tilde x,
\alpha(\tilde y))\xi_1} e^{\pm ir|\xi|}\alpha_\nu(|\xi|)
a(r/c_0\log\la) 
\\
\times \Hat m_\ell(\la,t-s;r)
\, d\xi dr.
\end{multline}
Assuming as we are that $|t-s|\le c_0\log\la$, 
modulo a $O(\la^{-N})$ term, just as in \eqref{j29'},
this equals
\begin{equation}\label{z16}
w_\nu(\tilde x,\alpha(\tilde y))
\int_{{\mathbb R}^{n-1}}
\beta^2(|\xi|/\la) 
b(\la^{-7/8}(|\xi|-\la
\kappa^{\theta_0}_\ell)) \,
e^{id_{\tilde g}(\tilde x,\alpha(\tilde y))\xi_1}
\, \alpha_\nu(|\xi|) \, 
e^{i(t-s)\la^{-1}|\xi|^2} \, d\xi.
\end{equation}
By an easy stationary phase calculation if $|t-s|\ge1$
the last integral is 
$O(\la^{\frac{n-1}2-\nu})
$ in view
of the last part of \eqref{k15}.  
Since, as we noted before there are only $O(\log\la)$ summands in \eqref{z15}, we conclude
that the contribution of the higher order terms, $\nu=1,2,\dots$, in Hadamard parametrix
will be $O(\la^{\frac{n-1}2-\frac12})$, which is much better than we need for \eqref{z15}.

We next notice that, similar to \eqref{j30}, the contribution of the remainder term in \eqref{k13}
will be
$$\int_{-\infty}^\infty
\beta^2(|\tau|/\la)e^{i(t-s)\la^{-1}\tau^2} b(\la^{-7/8}(|\tau|-\la\kappa^{\theta_0}_\ell))
\bigl[ a((c_0\log\la)^{-1}\, \cdot \, )R(\, \cdot\, ,\tilde x,\tilde y)\bigr]\, \widehat{}\, \,(\tau) \, d\tau.$$
By \eqref{k16}, the last factor in the integral
is $O(\exp(Cc_0\log\la))\le \la^{1/16}$ if 
$c_0$ is small enough.  Since the rest of the integrand
is bounded and supported on a set of size
$\la^{7/8}$, we conclude that the contribution
of the remainder term in the Hadamard parametrix to
\eqref{z10} also not only satisfies the bounds in \eqref{z15}, but, moreover, like the above 
terms for $\nu=1,2,\dots$ in \eqref{k13}, satisfies the improved ones
in \eqref{z3}.  Indeed, its contribution will be $O(\la^{15/16}\log\la)$ for such $c_0$.

We still have to deal with the main term in the
Hadamard parametrix, i.e., the contribution of the
$\nu=0$ term in \eqref{k13} to \eqref{z151}.  
Arguing as before, the proof of \eqref{z15} would be complete if we could show that
\begin{multline}\label{z17}
\sum_{\{\alpha: \, \alpha(D)\cap{\mathcal T}_R(\tilde \gamma) \ne \emptyset\} }
w_0(\tilde x,\alpha(\tilde y)) \int_{{\mathbb R}^{n-1}}
\beta^2(|\xi|/\la) \, b(\la^{-7/8}(|\xi|-\la\kappa^{\theta_0}_\ell))
\\
\times \cos(r|\xi|) e^{id_{\tilde g}(\tilde x,\alpha(\tilde y)\xi_1}
e^{i(t-s)\la^{-1}|\xi|^2}
\, dr 
\\
 =O(\la^{\frac{n-1}2}|t-s|^{-\frac{n-1}2}), 
\quad
 \text{if } \, \, 1\le |t-s|\le c_0\log\la.
 \end{multline}
 
 To obtain \eqref{z17} we shall use the fact that
 for $|t-s|\ge1$ we have: 
\begin{multline}\label{z18}
\int_{{\mathbb R}^{n-1}} \beta^2(|\xi|/\la)
b(\la^{-7/8}(|\xi|-\la \kappa)) 
e^{id_{\tilde g}(\tilde x, \alpha(\tilde y))\xi_1} 
e^{i(t-s)\la^{-1}|\xi|^2} \,d\xi
=O(\la^{-N}) \, \forall N 
\\
\text{if } \, \, 
d_{\tilde g}(\tilde x, \alpha(\tilde y))\notin
I_{t,s,\kappa}=
\bigl[ 2|t-s|(\kappa-C\la^{-1/8}), \,
2|t-s|(\kappa+C\la^{-1/8})\big], \, \,
\text{if }\, \kappa=\kappa^{\theta_0}_\ell
\end{multline}
with $C$ large enough,
and
\begin{multline}\label{z19}
\int_{{\mathbb R}^{n-1}} \beta^2(|\xi|/\la)
b(\la^{-7/8}(|\xi|-\la \kappa)) 
e^{id_{\tilde g}(\tilde x,\alpha(\tilde y))\xi_1} 
e^{i(t-s)\la^{-1}|\xi|^2} \, d\xi
=O(\la^{\frac{n-1}2}|t-s|^{-\frac{n-1}2}),
\\
\text{if } \, \, d_{\tilde g}(\tilde x,\alpha(\tilde y))
\in I_{t,s,\kappa}, \, \, \, \kappa=\kappa^{\theta_0}_\ell.
\end{multline}
The first estimate  just follows from a simple integration by parts argument.  For if $\phi=d_{\tilde g}(\tilde x,\tilde y)\xi_1+(t-s)\la^{-1}|\xi|^2$, then, if
$C$ in the definition of $I_{t,s,\kappa}$ is fixed large enough, $|\nabla_\xi \phi|\ge \la^{-1/8}$ if $d_{\tilde g}(\tilde x,\tilde y) \notin
I_{\kappa,t,s}$, and, also, derivatives of the amplitude of the integral are $O(\la^{-7/8})$.  Thus, in this case, every integration by parts
gains a power of $\la^{-3/4}$, resulting in \eqref{z18}.  The other estimate, \eqref{z19} just follows from stationary phase.

If we note that the interval $I_{t,s,\kappa}$ has length $O(|t-s| \, \la^{-1/8})$ which is much smaller than 1, if as above we
assume that $|t-s|\le c_0\log\la$, we conclude that there can only be $O(1)$ terms in \eqref{z17} which are not
$O(\la^{-N})$, which leads to \eqref{z17} since, by \eqref{k17} $w_0$, is bounded.

This completes the proof of \eqref{z2}.

To prove the much stronger bounds \eqref{z3} which require that $M^{n-1}$ have {\em negative} sectional curvatures, we
note that the contribution of all of the terms in the Hadamard parametrix other than the main one, corresponding to $\nu=0$,
involved a $\lambda$-power improvement of what
was needed for \eqref{z2} and thus lead to bounds of the form \eqref{z3} since we are assuming that
$|t-s|=O(\log\la)$.  Thus, to prove \eqref{z3}, it is enough to show that under these curvature assumptions we have the analog
of \eqref{z17}  with $O(\la^{\frac{n-1}2}|t-s|^{-\frac{n-1}2})$ replaced by
$O(\la^{\frac{n-1}2}|t-s|^{-N})$ for every $N$.  To do this, we also use the simple fact, which follows from an integration
by parts argument, that we have the $O(\la^{-N})$ bounds in \eqref{z18} if $d_{\tilde g}(\tilde x,\alpha(\tilde y))
\notin [C_1^{-1}|t-s|, C_1|t-s|]$ with $C_1$ fixed sufficiently large.  In view of \eqref{k18} each of the $O(1)$ nontrivial
terms in the sum in the left side of \eqref{z17} must be $O(\la^{\frac{n-1}2}|t-s|^{-N})$ for every $N$, as desired.
This completes the proof of Proposition~\ref{mick}.
\end{proof}
  
  \noindent{\bf 4.3. Estimates for kernels involving local auxiliary operators.}
  
  Let us prove the kernel estimates we used in \S 3.

\medskip

\noindent{\bf Proof of Lemma~\ref{ker}.}
Let us now prove Lemma~\ref{ker}, which allowed us
to use parabolic scaling and results from \cite{LeeBilinear} to obtain the bilinear estimates
\eqref{q28}.  This lemma follows from a  straightforward variation
on the stationary phase arguments used to prove
\cite[Lemma 5.1.2]{SFIO2}.  Moreover, Lemma~\ref{ker}
is essentially Lemma 3.2 in \cite{blair2015refined}
or Lemma 4.3 in \cite{SBLog}, and in fact the latter
result almost immediately gives our results given
how we have constructed the local operators in \eqref{22.5}.

We first note that the kernel of our local operators
are given by
\begin{multline}\label{w1}
\tilde \sigma_\la(x,t;y,s)=
\bigl(B\circ\sigma_\la\bigr)(x,t;y,s)
\\
=(2\pi)^{-2}
\iint e^{i(t-s)\tau}e^{ir\la^{1/2}\tau^{1/2}}
\tilde \beta(\tau/\la) \, \Hat \sigma(r)
\, \bigl(B\circ e^{-irP}\bigr)(x,y) \, dr d\tau.
\end{multline}  
Since $(r,x,y)\to (B\circ e^{-irP})(x,y)$ is smooth when
$d_g(x,y)\ne |r|$ and $\Hat \sigma$ is as in
\eqref{22.6}, by a simple integration by parts
argument,
\begin{equation}\label{w2}
\tilde \sigma_\la(x,t;y,s)=O(\la^{-N}) \, \forall \, N
\quad \text{if } \, \, |d_g(x,y)-\delta|\ge \tfrac32
\delta_0\delta.
\end{equation}
This leads to the first part of \eqref{q7} if $\theta$
is small since the kernels of our microlocal
cutoffs, $A_\nu^{c_0\theta}$, satisfy
\begin{equation}
\label{w2'}A^{c_0\theta}_\nu(x,y)=O(\la^{-N}) \, \forall \, N
\quad \text{if } \, \, d_g(x,y)\ge C_1\theta,
\end{equation}
for a uniform constant $C_1$ since we are assuming
that $\la^{-1/2}\ll \la^{-1/8}\le \theta$.  Also, since the symbols $A_\nu^{c_0\theta}(x,\xi)=0$ if $\xi$ is not
in a small conic neighborhood of $(0,\dots,0,1)\in 
{\mathbb R}^{n-1}$, it follows that
$(r,x,y) \to \bigl(B\circ e^{-irP}  \circ A_\nu^{c_0\theta}  \bigr)(x,y)$ is smooth when $x_{n-1}-y_{n-1}<0$ and $\Hat \sigma \ne 0$, which yields the other half of \eqref{q7} via another
simple integration by parts argument.

Next, we recall that, by \eqref{m10} $A_\nu^{c_0\theta}
=A^{c_0\theta}_j(x,D) \circ A_\ell^{c_0\theta}(P)$, where 
$A_j$ localizes to a $c_0\theta$ neighborhood of a
geodesic $\overline \gamma_j\in \Omega$ about which we have chosen Fermi normal coordinates and
$A_\ell^{c_0\theta}(P)$ is the ``height operator'' given
by \eqref{m7}.  The other operator $A_{\nu'}^{c_0\theta}$ localizes at scale $c_0\theta$ to a geodesic $\overline \gamma_{j'}$ and 
height operator $A_{\ell'}^{c_0\theta}$ which are $\theta$-close to the above.

Next, let us use the fact that, by Lemma 3.2 in \cite{BlairSoggeRefined} or
Lemma 4.3 in 
\cite{SBLog},\footnote{In \cite{SBLog} different notation was used to denote the 
pseudodifferential cutoff $B$ due to the semiclassical notation there.}
 if we just consider the localizations
coming from the ones arising from geodesics, we have,
for $\omega\approx \la$, that the following kernels on
$M^{n-1}$ satisfy
\begin{equation}\label{w3}
\bigl(
\tilde \sigma(\omega-P) \circ   A^{c_0\theta}_\iota \bigr)(x,y)
=\omega^{\frac{n-2}2} e^{i\omega d_g(x,y)}
a_{\iota,\theta}(\omega;x,y) +O(\la^{-N}),
\, \, \iota=j,j',
\end{equation}
where the amplitude satisfies
$a_{\iota,\theta}=0$ if \eqref{q6} is valid, and, 
additionally, since we are working in Fermi normal
coordiates about $\overline \gamma_j$ 
\begin{equation}\label{w4}
\bigl| \partial^i_\omega 
\partial^k_{x_{n-1}}\partial^{k'}_{y_{n-1}}
D^\beta_{x,y}a_{\iota,\theta}(\omega;x,y)|
\le C_{i,k,k',\beta} \, \omega^{-i}\theta^{-|\beta|}, \quad \iota=j,j'.
\end{equation}
If $i=0$, this  just follows \cite[Lemma 4.3]{SBLog}
and our choice of coordinates.  In order to get the $c_0\theta$-scale concentration
as in \eqref{q6} that we used in the last section, we apply \cite[Lemma 4.3]{SBLog} with
$\theta$  there replaced by $c_0\theta$.
The fact that we also
have a $\omega^{-1}$ improvement for each $\omega$-derivative just comes from the fact that if
we use parametrices for $e^{-irP}$ to represent $e^{-i\omega d_g(x,y)}$ times the right side of 
\eqref{w3} as an oscillatory integral in the standard way, such as in \cite{SBLog}, each $\omega$-derivative
brings down a factor of the phase function (normalized to vanish at the stationary points) and so
results in a $\omega^{-1}$ improvement, just as in
standard stationary phase with parameters results
(see e.g., \cite[Corollary 1.1.8]{SFIO2}).

To obtain \eqref{q6} for our kernels, we first note that, by \eqref{w1}
and \eqref{m10},
\begin{multline}\label{w5}
\bigl(
\tilde \sigma_\la \circ A_\nu^{c_0\theta}\bigr)(x,t;y,s)
\\
=
\Bigl(\bigl[\, (2\pi)^{-1}
\int e^{i\tau(t-s)}
 \,
\tilde \sigma(\la^{1/2}\tau^{1/2}-P) \circ A_j^{c_0\theta}\circ
\tilde \beta(\tau/\la)\, d\tau
\, \bigr] \circ A_\ell^{c_0\theta}\Bigr).
\end{multline}
If we consider the kernel of the operator inside
the square brackets, by \eqref{w3}, we can write it
as
\begin{align}\label{w6}
%K_{\la,\nu}^{c_0\theta}&(x,t;y,s)
%\\
%&=
(2\pi)^{-1}
&\int^\infty_{-\infty} e^{i\tau(t-s)} (\la\tau)^{\frac{n-2}4}
e^{i\la^{1/2}\tau^{1/2}d_g(x,y)}
a_{j,\theta}(\la^{1/2}\tau^{1/2};x,y)\, \tilde \beta(\tau/\la) \, d\tau +O(\la^{-N})
%\notag
\\
&=(2\pi)^{-1}\la^{n/2} 
\int^\infty_{-\infty} e^{i\la[\tau(t-s)+ \tau^{1/2}d_g(x,y)]}
a_{j,\theta}(\la\tau^{1/2};x,y)\, 
\tau^{\frac{n-2}4}\tilde \beta(\tau) d\tau +O(\la^{-N}),
\notag
\\
&=\pi^{-1} \la^{n/2}\int_0^\infty e^{i\la[\tau^2(t-s)+\tau d_g(x,y)]}
a_{j,\theta}(\la\tau;x,y) \, \tau^{n/2} \, \tilde\beta(\tau^2) \, d\tau +O(\la^{-N}), \notag
\end{align}
where $a_{j,\theta}$  is
as in \eqref{w3} and so vanishes when $x$ or $y$ is outside a $O(c_0\theta)$-tube
about $\overline\gamma_j$.  Since the kernel of
$A_\ell^{c_0\theta}$ also satisfies \eqref{w2'}, by combining \eqref{w5} and \eqref{w6},
we obtain \eqref{q6} for $\nu=(c_0\theta j,c_0\theta \ell)$.  The same argument
gives us this for $\nu'=(c_0\theta j',c_0\theta\ell')$.
% since $|\nu-\nu'|\approx \theta$.

Next, let us use \eqref{w4}--\eqref{w6} to prove 
the remaining parts of Lemma~\ref{ker} saying that
the kernels are also $O(\la^{-N})$ in the
regions described by \eqref{q5} and \eqref{qtube} and
that outside of these and the ones in \eqref{q6} and
\eqref{q7} (where we already know this), they are as in \eqref{q3} and \eqref{q4}.

In order to do this we argue as in H\"ormander~ \cite{HSpec} or more specifically as in 
\cite[\S4.3]{SFIO2} to see that the we can write
the kernel of the height operators, $m=\ell, \ell'$, as
\begin{align}\label{w7}
&A_m^{c_0\theta}(x,y)
\\
&=
\int_{{\mathbb R}^{n-1}} 
e^{i\varphi(x,y;\xi)}
b\bigl((c_0\theta\la)^{-1}(p(x,\xi)-\la\kappa_m^{c_0\theta})
\bigr)
\Upsilon(p(x,\xi)/\la) \, 
q(x,y,\xi) \, d\xi +O(\la^{-N})
\notag
\\
&= \la^{n-1}
\int_{{\mathbb R}^{n-1}}
e^{i\la\varphi(x,y;\xi)}
b\bigl((c_0\theta)^{-1}(p(x,\xi)-\kappa_m^{c_0\theta})
\bigr)
\Upsilon(p(x,\xi)) \, 
q(x,y,\la\xi) \, d\xi +O(\la^{-N}),
\notag
\end{align}
where $b\in C_0^\infty((-1,1))$ is as in
\eqref{m7}, $\Upsilon$ as in \eqref{ups}, and $p(x,\xi)$ is the principal symbol
of $P$, $q\in S^{0}_{1,0}$, $(2\pi)^{-(n-1)}-q\in S^{-1}_{1,0}$
and $\varphi$ is homogeneous of degree one in $\xi$
and satisfies
\begin{equation}\label{w8}
\varphi(x,y;\xi)=\langle x-y,\xi\rangle
+O(|x-y|^2|\xi|), \quad
\text{on supp } \, q.
\end{equation}
So, in particular,
\begin{equation}\label{w9}
\nabla_\xi \varphi=0
\, \iff \, x=y, \, \, \, \text{and } \, 
\nabla_x\varphi=\xi \, \, \text{as well as } \,
\frac{\partial^2\varphi}{\partial x\partial \xi}=I_{n-1} 
\, \, \text{if } \, x=y.
\end{equation}
Indeed, to see this, one recalls that the Lax parametrix allows to write
for small $|t|$
$$\bigl(e^{itP}\bigr)(x,y)
=
\int e^{i\varphi(x,y;\xi)+itp(x,\xi)} q(x,y,t;\xi) \, 
d\xi,$$
for $q\in S^0_{1,0}$ solving a transport equation and
so $(2\pi)^{-(n-1)}-q(0,x,y;\xi)\in S^{-1}_{1,0}$.  Using this, and
the fact that the Fourier transform of
$\tau\to b\bigl((c_0\theta\la)^{-1}(\tau-\la\kappa_m^{c_0\theta})
\bigr)\Upsilon(\tau/\la)$ is $O(\la^{-N})$ 
and rapidly decreasing outside
of a fixed interval about the origin, allows one to
argue as in \cite[\S 4.3]{SFIO2} or the previous
two subsections here to obtain \eqref{w7}.

In the regions where we do not already know that
the kernel $K^{c_0\theta}_{\la,\mu}$ in \eqref{q3}
is $O(\la^{-N})$, by \eqref{w5}, \eqref{w6}
and \eqref{w7}, we can write
\begin{multline}\label{w10}
K^{c_0\theta}_{\la,\mu}(x,y)=
c\la^{\frac{n}2}\int_{0}^\infty\Bigl[\,
\la^{n-1}
\int_{{\mathbb R}^{2(n-1)}} e^{i\la[\tau d_g(x,z)+
\varphi(z,y;\xi)]} a_{\iota,\theta}(\la\tau;x,z) \, 
\tau^{\frac{n}2}\tilde \beta(\tau^2)
\\
\qquad \qquad\qquad \times
b((c_0\theta)^{-1}(p(z,\xi)-\kappa_m^{c_0\theta}))\Upsilon(p(z,\xi)/\la) 
q(z,y;\la\xi) dz d\xi\, \Bigr] \, 
e^{i\la\tau^2(t-s)} d\tau,
\\
\mu=\nu,\nu', \, \iota=j,j', \, m=\ell,\ell'.
\end{multline}
If we consider the oscillatory integral over ${\mathbb R}^{2(n-1)}$ in the
square brackets here,
the phase function is
$$\phi(z,\xi)=\phi(x,y,\tau;z,\xi)=
\tau d_g(x,z)+\varphi(z,y;\xi).$$  It has
a unique stationary point when
$$y=z \quad \text{and } \, \tau\nabla_z d_g(x,z)=-
\nabla_z \varphi(z,y;\xi)=-\xi,$$
with the last inequality coming from the second part
of \eqref{w9}.
This stationary point is non-degenerate by the
last part of \eqref{w9}, and
$\phi=\tau d_g(x,y)$ there.
Also, since 
%$\nabla_z d_g(x,z)=p(z,\nabla_z d_g(x,z))$,
$p(z,\nabla_z d_g(x,z))=1$ and $p(z,\xi)=p(z,-\xi)$, 
we conclude
that
\begin{equation}\label{w11}\tau = p(z,\xi) \, \,
\text{and } \, \, \varphi=0 \quad 
\text{when } \, \, \nabla_{z,\xi}\phi=0.
\end{equation}

Since $\theta\ge \la^{-1/8}\gg \la^{-1/2}$ we may use 
\eqref{w9} and \eqref{w11} 
along with stationary phase to evaluate $\la^{n-1}$ times the oscillatory
integral inside the square brackets in \eqref{w10}.
It must be of the form
\begin{equation}\label{w12}
    \begin{aligned}
        &e^{i\tau d_g(x,y)}\tilde a_{\iota,\theta}(\la\tau;
x,y)\, \tau^{\frac{n}2} \tilde \beta(\tau^2)
\, 
\tilde b\bigl((c_0\theta)^{-1}(\tau-\kappa_m^{c_0\theta})\bigr) 
\tilde q(x,y;\la\tau )
\\
&=
e^{i\tau d_g(x,y)}a_{\iota,\theta}(\la\tau;
x,y) \, \tau^{\frac{n}2} \tilde \beta(\tau^2)
\, 
b\bigl((c_0\theta)^{-1}(\tau-\kappa_m^{c_0\theta})\bigr) \Upsilon(\tau/\la) 
%\tilde \beta(\tau)
q(y,y;-\la\tau \nabla_y d_g(x,y)) \\
&\quad+O(\la^{-3/4}).
    \end{aligned}
\end{equation}
Here $\tilde a_{\iota,\theta}$ satisfies the bounds
in \eqref{w4}, like $b$, the smooth function $\tilde b$ vanishes outside of $[-1,1]$, and, finally, 
%$\tilde q(x,y;\tau)\in S^0_{1,0}$.  
$\tilde q\in S^0_{1,0}$.  

If we combine \eqref{w10} and \eqref{w11}, we conclude
that
\begin{multline}\label{w13}
K^{c_0\theta}_{\la,\mu}(x,y)=
c\la^{\frac{n}2}\int_{0}^\infty
e^{i\la[\tau d_g(x,y)+\tau^2(t-s)]}
\tilde a_{\iota,\theta}(\la\tau;
x,y) \tau^{\frac{n}2} \tilde \beta(\tau^2)
\\
\times \tilde b\bigl((c_0\theta)^{-1}(\tau-\kappa_m^{c_0\theta})\bigr)
\tilde q(x,y;\la\tau)\,  d\tau,
\, \, 
\mu=\nu,\nu', \, \iota=j,j', \,
m=\ell, \ell'.
\end{multline}

Now we shall prove \eqref{q6}-\eqref{qtube}, by a simple integration by parts argument, we obtain
\eqref{q5} from \eqref{w13}, and, by using the properties of the amplitude function $\tilde a_{\iota,\theta}(\la\tau;
x,y)$, we have 
%as \eqref{qtube} 
the assertion in \eqref{qtube} that the amplitudes are $O(\la^{-N})$  when $|(x_1,\dots,x_{n-2}|+|(y_1,\dots,y_{n-2})|$ is larger than 
a fixed multiple of $\theta$ for both $\mu=\nu, \nu'$ 
since $|\nu-\nu'|=O(\theta)$.  For the last part of \eqref{qtube}, saying that the amplitudes
are also trivial when $|(x_{n-1}-y_{n-1})+2\kappa^{c_0\theta}_\ell(t-s)|$ is larger than a
fixed multiple of $\theta$,
we use
the fact that $d_g(x,y)=x_{n-1}-y_{n-1}+O(\theta)$ in our Fermi normal coordinates if
\eqref{q7} is valid and $x,y$ are in a $O(\theta)$-tube about $\overline{\gamma_j}$ as in \eqref{q6}.
By \eqref{qtube},
along with the earlier steps, we conclude that these kernels
are $O(\la^{-N})$ in the regions described by
\eqref{q6}--\eqref{q7}.

Also, since the phase function in \eqref{w13} has a unique stationary point when $\tau=-d_g(x,y)/2(t-s)$
which is non-degenerate, and since the phase equals 
$-(d_g(x,y))^2/4(t-s)$ there, we conclude that the kernels in \eqref{w12} must be of the form
\eqref{q3}.  It is also straightforward that the amplitudes must satisfy the estimates in \eqref{q4}
in the special cases where both $m_1$ and $m_2$ are
zero due to \eqref{w4}.

To prove the estimates \eqref{q4} involving $(m_1,m_2)\ne(0,0)$, we first note that
$$(d_g(x,y))^2/4(t-s)+\tau d_g(x,y)+\tau^2(t-s)
=(t-s) \cdot \bigl(\tau+d_g(x,y)/2(t-s)\bigr)^2,$$
and, also, by \eqref{qtube} and \eqref{q7} $t-s\approx -\delta$ when the kernel is non-trivial.
Therefore, by \eqref{w13},
the amplitude in \eqref{q3} is of the form
\begin{multline}\label{ww1}
a_{\la,\mu}(x,t;y,s)=
\\
\la^{\frac12}
\int_{-\infty}^\infty e^{-i\la \tau^2}
h_{\la,\iota,\theta}\bigl(x,t;y,s; (\tau- \tfrac{(s-t)^{1/2}d_g(x,y)}{2(t-s)});
%(s-t)^{1/2} d_g(x.y)/2(t-s));
 \, \tfrac1{(s-t)^{1/2}c_0\theta}
%\bigl((s-t)^{1/2} c_0\theta)^{-1}
(\tau-(s-t)^{1/2}[\kappa_m^{c_0\theta}+
%d_g(x,y)/2(t-s)
\tfrac{d_g(x,y)}{2(t-s)}
])\bigr) \, d\tau,
\\
\mu=\nu=(j,k), \nu'=(j',\ell'), \, \, \iota=j, j', \, \, m=\ell, \ell',
\end{multline}
with
\begin{multline}\label{ww2}
h_{\la,\iota, \theta}(x,t;y,s; u;r)
\\
=\tilde a_{\iota,\theta}(\la u/(s-t)^{1/2};x,y) \, (u_+/(s-t)^{1/2})^{n/2} \, \tilde \beta\bigl(u^2/(s-t)\bigr) \,
\, \tilde q\bigl(x,y;\la u/(s-t)^{1/2}\bigr) \, \tilde b(r) .
\end{multline}
Here $u_+=u$ if $u\ge0$ and $0$ otherwise.
What is important for us and follows from the fact
that $\tilde a_{\iota,\theta}$ satisfies\eqref{w4}, the support properties of $\tilde b$ and
$\tilde \beta$, as well as the fact that $s-t$ is bounded away from zero and $\tilde q\in S^0_{1,0}$, 
is that we have
\begin{equation}\label{ww3}
h_{\la,\iota, \theta}(x,t;y,s; u;r) =0 \quad \text{if } \, \, 
|u| +|r| \ge C \, \, \, \text{or } \, \, |u|\le C^{-1},
\end{equation}
for some fixed $C=C_\delta$, and, moreover
\begin{equation}\label{ww4}
D^{\beta_1}_{t,s,u,r,x_{n-1},y_{n-1}} (\theta D_{x',y'})^{\beta_2} h_{\la,\iota, \theta}(x,t;y,s; u;r)
=O_{\beta_1,\beta_2}(1).
\end{equation}

Let us now use this to prove \eqref{q4} for $(m_1,m_2)\ne 0$.  
 We shall 
first consider the special case where $\mu=\nu$
and  on $x,y\in \overline \gamma_j$, which is a 
portion of the $(n-1)$-axis in the Fermi normal
coordinate system in which we are working.  We
then have $d_g(x,y)=x_{n-1}-y_{n-1}$ if the kernel
is nontrivial by \eqref{q7}.  Note that
\eqref{q5} tells us that the amplitude
$a_{\la,\nu}$ in \eqref{q3} is also very highly 
concentrated on the Schr\"odinger curve
where we also have $x_{n-1}-y_{n-1}=-2\kappa^{c_0\theta}_\ell(t-s)$.  With this in mind, let us prove 
\eqref{q4} when $\beta=0$, $m_2=0$ and $m_1=1$ and we
are on this Schr\"odinger curve.  We then have
$-d_g(x,y)/2(t-s)\equiv \kappa_\ell^{c_0\theta}$.  In this case we take $\kappa_m^{c_0\theta}=\kappa^{c_0\theta}_\ell$, $\iota=j$ and
$\mu=\nu$ in \eqref{ww1} and see that we would have \eqref{q4} for this special case if
\begin{equation}\label{ww5}
\int e^{-i\la \tau^2}
\bigl(\partial_u h_{\la,\iota,\theta}\bigr)(x,t;y,s; \tau+
(s-t)^{1/2}\kappa^{c_0\theta}_\ell; ((s-t)^{1/2}c_0\theta)^{-1}\tau) \, d\tau
=O(\la^{-1/2}),
\end{equation}
as well as
\begin{multline}\label{ww6}
\theta^{-1}\int \tau e^{-i\la \tau^2}
 \bigl(\partial_r h_{\la,\iota,\theta}\bigr)(x,t;y,s; \tau+
 (s-t)^{1/2}\kappa^{c_0\theta}_\ell; ((s-t)^{1/2}c_0\theta)^{-1}\tau) \, d\tau
 \\
=O(\la^{-1/2}).
\end{multline}
The first estimate, \eqref{ww5} just follows from stationary phase and \eqref{ww4}.  We obtain the second estimate
by realizing that, after integrating by parts, we can rewrite the left side as
\begin{multline}\label{ww7}
(2i\la\theta)^{-1}\int  e^{-i\la \tau^2} \, 
\frac\partial{\partial \tau}\Bigl[
 \bigl(\partial_r h_{\la,\iota,\theta}\bigr)(x,t;y,s; \tau+
 (s-t)^{1/2}\kappa^{c_0\theta}_\ell; ((s-t)^{1/2}c_0\theta)^{-1}\tau)\Bigr] \, d\tau
 \\
 =O((\la \theta^2)^{-1}\la^{-1/2})=O(\la^{-1/2}),
 \end{multline}
with the bounds in the right holding by \eqref{ww4} along with stationary phase and the fact that $\theta \gg \la^{-1/2}$.  In view of
\eqref{ww4}, it is clear that 
by induction this argument will give the rest of \eqref{q4} in this special
case where both $(x,t)$ and $(y,s)$ lie on this special Schr\"odinger curve.

If $x,y$ are in a $O(\theta)$-tube about
$\overline \gamma_j$ with $d_g(x,y)\approx \delta$
and we let $\gamma(\, \cdot \, )$ be the unit speed geodesic
in $M^{n-1}$ with $\gamma(0)=0$ and $\gamma(r)=x$,
$r=d_g(x,y)$, then the argument also yields
$$\bigl(2\kappa^{c_0\theta}_\ell
\partial_r-\partial_t\bigr)
a_{\la,\nu}(\gamma(r),t;y,s)=O(1) \quad
\text{if } \, \, r=d_g(x,y)
\, \, \text{and } \, r=-2\kappa^{c_0\theta}_\ell (t-s),
$$
due to the fact that the Schr\"odinger curve connecting
$(x,t)$ and $(y,s)$ is as in \eqref{q1}.
Since we are working in Fermi normal coordinates
about $\overline \gamma_j$ this equals
$$(2\kappa^{c_0\theta}\partial_{x_{n-1}}-\partial_t)
a_{\la,\nu}(x,t;y,s)+O(\theta|\nabla_x a_{\la,\nu}|),$$
and the error term here is $O(1)$ by our known estimate in
\eqref{q4} where $|\beta|=1$ and $m_1=m_2=0$.  Thus,
if there is a $\kappa^{c_0\theta}_\ell$-speed 
Schr\"odinger curve connecting $(x,t)$ and $(y,s)$
and the kernel is not $O(\la^{-N})$
we have \eqref{q4} with $m_1=1$, $m_2=0$ and $\mu=\nu$.  By an
induction argument, it must be valid for all
$(m_1,m_2,\beta)$ in this case.  

%Since, by \eqref{qtube}, 
If the kernel is non-trivial at 
%these points, 
at $(x,t;y,s)$,
then by \eqref{qtube} and \eqref{q7}
there is a Schr\"odinger curve as in \eqref{q1}
with associated speed $\kappa
=\kappa^{c_0\theta}_\ell +O(\theta)$, passing through $(x,t)$ and $(y,s)$.
So, by the above argument, we would 
we have
\eqref{q4} for $\mu=\nu$, $m_1=1$ and $m_2=0$
 if we had the following variants of \eqref{ww5} and \eqref{ww6}:
\begin{multline}\label{ww5'}
\int e^{-i\la \tau^2}
\bigl(\partial_u h_{\la,\iota,\theta}\bigr)(x,t;y,s; \tau+
(s-t)^{1/2}\kappa; ((s-t)^{1/2}c_0\theta)^{-1}(\tau-(s-t)^{1/2}(\kappa^{c_0\theta}_\ell
-\kappa)) \, d\tau
\\
=O(\la^{-1/2}),
\end{multline}
as well as
\begin{multline}\label{ww6'}
\theta^{-1}\int 
(\tau-(\kappa^{c_0\theta}_\ell
-\kappa)) \cdot
e^{-i\la \tau^2}
\\
\times  \bigl(\partial_r h_{\la,\iota,\theta}\bigr)(x,t;y,s; \tau+
 (s-t)^{1/2}\kappa; ((s-t)^{1/2}c_0\theta)^{-1}(\tau- (s-t)^{1/2}(\kappa^{c_0\theta}_\ell
-\kappa)) \, d\tau
\\
=O(\la^{-1/2}).
\end{multline}
Just as with \eqref{ww5}, \eqref{ww5'} follows immediately from stationary phase arguments and \eqref{ww3}--\eqref{ww4}.
We also get \eqref{ww6'} since, as we mentioned before, we must have $\kappa^{c_0\theta}_\ell -\kappa=O(\theta)$,
and so the left side of \eqref{ww6'} splits into two terms, one of which is of the form \eqref{ww5}, while the other is
of the form \eqref{ww6}.  
Thus,  
\eqref{ww5'} and \eqref{ww6'} imply \eqref{q4} for $\mu=\nu$
when $m_1=1$ and $m_2=0$.  Also, just as before, one obtains the remaining cases of \eqref{q4} by an
induction argument.

Finally, since $|\kappa^{c_0\theta}_\ell-\kappa^{c_0\theta}_{\ell'}|, \, \theta|j-j'|=O(\theta)$, it is also
clear that \eqref{q4}
also must hold when $\nu=(c_0\theta j, c_0\theta \ell)$ is replaced by
$\nu'=(c_0\theta j', c_0\theta\ell')$, which completes the proof of Lemma~\ref{ker}.

\bibliography{refs}
\bibliographystyle{abbrv}

%\begin{thebibliography}{MA}
%\bibitem{burq} N. Burq: {\em D\'ecroissance de l\'energie locale
%de l'\'equation des ondes pour le probl\`me ext\'erieur et absence
%de r\'sonance au voisinage du r\'eel}, Acta Math. {\bf 180}
%(1998), 1--29.
% \bibitem{gilbarg} D. Gilbarg and N. Trudinger:
%{\em Elliptic partial differential equations of second order},
%Springer, Second Ed., Third Printing, 1998.
%\bibitem{KS} S. Klainerman and T. Sideris: {\em On almost global existence for
%nonrelativistic wave equations in 3d} Comm. Pure Appl. Math. {\bf
%49}, (1996), 307--321.
%
%\end{thebibliography}

\end{document}